\documentclass[english,a4paper]{article}

\usepackage[utf8]{inputenc}

\usepackage{parskip}
\usepackage[colorlinks]{hyperref}
\hypersetup{
  allcolors = [rgb]{0,0,0.5}
}

\usepackage[nobysame,alphabetic,initials,msc-links]{amsrefs}

\DefineSimpleKey{bib}{how}

\renewcommand{\eprint}[1]{#1}
\BibSpec{misc}{%
  +{}{\PrintAuthors}  {author}
  +{,}{ \textit}      {title}
  +{,}{ }             {how}
  +{}{ \parenthesize} {date}
  +{,} { available at \eprint}        {eprint}
  +{,}{ available at \url}{url}
  +{,}{ }             {note}
  +{.}{}              {transition}
}

\usepackage{float}
\usepackage{mathrsfs,amsmath,mathtools}
\usepackage{amsfonts}
\usepackage{latexsym}
\usepackage{faktor}
\usepackage{xfrac}
\usepackage{amsthm}
\usepackage{amssymb}
\usepackage{enumerate}
\usepackage{multicol}
\usepackage{tikz-cd}
\usepackage{stmaryrd}

\usepackage{amssymb}
\usepackage{tikz-cd}

\usepackage{tikz}
\usetikzlibrary{shapes.geometric,calc}

\usepackage{enumerate}
\usepackage{multicol}

\numberwithin{equation}{section}

\theoremstyle{definition}

\newtheorem{defi}{Definition}[section]

\theoremstyle{plain}
\newtheorem{thm}[defi]{Theorem}
\newtheorem{lem}[defi]{Lemma}
\newtheorem{prop}[defi]{Proposition}

\theoremstyle{remark}

\newtheorem{rem}[defi]{Remark}

\newcommand{\bk}{\Bbbk}

\newcommand{\cx}{\mathcal{X}}

\newcommand{\cz}{\mathcal{Z}}
\newcommand{\rd}{\mathrm{d}}
\newcommand{\rI}{\mathrm{I}}
\newcommand{\rJ}{\mathrm{J}}

\newcommand{\Arr}{\blacktriangleright}
\newcommand{\IndexSet}{\mathrm{I}}
\newcommand{\EQ}{\mathrm{E}}
\newcommand{\id}{\mathrm{id}}
\newcommand{\Id}{\mathtt{e}}
\newcommand{\tr}{\mathrm{tr}}

\newcommand{\ptr}{\mathrm{ptr}}
\newcommand{\vSp}{\mathrm{H}}
\newcommand{\twist}{\mathtt{\tau}}
\newcommand{\un}{\mathbf{1}}
\newcommand{\Boxt}{{\,\Box\,}}

\newcommand{\Hom}{\operatorname{Hom}}
\newcommand{\End}{\operatorname{End}}

\newcommand{\Irr}{\operatorname{Irr}}
\newcommand{\Ind}{\operatorname{Ind}}
\newcommand{\Rep}{\operatorname{Rep}}
\newcommand{\Iso}{\operatorname{Iso}}

\newcommand{\Hsp}[1]{\mathrm{H}\mathopen{}\left[#1\right]\mathclose{}}

\newcommand{\auxprod}{\Pi_{a,b{;}c}}
\newcommand{\Auxprod}[3]{{\Pi_{#1{,}#2{;}#3}}}
\newcommand{\coev}{\mathrm{coev}}
\newcommand{\ev}{\mathrm{ev}}
\newcommand{\tcoev}{\overline{{\coev}}}
\newcommand{\tev}{\overline{{\,\ev}}}
\newcommand{\Tube}{\mathrm{Tube}(\cx)}
\newcommand{\Tubecx}{\mathrm{Tube}(\cx,\IndexSet)}

\newcommand{\kVec}{\mathrm{Vec}_\bk}

\newcommand{\Colon}{\colon}

\newcommand{\tenEl}[7]{
\def\radL{1.7cm}
\def\radS{0.5cm}
\begin{tikzpicture}[xscale=0.9,yscale=0.9,baseline={([yshift=-.5ex]current bounding box.center)}]
\draw[line width=1] (0,0) circle (\radL)
(-\radS,0) circle (\radS)
(\radS,0) circle (\radS);
\draw[line width=1] 
    (0,0) to (0,\radL);

\node[rotate=90] at (0,0.5*\radL) {\scriptsize $\Arr$};
\node[rotate=180] at (-\radS,-\radS) {\scriptsize $\Arr$};
\node[rotate=180] at (\radS,-\radS) {\scriptsize $\Arr$};
\node[rotate=180] at (0,-\radL) {\scriptsize $\Arr$};

\filldraw [black] (0,0) circle (1.5pt);
\filldraw [black] (0,\radL) circle (1.5pt);

\node at (\radS,0) {$#6$};
\node at (-\radS,0) {$#7$};

\node at (0,0.5*\radL) [anchor=south] {$=$};
\node at (0,0.55*\radL) [anchor=west] {$#3$};
\node at (-40:\radL) [anchor=north] {$#2$};
\node at (2*\radS,0) [anchor=west] {$#4$};
\node at (-2*\radS,0) [anchor=east] {$#5$};
\node at (0,-0.4*\radL) [anchor=north] {{$#1$}};
\end{tikzpicture}
}

\newcommand{\pielem}[5]{
\begin{tikzpicture}[xscale=0.8,baseline={([yshift=-.5ex]current bounding box.center)}]
    \draw[line width=1] 
    (-0.75*\len,0) to (0.75*\len,0)
    (-0.75*\len,0) to (-0.75*\len,0.5*\len)
    (0.75*\len,0) to (0.75*\len,0.5*\len)
    (-0.75*\len,0.5*\len) to (0.75*\len,0.5*\len);

\node[rotate=-90] at (-0.75*\len,0.25*\len) {\scriptsize $\Arr$};
\node[rotate=-90] at (0.75*\len,0.25*\len) {\scriptsize $\Arr$};
\node[rotate=180] at (0,0) {\scriptsize $\Arr$};
\node[rotate=180] at (-0.4*\len,0.5*\len) {\scriptsize $\Arr$};
\node[rotate=180] at (0.4*\len,0.5*\len) {\scriptsize $\Arr$};

\filldraw [black] (-0.75*\len,0) circle (1pt);
\filldraw [black] (0.75*\len,0) circle (1pt);
\filldraw [black] (0,0.5*\len) circle (1pt);
\filldraw [black] (-0.75*\len,0.5*\len) circle (1pt);
\filldraw [black] (0.75*\len,0.5*\len) circle (1pt);

\node at (-0.75*\len,0.25*\len) [anchor=east] {$#4$};
\node at (0.75*\len,0.25*\len) [anchor=west] {$#4$};
\node at (0,0) [anchor=north] {$#3$};
\node at (0.4*\len,0.5*\len) [anchor=south] {$#1$};
\node at (-0.4*\len,0.5*\len) [anchor=south] {$#2$};

\draw (0,0.25*\len) node {${#5}$};
\end{tikzpicture}
}

\newcommand{\pielemUnit}[4]{
\begin{tikzpicture}[xscale=0.8,baseline={([yshift=-.5ex]current bounding box.center)}]
    \draw[line width=1] 
    (-0.75*\len,0) to (0.75*\len,0)
    (-0.75*\len,0.5*\len) to (0.75*\len,0.5*\len);
    \draw[line width=1,dotted] 
    (-0.75*\len,0) to (-0.75*\len,0.5*\len)
    (0.75*\len,0) to (0.75*\len,0.5*\len);

\node[rotate=180] at (0,0) {\scriptsize $\Arr$};
\node[rotate=180] at (-0.4*\len,0.5*\len) {\scriptsize $\Arr$};
\node[rotate=180] at (0.4*\len,0.5*\len) {\scriptsize $\Arr$};

\filldraw [black] (-0.75*\len,0) circle (1pt);
\filldraw [black] (0.75*\len,0) circle (1pt);
\filldraw [black] (0,0.5*\len) circle (1pt);
\filldraw [black] (-0.75*\len,0.5*\len) circle (1pt);
\filldraw [black] (0.75*\len,0.5*\len) circle (1pt);

\node at (0,0) [anchor=north] {$#3$};
\node at (0.4*\len,0.5*\len) [anchor=south] {$#1$};
\node at (-0.4*\len,0.5*\len) [anchor=south] {$#2$};

\draw (0,0.25*\len) node {${#4}$};
\end{tikzpicture}
}

\newcommand{\centrelem}[3]{
\def\len{2cm}
\begin{tikzpicture}[baseline={([yshift=-.5ex]current bounding box.center)}]
\draw[line width=1]
    (-0.25*\len,0) .. controls (-0.25*\len,0.4*\len) and (0.25*\len,0.4*\len)..  node[midway,rotate=180] {\scriptsize $\Arr$} (0.25*\len,0) node[midway,above] {\footnotesize $#2$};
\draw[line width=1] 
    (-0.25*\len,0) to (0.25*\len,0);

\node[rotate=180] at (0,0) {\scriptsize $\Arr$};

\filldraw [black] (-0.25*\len,0) circle (1pt);
\filldraw [black] (0.25*\len,0) circle (1pt);

\node at (0,0) [anchor=north] { $#1$};
\draw (0,0.15*\len) node {${#3}$};
\end{tikzpicture}
}

\newcommand{\unelem}[2]{
\def\len{2cm}
\begin{tikzpicture}[baseline={([yshift=-.5ex]current bounding box.center)}]
\draw[line width=1,dotted]
    (-0.25*\len,0) .. controls (-0.25*\len,0.4*\len) and (0.25*\len,0.4*\len).. (0.25*\len,0);
\draw[line width=1] 
    (-0.25*\len,0) to (0.25*\len,0);

\node[rotate=180] at (0,0) {\scriptsize $\Arr$};

\filldraw [black] (-0.25*\len,0) circle (1pt);
\filldraw [black] (0.25*\len,0) circle (1pt);

\node at (0,0) [anchor=north]{$#1$};
\draw (0,0.15*\len) node {${#2}$};
\end{tikzpicture}
}

\newcommand{\tubelem}[4]{
\begin{tikzpicture}[xscale=0.8,baseline={([yshift=-.5ex]current bounding box.center)}]
\draw[line width=1] 
    (-0.5*\len,0) to (0.5*\len,0)
    (-0.5*\len,0) to (-0.5*\len,0.5*\len)
    (0.5*\len,0) to (0.5*\len,0.5*\len)
    (-0.5*\len,0.5*\len) to (0.5*\len,0.5*\len);

\node[rotate=-90] at (-0.5*\len,0.25*\len) {\scriptsize $\Arr$};
\node[rotate=-90] at (0.5*\len,0.25*\len) {\scriptsize $\Arr$};
\node[rotate=180] at (0,0) {\scriptsize $\Arr$};
\node[rotate=180] at (0,0.5*\len) {\scriptsize $\Arr$};

\filldraw [black] (-0.5*\len,0) circle (1pt);
\filldraw [black] (0.5*\len,0) circle (1pt);
\filldraw [black] (-0.5*\len,0.5*\len) circle (1pt);
\filldraw [black] (0.5*\len,0.5*\len) circle (1pt);

\node at (-0.5*\len,0.25*\len) [anchor=east] {$#3$};
\node at (0.5*\len,0.25*\len) [anchor=west] {$#3$};
\node at (0,0) [anchor=north] {$#2$};
\node at (0,0.5*\len) [anchor=south] {$#1$};

\draw (0,0.25*\len) node  {${#4}$};
\end{tikzpicture}
}

\usepackage{fullpage}

\title{The braided monoidal structure of tube algebra representations}

\usepackage{authblk}
\author{David Jaklitsch\footnote{dajak@math.uio.no} }
\author{Makoto Yamashita\footnote{makotoy@math.uio.no}}
\affil{Department of Mathematics, University of Oslo}

\date{November 11, 2025}

\begin{document}

\maketitle

\begin{abstract}
We consider the tube algebra of a spherical semisimple multitensor category $\cx$, and construct a braided monoidal structure with twist for its representations.
We further show that this category is braided tensor equivalent with the Drinfeld center of the ind-category of $\cx$, extending the well-known linear equivalence.
\end{abstract}

\tableofcontents

\newpage
\section{Introduction}

\emph{Tube algebras} first appeared in the work of Ocneanu~\cite{MR1317353} on the structure of subfactors and associated topological quantum field theory.
This construction was analogous to the quantum double of Hopf algebras, and a precise connection between the Drinfeld center of the associated tensor category was clarified through the work of Izumi~\cite{MR1782145} and Müger~\cite{MR1966525}.
To be precise, these works show that, given a unitary fusion category $\cx$, the linear category of modules over Ocneanu's tube algebra $\Tube$ is equivalent to the underlying linear category of the Drinfeld center category $\cz(\cx)$.
This correspondence has been generalized to the setting without finiteness assumption on the number of irreducible classes, and expanded to annular algebras when $\cx$ is a rigid C$^*$-tensor category as studied in~\cites{MR3447719,NY15tubealgebra}, which gave a linear equivalence of between $\Tube$-modules and the Drinfeld center.

While it is possible to transport the braided monoidal structure of $\cz(\cx)$ to $\Tube$-modules via this equivalence, in view of the connection to the Levin--Wen type lattice models~\cites{MR4642306,DOI:10.22331/q-2024-03-28-1301}, it is desirable to have a direct description of fusion and braiding on the category of modules over the tube algebra, as conjectured in~\cite{MR2503392}.
In this direction, Das, Ghosh, and Gupta~\cite{MR3254423} showed that, starting from a shaded planar algebra (equivalent to a unitary rigid $2$-category $\cx$ generated by a single $1$-cell), the unitary representations of the associated annular algebra has a structure of unitary braided category, and that a certain subcategory of `finite modules' is braided equivalent to the Drinfeld center $\cz(\cx)$.
A similar claim can be found in~\cite{hoek-master-thesis}, although precise details seems to be missing there.

In this paper we relax the unitarity, single generation, and finiteness assumptions, and establish a corresponding comparison result.
That is, we define a braided monoidal structure on the category of modules over the tube algebra $\Tubecx$ of a pivotal $2$-category $\cx$ and a generating set $\IndexSet$ of objects  ($1$-cells).
Our main result (Theorem~\ref{thm:main}) shows that this indeed recovers the (ribbon) braided monoidal category $\cz(\cx)$.

Our model of monoidal product and braiding is loosely motivated by the topological presentation of corresponding structures on the fusion ring in~\cite{MR1317353}.
We hope that this gives a more intuitive description of topological operations on the Levin--Wen type lattice models than the conventional ones, as in~\cite{MR3254423}, based on string diagrams.
% Other references~\cite{arXiv:2503.06731,hoek-master-thesis}, David Green

The outline of paper is as follows.
In Section~\ref{sec:prelim}, we review some basic materials on rigid $2$-categories.
In Section~\ref{sec:graph-calc}, we set up a graphical calculus scheme adapted from the standard topological presentation of $2$-categories and planar algebra theory.
In Section~\ref{sec:tub-alg-rep}, we define the monoidal structure and braiding on the category of modules over the tube algebra, using our scheme from Section~\ref{sec:graph-calc}.
Finally in Section~\ref{sec:drinfeld_center}, we carry out the comparison with the Drinfeld center.

\paragraph{Acknowledgments}
This research was funded by The Research Council of Norway [project 324944].
We would like to thank Corey Jones for fruitful discussions and encouragements during the work on this project, and David Penneys for comments on an early draft of this paper.

\section{Preliminaries}\label{sec:prelim}

\subsection{Spherical multitensor categories}

Let us begin with a quick review of the basic concepts on multitensor categories, mostly following the convention of~\cite{EGNO}.

Let $\bk$ be an algebraically closed field.
A \textit{multitensor category over $\bk$} is a locally finite $\bk$-linear abelian category $\cx$ together with a rigid monoidal structure, such that the tensor product $\otimes$ is bilinear on morphisms.
Without loss of generality we may and do assume $\cx$ to be strict, in view of MacLane's coherence theorem.

We have the canonical irreducible decomposition of monoidal the unit,
\[
\un = \bigoplus_{i \in \Gamma} \un_i
\]
for the finite set $\Gamma$ parametrizing the primitive idempotents of the semisimple commutative algebra $\End_\cx(\un)$.
This induces the corresponding decomposition of the category
\[
\cx = \bigoplus_{i, j \in \Gamma} \cx_{i j},
\]
where the subcategory $\cx_{i j} \subset \cx$ consists of the objects $x$ satisfying $x \cong \un_i \otimes x \otimes \un_j$.
We use the notation $\cx_i\coloneqq \cx_{ii}$ for the diagonal tensor categories. We freely identify $\cx$ with the associated $2$-category, whose $0$-cells are elements of $\Gamma$, $1$-cells are objects of $\cx_{i j}$, and $2$-cells are morphisms of $\cx_{i j}$. The horizontal composition is given by the tensor product~\cite[Remark 4.3.7]{EGNO}.

A \textit{pivotal structure} on $\cx$ is given by a monoidal natural isomorphism from the identity functor of $\cx$ to the double dual functor, i.e., a natural family of isomorphisms
\[
p_x \colon x \xrightarrow{\sim} x^{**}
\]
compatible with the natural isomorphism $(x \otimes y)^{**} \cong x^{**} \otimes y^{**}$.
By~\cite[Theorem 2.2]{NS07pivotal}, we may assume that $p$ is strict, meaning that $(x\otimes y)^{*}=y^{*}\otimes x^{*} $ and $p_{x}=\id_x$.
We also identify left and right duals, i.e., for every $x\in\cx$ we may assume
\begin{equation*}
  x^*={}^*x,
\end{equation*}
and write $\overline{x}$ for this object. 
We can thus denote the corresponding evaluation and coevaluation morphisms by
\begin{align*}
  {\ev}_x&\colon \overline{x} \otimes x\to \un, &
  {\coev}_x&\colon \un \to x \otimes \overline{x}, \\
  \tev_x&\colon x \otimes \overline{x}\to \un, &
  \tcoev_x&\colon \un\to \overline{x}\otimes x.   
\end{align*}  

These allow us to define left and right traces of an endomorphism $f\in\Hom_\cx(x,x)$ as
\begin{equation*}
    \tr_x^\mathrm{r}(f)\coloneqq  \tev_x\circ (f\otimes\id) \circ\coev_x,\quad\text{ and }\quad \tr_x^\mathrm{l}(f)\coloneqq  \ev_x\circ (\id\otimes f) \circ\tcoev_x\,.
\end{equation*}
In a pivotal multitensor category, $\overline{x}\in\cx_{j i}$ holds for a homogeneous object $x\in\cx_{i j}$. In that case, we have that $\tr_x^\mathrm{r}(f)\in \Hom_\cx\!\left(\un_j,\un_j\right)$ and $\tr_x^\mathrm{l}(f)\in \Hom_\cx\!\left(\un_i,\un_i\right)$. We identify these traces as scalars in the base field via the unique unital isomorphism of $\bk$-algebras
\begin{equation}\label{scalar_id}
\Hom_\cx\!\left(\un_i,\un_i\right) \xrightarrow{~\sim~} \bk\,,\quad\id_{\un_i}\mapsto 1 \,.
\end{equation}

Another consequence of the pivotal structure is the equality
\begin{equation}\label{eq:triv-monodromy}
((\tev_x \circ (\id \otimes \ev_y \otimes \id)) \otimes \id) \circ (\id \otimes f \otimes \id) \circ (\id \otimes ((\id \otimes \coev_x \otimes \id) \circ \tcoev_y)) = f
\end{equation}
for $f \in \Hom_\cx(x, y)$.

\begin{defi}[cf.~\cites{MR1966524,MR3342166}]\label{def:spherical_multitensor_category}
A pivotal multitensor category $\cx$ is called \textit{spherical} when left and right traces coincide under~\eqref{scalar_id} for every endomorphism of $f\in\Hom_{\cx_{i j}}(x,x)$, for all $i, j \in \Gamma$ and all $x \in \cx_{i j}$. In that case, we denote this scalar by $\tr_x(f)$.
\end{defi}
The \textit{dimension} of an object $x\in\cx_{i j}$ is the scalar 
$\rd_x\coloneq\tr_x(\id_x)$. Sphericality ensures that the dimension of dual objects obey $\rd_{\overline{x}}=\rd_x$.

The sphericality condition in Definition~\ref{def:spherical_multitensor_category} is equivalent to the following~\cite[Proposition 5.8]{MR1966524}:
for any fixed $i_0 \in \Gamma$, the rigid tensor category $\cx_{i_0}$ is spherical, and for each $j \in \Gamma$ and any fixed irreducible object $x \in \cx_{i_0 j}$, the Frobenius algebra $Q = x \otimes \overline{x} \in \cx_{i_0}$ is normalized in the sense of~\cite[Definition 3.13]{MR1966524}.

\begin{rem}
In the finite setting, given a spherical fusion category $\cx$ there is a pivotal $2$-category $\mathbf{Mod}^\mathrm{sph}(\cx)$ of spherical module categories over $\cx$ as defined in~\cite[Section 3.6]{manifestlymorita} and~\cite[Definition 5.22]{sphericalmorita}. By considering a choice of simple objects in this semisimple $2$-category as indexing set $\Gamma$, we obtain a spherical multifusion category as in Definition~\ref{def:spherical_multitensor_category}.
Reciprocally, given a spherical multifusion category $\cx$, we have that $\cx_i$ is a spherical fusion category for any $i\in \Gamma$. Moreover, the (idempotent completed~\cite{douglas2018fusion2categories}) $2$-category associated to $\cx$ can be realized as as a subcategory of the $2$-category $\mathbf{Mod}^\mathrm{sph}(\cx_i)$.
\end{rem}

Throughout the paper we will tacitly assume that $\cx$ is semisimple and spherical.
In particular, any object of $\cx$ is a direct sum of simple objects, but we allow for $\cx$ to have infinitely many isomorphism classes of simple objects.

We fix for each $i,j\in\Gamma$ a set $\Iso \cx_{i j}$ (resp.~$\Irr\cx_{i j}$) of representatives of isomorphism classes of objects (resp.~simple objects) in $\cx_{i j}$.

We will use the following notation
\[
x_1\otimes \cdots \otimes x_n\equiv x_1\cdots x_n
\]
in order to shorten lengthy expressions involving the tensor product of several objects in $\cx$.
In order to simplify the presentation, we assume these objects to be from homogeneous components $\cx_{i j}$.
In this case, the above product is nontrivial if and only if each object $x_k$ belongs to $\cx_{i_k i_{k + 1}}$, and is nonzero, for some choice of $0$-cells $i_1, \ldots, i_{n + 1} \in \Gamma$.

\subsection{Ind-completion}

The \textit{ind-completion} of a category $\cx$ is the category $\Ind\cx$ whose objects are colimits of filtered diagrams in $\cx$.
Concretely, one can represent its objects by functors $F\colon I\to \cx$ where $I$ is a filtered category, and with morphism spaces between two such functors given by
\begin{equation*}
    \Hom_{\Ind\cx}(F,G)\coloneqq {\lim}_i\mathrm{colim}_j\Hom_\cx(F(i),G(j)).
\end{equation*}

When $\cx$ is a semisimple $\bk$-linear category, every object in $\Ind\cx$ is a formal direct product of objects in $\cx$, that is,
\begin{equation}\label{eq:formal-dir-sum}
\bigoplus_{a \in \Irr\cx} V_a \otimes a
\end{equation}
for a collection of vector spaces $V_a$ labeled by the irreducible classes, and the morphism space between two such objects is given by
\[
\Hom_{\Ind\cx}\biggl(\bigoplus_{a \in \Irr\cx} V_a \otimes a, \bigoplus_{a \in \Irr\cx} W_a \otimes a \biggr) = \prod_{a \in \Irr\cx} \Hom(V_a, W_a).
\]
Indeed,~\eqref{eq:formal-dir-sum} can be considered as the colimit of diagram $F \colon I \to \cx$ where $I$ is the category whose objects are $V' = \bigoplus_{a \in \Irr\cx} V'_a$, a direct sum of subspaces $V'_a \subset V_a$ for $a \in \Irr\cx$ such that $\dim V' < \infty$, ordered by inclusion.
Conversely, suppose we are given a filtered diagram $F \colon I \to \cx$.
Then, for each fixed $a \in \Irr\cx$, the finite-dimensional vector spaces $\Hom(a, F(i))$ form a directed system.
Writing $V_a$ for its colimit, we see that~\eqref{eq:formal-dir-sum} is a presentation of the overall colimit of $F$.

Furthermore, when $\cx$ is a multitensor category, the $\Ind\cx$ inherits from $\cx$ a monoidal structure with the same set of $0$-cells.
However, $\Ind\cx$ is not rigid since we are allowing infinite direct sums of $\un$, among others, as its objects.

\subsection{Drinfeld center}

Next, let us recall the \emph{Drinfeld center} $\cz(\Ind\cx)$ of the multitensor category $\Ind\cx$ for a semisimple multitensor category, cf.~\cites{MR3398725,NY15tubealgebra}.
Its objects are pairs $(Z, \sigma^Z)$, where $Z \in \Ind\cx$ and $\sigma^Z$ is a natural isomorphism
\[
\sigma^Z_X \colon X \otimes Z \to Z \otimes X
\]
obeying the corresponding hexagon axiom.
The centrality condition against $X = \un$ shows that $Z$ must be of the form $\bigoplus_{i \in \Gamma} Z_i$ for $Z_i \in \Ind\cx_{i }$.
Moreover, $\sigma^Z$ is completely determined by its action on objects of $\cx_{i j}$ for $i, j \in \Gamma$, so that our data can be represented by natural isomorphisms
\begin{equation*}
\sigma^Z_a \Colon  a \otimes Z_j \to Z_i \otimes a \quad (i, j \in \Gamma, a \in \cx_{i j}).
\end{equation*}
Morphisms in $\cz(\Ind\cx)$ are morphisms in $\Ind\cx$ between the underlying objects that commute with the respective half-braidings.

Of course, we can restrict to the case $Z_i \in \cx_{i }$, which gives the category $\cz(\cx)$.
If we further assume that $\cx$ is a tensor category, we recover the usual definition of the Drinfeld center category.

The monoidal product of $\cx$ induces a monoidal structure on $\cz(\Ind\cx)$ where the underlying object of $Z \otimes Z'$ is given by $(Z_i \otimes Z'_i)_{i \in \Gamma}$, endowed with the composition of the respective half-braidings.
Moreover, $\cz(\Ind\cx)$ is a braided category, whose braiding $(X, \sigma^X) \otimes (Y, \sigma^Y) \to (Y, \sigma^Y) \otimes (X, \sigma^X)$ is given by
\begin{equation*}%\label{eq:braiding_center}
\sigma_{X,Y}=\sigma^{Y}_X\Colon X \otimes Y \to Y \otimes X,
\end{equation*}
i.e., the action of the half-braiding on $Y$.

Under the rigidity assumption on $\cx$, for each $i \in \Gamma$, the restriction functor $\cz(\Ind\cx) \to \cz(\Ind\cx_{i })$ given by $Z \mapsto Z_i$ are braided monoidal equivalences, see~\cites{MR1822847,MR1966524}.

\subsection{Algebras with local units}
To work with tube algebras in the infinite setting, we will need to forgo assuming the existence of a unit.
\begin{defi}[\cites{MR695890,MR899719}]
An associative \textit{$\bk$-algebra with local units} consists of a $\bk$-vector space $A$ together with a product, i.e., a linear map $\cdot\colon A\otimes A\to A$ obeying the following properties.
\begin{itemize}
\item Associativity: for every $f,g,h\in A$
\begin{equation*}
f\cdot(g\cdot h)= (f\cdot g)\cdot h    
\end{equation*}
\item Local units: for every finite dimensional vector subspace $V\subset A$ there exists an idempotent $e\in A$ such that $V\subset eAe$.
\end{itemize}
\end{defi}
Similar to the case of unital algebras, there is an associated notion of representation for algebras with local units.
\begin{defi}
Let $A$ be an associative $\bk$-algebra with local units.
\begin{enumerate}[\rm (i)]
    \item 
A \textit{locally presented 
$($right$)$ representation} of $A$ is a $\bk$-vector space $M$ together with a module action, i.e., a linear map $M\otimes A\to M, m \otimes f \mapsto m.f$ fulfilling the following properties.
\begin{itemize}
    \item Module associativity: $m.(f\cdot g)=(m.f).g$ for $m\in M$ and $f,g\in A$.    
    \item Locality: for every finite dimensional vector subspace $W\subset M$ there exists an idempotent $e\in A$ such that $W.e=W$.
\end{itemize}
\item A \textit{morphism of locally presented representations} is simply a module homomorphism, i.e., a linear map $F\colon M\to N$ obeying $F(m.f)=F(m).f$ for $m\in M$ and $f\in A$.
\item We denote by $\Rep A$ the $\bk$-linear category of locally presented representations over $A$.
\end{enumerate}
\end{defi}

\section{Graphical calculus and spaces of tubes}\label{sec:graph-calc}

\subsection{Cyclically invariant vector spaces of morphisms}
Let $\cx$ be a pivotal multitensor category. In the spirit of~\cite{BK1}, we define in this section a cyclically invariant vector space of morphisms.
Let $x_1,\ldots, x_n$ be objects of $\cx$.
Then the vector space $\Hom_\cx(\un,x_1\cdots x_n)$ only depends on the cyclic order of $(x_1,\ldots,x_n)$ up to canonical isomorphism.
We will mostly consider the case where $x_1,\ldots, x_n$ are homogeneous and \emph{cyclically composable}, in the sense that $x_k \in \cx_{i_k i_{k+1}}$ with $i_{n+1} = i_1$.

To make this more precise, consider an indexing category $\mathfrak{n}$ whose objects are $\{1,2,\ldots, n\}$ and with a unique isomorphism between every pair of objects. We define a $\mathfrak{n}$-shaped diagram of finite dimensional vector spaces 
\begin{equation*}
    \mathrm{X}\Colon \mathfrak{n}\to\kVec,\qquad i\mapsto \Hom_\cx(\un,x_i \cdots  x_n x_1\cdots x_{i-1})
\end{equation*}
assigning to the unique isomorphism between $i$ and $i+1$ the isomorphism
\begin{equation}\label{cyclic_perm}
\begin{aligned}
\mathrm{X}_{i,i+1}\Colon \Hom_\cx(\un,x_i \cdots  x_n x_1\cdots x_{i-1})&\to
\Hom_\cx(\un,x_{i+1} \cdots  x_n x_1\cdots x_{i-1}x_i)\\
    f&\mapsto (\ev_{x_i}\otimes \id) \circ (\id\otimes f \otimes \id) \circ \tcoev_{x_i}
\end{aligned}
\end{equation}
Equalities of the form~\eqref{eq:triv-monodromy} ensure that the composition $\mathrm{X}_{n,1}\circ\cdots\circ \mathrm{X}_{1,2}$ equals the identity map of the vector space $\Hom_\cx(\un,x_1\cdots x_n)$. The identification of the double-dual of a tensor product as the product of the double-duals implies the functoriality of $\mathrm{X}$.
% and
% \begin{equation}
% \begin{aligned}
% \overline{\Gamma}_{i,i+1}\colon \Hom_\cx(\un,\overline{x}_{i-1} \cdots  \overline{x_{i}})&\to
% \Hom_\cx(\un,\overline{x}_{i} \cdots  \overline{x_{i+1}})\\
%     f&\mapsto \id\otimes \tev_{x_i} \circ f\otimes \id \circ \coev_{x_i}
% \end{aligned}
% \end{equation}

\begin{defi}Let $\cx$ be a pivotal multitensor category and $x_1, \ldots ,x_n$ cyclically composable objects in $\cx$.
Define the \textit{vector space of morphisms bounded by $x_1, \ldots ,x_n$} as the limit
\begin{equation}\label{vsp_limit}
    \Hsp{x_1\cdots x_n}\coloneqq \lim \mathrm{X}
\end{equation}    
in the category $\kVec$ of finite-dimensional vector spaces.
\end{defi}
\begin{rem}By construction the vector space~\eqref{vsp_limit} is invariant under cyclic permutations of $(x_1, \ldots, x_n)$, thus we canonically identify
\begin{equation*}
     \Hsp{x_1\cdots x_n} = \Hsp{x_2\cdots x_n x_1} = \cdots = \Hsp{x_n x_1\cdots x_{n-1}}.
\end{equation*}
As remarked before, we will be mostly interested in the case when each $x_k$ is homogeneous, that is, $x_k \in \cx_{i_k i_{k+1}}$ for $i_1, \ldots, i_{n} \in \Gamma$, with the convention $i_{n + 1} = i_1$.
We thus identify $\Hom_{\cx_{i_1}}(\un_{i_1}, x_1 \cdots x_n)$ with $\Hom_{\cx_{i_2}}(\un_{i_2}, x_2 \cdots x_n x_1)$, etc., by the maps of the form~\eqref{cyclic_perm}, and call this $\Hsp{x_1 \cdots x_n}$.
\end{rem}

Rigidity allows us to define a composition relative to an object $b\in\cx$ via the following map
\begin{equation}\label{comp_t}
\begin{aligned}
\circ_b\Colon \Hom_\cx\left(\un,x\overline{b}\,\right)\otimes \Hom_\cx(\un,by)&\to \Hom_\cx(\un,xy),\\% \quad 
(f,g)&\mapsto (\id_x\otimes\ev_{b}\otimes\id_y)\circ (f\otimes g).
\end{aligned}
% \frac{1}{\sqrt{\rd_x\rd_y}}\tr(f\circ g)\,.
\end{equation}

The composition map~\eqref{comp_t} relative to an object $b\in\cx$, via the isomorphisms~\eqref{cyclic_perm}, descends to a map between cyclically invariant vector spaces
\begin{equation}\label{comp_b}
\circ_b\colon  \Hsp{x \,\overline{b}}\otimes \Hsp{b\,y}\to \Hsp{x\,y}
\end{equation}
In the case that $\cx$ is semisimple, these induce an isomorphism
\begin{equation*}
    \sum_{t\in \Irr \cx}\,\circ_t\Colon\bigoplus_{t\in \Irr \cx} \Hsp{x\,\overline{t}}\otimes \Hsp{t\,y}\xrightarrow{\;\cong\;} \Hsp{x\,y}
\end{equation*}

\begin{rem}
For each object $x\in\cx$  consider the universal limit cone defining $\Hsp{\overline{x}x}$:
\begin{equation*}
\begin{tikzcd}
{\Hsp{\overline{x}x}} \ar[d,equal] \ar[r] & \Hom_\cx(\un,\overline{x}x)\ar[d,"\eqref{cyclic_perm}"]\\
{\Hsp{x\overline{x}}} \ar[r]& \Hom_\cx(\un,x\overline{x})
\end{tikzcd}
\end{equation*}
the image of $\tcoev_x$ under the vertical isomorphism is the coevaluation morphism $\coev_x$.
We denote their common inverse image in $\Hsp{\overline{x}x}=\Hsp{x\overline{x}}$ by $\Id_x$, and refer to it as \textit{the unit morphism of $x$}. This morphism serves as identity of relative composition~\eqref{comp_b} as shown in Lemma~\ref{comp_prop}.

The unit morphism of $x\in\cx_{i j}$ can also be used to define
\textit{partial traces}: given $f\in\Hsp{yx\overline{x}}$ 
\begin{equation*}
    \ptr_x(f)\coloneqq f\circ_{x\overline{x}} \,\Id_x\in \Hsp{y}
\end{equation*}
\textit{Full traces}, or simply \textit{traces} are similarly defined for \textit{endomorphisms} $f\in\Hsp{x\overline{x}}$ with $x \in \cx_{i j}$ by
\begin{equation*}
    \tr_x(f)\coloneqq f\circ_{x\overline{x}} \,\Id_x\in \Hsp{\un_i}\overset{\eqref{scalar_id}}{\cong}\bk\,.
\end{equation*}
\end{rem}
(A word of caution: for arbitrary $x\in\cx$ we have natural isomorphisms $\Hsp{x \overline{x}} \cong \Hsp{x \overline{x}}$, but $\tr_x(f)$ and $\tr_{\overline{x}}(f)$ are different elements of $\bk^\Gamma$.
When $x \in \cx_{i j}$, these elements have only one nonzero components, respectively at $i$-th and $j$-th components, which are the same scalars.)

\begin{lem}\label{comp_prop}
Let $\cx$ be a pivotal multitensor category.
\begin{enumerate}[$($i$)$]
    \item The unit morphism of an object $x\in\Iso\cx_{ij}$ obeys
    \begin{equation*}
    \Id_x=\Id_{\overline{x}}, \quad \Id_{xy}=\Id_x\circ_{\un_j}\Id_y, \qquad\text{ and }\qquad f\circ_x \Id_x= f \,\in\Hsp{yx\overline{x}}\,,
    \end{equation*}
    for any $f\in\Hsp{yx\overline{x}}$, which justifies its name.
    \item Given $f\in\Hsp{x\overline{y}} $ and $g\in\Hsp{yz}$, it holds that 
    \begin{equation*}
        f\circ_y g=g\circ_{\overline{y}} f
    \end{equation*}
    as elements of $\Hsp{xz}$.
    \item For $f\in\Hsp{w\overline{x}\,\overline{y}}$ and $g\in\Hsp{yxz}$, it holds that 
    \begin{equation*}
        f\circ_{xy} g=\ptr_x(f\circ_{y} g)=\ptr_y(f\circ_{x} g)
    \end{equation*}
    as elements of $\Hsp{xz}$.    
    \item For $f\in\Hsp{xy_1y_2}$, $g\in\Hsp{\overline{y}_1w}$ and $h\in\Hsp{\overline{y}_2z}$, it holds that 
    \begin{equation*}
        (g\circ_{y_1} f)\circ_{\overline{y}_2}h=g\circ_{y_1} (f\circ_{\overline{y}_2}h)
    \end{equation*}
    as elements of $\Hsp{xwz}$.
\end{enumerate}
\end{lem}
\begin{proof}
Statement (i) follows from $\coev_{\overline{x}} = \tcoev_x$.
Next, for (ii), take $f\in\Hom_\cx(\un, x\overline{y})$ and $g\in\Hom_\cx(\un, y\overline{z})$.
We then have the equality
\begin{equation*}
    f\circ_y g=(\id \otimes \ev_y \otimes \id) \circ (f\otimes g)= (\id \otimes \tev_y \otimes \id) \circ (g\otimes f)
\end{equation*}
in $\Hsp{x z}$.
Item (iii) follows from the standard snake relations for the structure morphisms of duality.
Finally, (iv) follows from associativity of the composition in $\cx$ and uses the fact that the double dual of a tensor product is the product of the double duals.     
\end{proof}

\subsection{Diagrammatic calculus}

We set up the conventions for a graphical calculus using the $2$-dimensional graphical representation of $2$-categories (which in a sense is the dual graph to the usual string diagram for tensor categories).

First, a $0$-cell $i \in \Gamma$ is depicted by a vertex.
Properly speaking, vertices should be decorated by the labels $i \in \Gamma$, but we supress this from our diagrams.

Next, a $1$-cell, that is, a homogeneous object $x\in \Iso\cx_{i j}$ will be depicted by an oriented line segment.
Changing the orientation of the edge results in changing the label by the dual object.
\[\def\len{2.5cm}
\begin{tikzpicture}[baseline=(A1.base)]
\draw[line width=1]
(-0.5*\len,0) to node[midway,rotate=0](A1) { $\Arr$}
    node[midway,below] { $x$}(0.5*\len,0);
\filldraw [black] (0.5*\len,0)  circle (2pt);    
\filldraw [black] (-0.5*\len,0)  circle (2pt);      
\end{tikzpicture}
\quad=\quad
\begin{tikzpicture}[baseline=(A2.base)]
\draw[line width=1]
(-0.5*\len,0) to node[midway](A2) { $\blacktriangleleft$}
    node[midway,below] { $\overline{x}$}(0.5*\len,0);
\filldraw [black] (0.5*\len,0)  circle (2pt);    
\filldraw [black] (-0.5*\len,0)  circle (2pt);      
\end{tikzpicture}
\]
The $i$-th component of the monoidal unit is depicted by a dotted edge:
\begin{equation*}
\def\len{3cm}
\begin{tikzpicture}
\draw[line width=1.5,dotted]
(-0.5*\len,0) to  node[midway,below] { $\un_i$}(0.5*\len,0);
\filldraw [black] (0.5*\len,0)  circle (2pt);    
\filldraw [black] (-0.5*\len,0)  circle (2pt);      
\node at (0.5*\len,0) [anchor=north]{\small $i$};
\node at (-0.5*\len,0) [anchor=north]{\small $i$};
\end{tikzpicture}
\end{equation*}
and as mentioned we will suppress the vertex decoration in our diagrams. Often we will also omit the subscript $i$ on $\un_i$ since this is determined by the 
labels adjacent to the vertices.
Finally, a $2$-cell, which we take to be an element of $\Hsp{x_1\cdots x_n}$, will be depicted by an oriented disk (we will always use by convention a counterclockwise orientation) whose boundary comes with an oriented graph whose edges are labeled by composable homogeneous objects $x_1,\dots, x_n$. For instance, an element $f\in \Hsp{x_1x_2x_3x_4 x_5}$ is depicted by
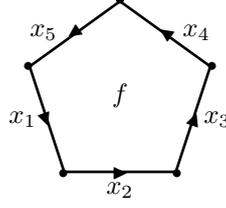
\[
\def\len{2.5cm}
\begin{tikzpicture}[baseline={([yshift=-.5ex]current bounding box.center)}]
    \node[regular polygon, draw=black, line width=1, regular polygon sides = 5, minimum size=\len] (p) at (0,0) {$f$};
    \foreach \x in {1,2,...,5}  \fill (p.corner \x) circle (1.5pt);
    \node at (p.270) {\scriptsize $\Arr$};
    \node at (p.54) [rotate=145] {\scriptsize $\Arr$};
    \node at (p.126) [rotate=-145] {\scriptsize $\Arr$};
    \node at (p.342) [rotate=70]  {\scriptsize $\Arr$};
    \node at (p.198) [rotate=-70] {\scriptsize $\Arr$};
    \node at (p.198) [anchor=east] {$x_1$};
    \node at (p.270) [anchor=north] {$x_2$};  
    \node at (p.342) [anchor=west] {$x_3$};
    \node at (p.54) [anchor=west] {$\;x_4$};    
    \node at (p.126) [anchor=east] {$x_5\;$};        
\end{tikzpicture}
\]
Different graphs correspond to the same element of $\Hsp{x_1\cdots x_n}$ in the following situations:
\begin{enumerate}[(i)]
    \item Edge fusion (vertex removal): two consecutive edges meeting at a vertex can be seen as an edge with label being the tensor product of the starting edges.
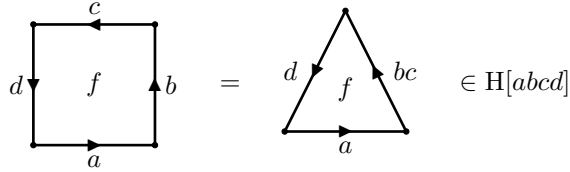
\[\def\len{1.6cm}
\begin{tikzpicture}[baseline={([yshift=-.5ex]current bounding box.center)}]
\draw[line width=1]
(-0.5*\len,0) to node[midway,rotate=0] {\scriptsize $\Arr$}
    node[midway,below] { $a$}(0.5*\len,0)
(-0.5*\len,0) to node[midway,rotate=-90] {\scriptsize $\Arr$}
    node[midway,left] { $d$}(-0.5*\len,1*\len)
(-0.5*\len,1*\len) to node[midway,rotate=180] {\scriptsize $\Arr$}
    node[midway,above] { $c$}(0.5*\len,1*\len)
(0.5*\len,0) to node[midway,rotate=90] {\scriptsize $\Arr$}
    node[midway,right] { $b$}(0.5*\len,1*\len) ;
    
\filldraw [black] (0.5*\len,1*\len)  circle (1pt);    
\filldraw [black] (-0.5*\len,1*\len)  circle (1pt);    
\filldraw [black] (0.5*\len,0)  circle (1pt);    
\filldraw [black] (-0.5*\len,0)  circle (1pt);    
\node at (0,0.5*\len) {$f$};
\end{tikzpicture}
\quad=\quad
\begin{tikzpicture}[baseline={([yshift=-.5ex]current bounding box.center)}]
\draw[line width=1]
(-0.5*\len,0) to node[midway] {\scriptsize $\Arr$}
    node[midway,below] { $a$}(0.5*\len,0)
(-0.5*\len,0) to node[midway,rotate=-115] {\scriptsize $\Arr$}
    node[midway,left] { $d\;$}(0,1*\len)
(0.5*\len,0) to node[midway,rotate=115] {\scriptsize $\Arr$}
    node[midway,right] { $\;bc$}(0,1*\len) ;
    
\filldraw [black] (0,1*\len)  circle (1pt);    
\filldraw [black] (0.5*\len,0)  circle (1pt);    
\filldraw [black] (-0.5*\len,0)  circle (1pt);    
\node at (0,0.35*\len) {$f$};
\end{tikzpicture}
\quad\in\Hsp{abcd}
\]
\item Unit insertion: a dotted edge can be inserted at a vertex.%\vspace{-0.5em}
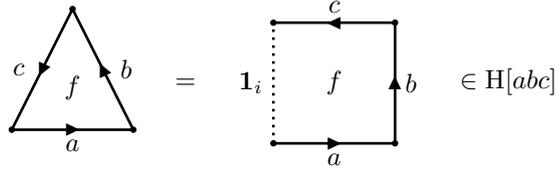
\begin{equation*}
\def\len{1.6cm}
\begin{tikzpicture}[baseline={([yshift=-.5ex]current bounding box.center)}]
\draw[line width=1]
(-0.5*\len,0) to node[midway] {\scriptsize $\Arr$}
    node[midway,below] {$a$}(0.5*\len,0)
(-0.5*\len,0) to node[midway,rotate=-115] {\scriptsize $\Arr$}
    node[midway,left] { $c\;$}(0,1*\len)
(0.5*\len,0) to node[midway,rotate=115] {\scriptsize $\Arr$}
    node[midway,right] { $\;b$}(0,1*\len) ;
    
\filldraw [black] (0,1*\len)  circle (1pt);    
\filldraw [black] (0.5*\len,0)  circle (1pt);    
\filldraw [black] (-0.5*\len,0)  circle (1pt);    
\node at (0,0.35*\len) {$f$};
\end{tikzpicture}
\quad=\quad
\begin{tikzpicture}[baseline={([yshift=-.5ex]current bounding box.center)}]
\draw[line width=1]
(-0.5*\len,0) to node[midway] {\scriptsize $\Arr$}
    node[midway,below] { $a$}(0.5*\len,0)
(0.5*\len,0) to node[midway,rotate=90] {\scriptsize $\Arr$}
    node[midway,right] { $b$}(0.5*\len,1*\len)
(-0.5*\len,1*\len) to node[midway,rotate=180] {\scriptsize $\Arr$}
    node[midway,above] { $c$}(0.5*\len,1*\len);
\draw[line width=1,dotted]    
(-0.5*\len,0) to  node[midway,left] { $\un_i$}(-0.5*\len,1*\len) ;
    
\filldraw [black] (0.5*\len,1*\len)  circle (1pt);    
\filldraw [black] (-0.5*\len,1*\len)  circle (1pt);    
\filldraw [black] (0.5*\len,0)  circle (1pt);    
\filldraw [black] (-0.5*\len,0)  circle (1pt);    
\node at (0,0.5*\len) {$f$};
\end{tikzpicture}
\quad\in\Hsp{abc}    
\end{equation*}
where $a\in\cx_{ij}$, $b\in\cx_{jk}$ and $c\in\cx_{ki}$.
\item Gluing: the composition map~\eqref{comp_b} is depicted by the gluing of the corresponding disks along the common boundary label. We have that
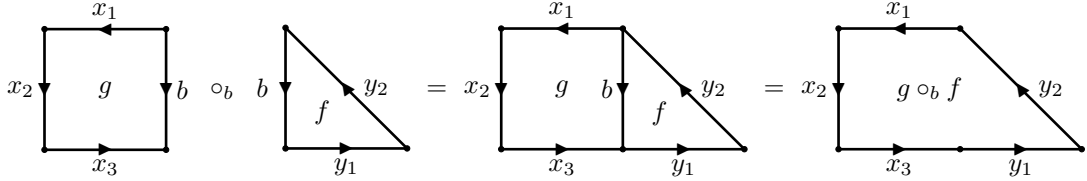
\[\def\len{1.6cm}
\begin{tikzpicture}[baseline={([yshift=-.5ex]current bounding box.center)}]
\draw[line width=1]
(-0.5*\len,0) to node[midway] {\scriptsize $\Arr$}
    node[midway,below] { $x_3$}(0.5*\len,0)
(-0.5*\len,0) to node[midway,rotate=-90] {\scriptsize $\Arr$}
    node[midway,left] { $x_2$}(-0.5*\len,1*\len)
(-0.5*\len,1*\len) to node[midway,rotate=180] {\scriptsize $\Arr$}
    node[midway,above] { $x_1$}(0.5*\len,1*\len)
(0.5*\len,0) to node[midway,rotate=-90] {\scriptsize $\Arr$}
    node[midway,right] { $b$}(0.5*\len,1*\len) ;
    
\filldraw [black] (0.5*\len,1*\len)  circle (1pt);    
\filldraw [black] (-0.5*\len,1*\len)  circle (1pt);    
\filldraw [black] (0.5*\len,0)  circle (1pt);    
\filldraw [black] (-0.5*\len,0)  circle (1pt);    
\node at (0,0.5*\len) {$g$};
\end{tikzpicture}
\;\circ_b\;
\begin{tikzpicture}[baseline={([yshift=0.8ex]current bounding box.center)}]
\draw[line width=1]
(-0.5*\len,0) to node[midway] {\scriptsize $\Arr$}
    node[midway,below] { $y_1$}(0.5*\len,0)
(0.5*\len,0) to node[midway,rotate=125] {\scriptsize $\Arr$}
    node[midway,right] { $\;y_2$}(-0.5*\len,1*\len)
(-0.5*\len,0) to node[midway,rotate=-90] {\scriptsize $\Arr$}
    node[midway,left] { $b\;$}(-0.5*\len,1*\len) ;
    
\filldraw [black] (-0.5*\len,1*\len)  circle (1pt);    
\filldraw [black] (0.5*\len,0)  circle (1pt);    
\filldraw [black] (-0.5*\len,0)  circle (1pt);    
\node at (-0.2*\len,0.3*\len) {$f$};
\end{tikzpicture}
~=
\begin{tikzpicture}[baseline={([yshift=-.5ex]current bounding box.center)}]
\draw[line width=1]
(-0.5*\len,0) to node[midway] {\scriptsize $\Arr$}
    node[midway,below] { $x_3$}(0.5*\len,0)
(-0.5*\len,0) to node[midway,rotate=-90] {\scriptsize $\Arr$}
    node[midway,left] { $x_2$}(-0.5*\len,1*\len)
(-0.5*\len,1*\len) to node[midway,rotate=180] {\scriptsize $\Arr$}
    node[midway,above] { $x_1$}(0.5*\len,1*\len)
(0.5*\len,0) to node[midway,rotate=-90] {\scriptsize $\Arr$}
    node[midway,left] { $b$}(0.5*\len,1*\len) ;
    
\filldraw [black] (0.5*\len,1*\len)  circle (1pt);    
\filldraw [black] (-0.5*\len,1*\len)  circle (1pt);    
\filldraw [black] (0.5*\len,0)  circle (1pt);    
\filldraw [black] (-0.5*\len,0)  circle (1pt);    
\node at (0,0.5*\len) {$g$};
\begin{scope}[xshift=1*\len]
\draw[line width=1]
(-0.5*\len,0) to node[midway] {\scriptsize $\Arr$}
    node[midway,below] { $y_1$}(0.5*\len,0)
(0.5*\len,0) to node[midway,rotate=125] {\scriptsize $\Arr$}
    node[midway,right] { $\;y_2$}(-0.5*\len,1*\len) ;

\filldraw [black] (0.5*\len,0)  circle (1pt);    

\node at (-0.2*\len,0.3*\len) {$f$};
\end{scope}
\end{tikzpicture}
~=
\begin{tikzpicture}[baseline={([yshift=-.5ex]current bounding box.center)}]
\draw[line width=1]
(-0.5*\len,0) to node[midway] {\scriptsize $\Arr$}
    node[midway,below] { $x_3$}(0.5*\len,0)
(-0.5*\len,0) to node[midway,rotate=-90] {\scriptsize $\Arr$}
    node[midway,left] { $x_2$}(-0.5*\len,1*\len)
(-0.5*\len,1*\len) to node[midway,rotate=180] {\scriptsize $\Arr$}
    node[midway,above] { $x_1$}(0.5*\len,1*\len);
    
\filldraw [black] (0.5*\len,1*\len)  circle (1pt);    
\filldraw [black] (-0.5*\len,1*\len)  circle (1pt);    
\filldraw [black] (0.5*\len,0)  circle (1pt);    
\filldraw [black] (-0.5*\len,0)  circle (1pt);    
\node at (0.25*\len,0.5*\len) {$g\circ_b f$};
\begin{scope}[xshift=1*\len]
\draw[line width=1]
(-0.5*\len,0) to node[midway] {\scriptsize $\Arr$}
    node[midway,below] { $y_1$}(0.5*\len,0)
(0.5*\len,0) to node[midway,rotate=125] {\scriptsize $\Arr$}
    node[midway,right] { $\;y_2$}(-0.5*\len,1*\len) ;

\filldraw [black] (0.5*\len,0)  circle (1pt);    
\end{scope}
\end{tikzpicture}
\]
as an element in the vector space $\Hsp{x_1x_2x_3y_1y_2}$.
\end{enumerate}

\subsection{Spherical pairing}

Throughout this section we assume that $\cx$ is a spherical semisimple multitensor category.
The pivotal structure of $\cx$ induces an important structure on the Hom-spaces of $\cx$, namely a non-degenerate pairing: for any $x,y\in\Iso\cx_{i j}$
\begin{equation}\label{pairing}
\begin{aligned}
\langle-,-\rangle_{x,y}\Colon \Hom_\cx(\overline{x},\overline{y})&\otimes \Hom_\cx(x,y)\to \Hom_\cx(\un_j,\un_j)\cong\bk,\\% \quad 
(f,g)&\mapsto 
% \frac{1}{\sqrt{\rd_x\rd_y}}
\ev_{y}\circ (f\otimes g)\circ \tcoev_x=\tev_{y}\circ (g\otimes f)\circ \coev_x\,.
\end{aligned}
% \frac{1}{\sqrt{\rd_x\rd_y}}\tr(f\circ g)\,.
\end{equation}
We extend this to arbitrary $x, y \in \Iso \cx$ by bilinearity, noting that $\Hom_\cx(x, y)$ is the direct sum of spaces $\Hom_{\cx_{i j}}(x_{i j}, y_{i j})$ with $x = \bigoplus_{i, j} x_{i j}$ for $x_{i j} \in \cx_{i j}$, and a similar decomposition for $y$.
In particular, we have for the vector space $\Hom_\cx(\un,x)$ that the pairing
$\langle-,-\rangle_{\un,x}=\circ_x$ is the composition relative to $x$. 

A direct computation involving the sphericality of the pivotal structure of $\cx$ proves that the pairing~\eqref{pairing} commutes with the isomorphisms $\mathrm{X}_{i,i+1}$ from equation~\eqref{cyclic_perm}.
Consequently, when $x_1, \ldots, x_n$ is a cyclically composable sequence of $1$-cells,~\eqref{pairing} descends to a pairing on $\vSp$ spaces
    \begin{equation}\label{Hpairing}
    \langle-,-\rangle = \tr_{x_n}(-\circ_{x_1\cdots x_{n-1}}-)\Colon
    \Hsp{\overline{x}_n\cdots \overline{x}_1}
    \otimes\Hsp{x_1\cdots x_n}
    \to \bk
    \end{equation}
where we interpret the value of $\tr_{x_n}$ as a scalar by the isomorphism $\End_\cx(\un_i) \to \bk$ as in~\eqref{scalar_id}.
The non-degeneracy of the pairing means that $\Hsp{x_1\cdots x_n}^*\cong \Hsp{\overline{x}_n\cdots \overline{x}_1}$.

\begin{defi}
Let $\alpha=\{\alpha^i\}$ and $\overline{\alpha}=\{\overline{\alpha}_i\}$ be bases of $\Hsp{ax}$ and $\Hsp{\overline{x}\,\overline{a}}$ respectively. The bases $(\alpha,\overline{\alpha})$ are called
a \textit{pair of dual bases} of $\Hsp{ax}$  with respect to~\eqref{Hpairing} when $\langle\overline{\alpha}_i,\alpha^j\rangle=\delta_{ij}$ holds.
\end{defi}

Given a pair of dual bases $(\alpha,\overline{\alpha})$ of $\Hsp{ax}$, the linear map defined by
\begin{equation}\label{eq:dual_bases}
\begin{aligned}
\bk&\to \Hsp{\overline{x}\,\overline{a}}\otimes \Hsp{a\, x}\\
1&\mapsto
    \sum_i\;\overline{\alpha}_i\otimes \alpha^i
\end{aligned}
\end{equation}
is independent of the choice of dual bases. The element $\sum_i\overline{\alpha}_i\otimes \alpha^i$ will be depicted by
\begin{equation*}
\begin{tikzpicture}[x=0.75pt,y=0.75pt,yscale=-1.5,xscale=1.5,baseline={([yshift=-.5ex]current bounding box.center)}]

\draw  [line width=1]  (0,-30) to (0,30);
\draw  [line width=1]  (0,-30) to [in=225, out=135,distance=25] (0,30) ;

\draw  [line width=1]  (30,-30) to (30,30);
\draw  [line width=1]  (30,-30) to [in=-45, out=45,distance=25] (30,30) ;

\node[rotate=-90] at (0,0) {\scriptsize $\Arr$} ;
\node[rotate=90] at (-18,0) {\scriptsize $\Arr$} ;
\node[rotate=-90] at (30,0) {\scriptsize $\Arr$} ;
\node[rotate=90] at (47.5,0) {\scriptsize $\Arr$} ;
\node[black] at (-18,0) [anchor=east] {\small $x$} ;
\node[black] at (0,0) [anchor=west] {\small $a$} ;
\node[black] at (30,0) [anchor=east] {\small $a$} ;
\node[black] at (47.5,0) [anchor=west] {\small $x$} ;

\node at (-8,0) {\small $\overline{\alpha}_i$} ;

\node[black] at (15,0) {\small $\otimes$};

\node at (38,0){\small $\alpha^i$};

\filldraw [black] (0,30) circle (1pt);
\filldraw [black] (0,-30) circle (1pt);
\filldraw [black] (30,30) circle (1pt);
\filldraw [black] (30,-30) circle (1pt);
\end{tikzpicture}
\end{equation*}
where the sum over the basis elements is implicit and indicated by Einstein's summation convention.

\begin{lem}[{\cite[Lemma 3.6]{Kir}}]\label{lem:base_change}
Given $f\in\Hsp{ab}$, we have that
\begin{equation}\label{eq:base_change}
\begin{tikzpicture}[x=0.75pt,y=0.75pt,yscale=-1.5,xscale=1.5,baseline={([yshift=-.5ex]current bounding box.center)}]

\draw  [line width=1]  
(0,-30) to (0,30)
(0,-30) to [in=225, out=135,distance=25] (0,30) 
(0,-30) to [in=-45, out=45,distance=25] (0,30) 
(60,-30) to (60,30)
(60,-30) to [in=-45, out=45,distance=25] (60,30) ;

\node[rotate=-90] at (0,0) {\scriptsize $\Arr$} ;
\node[rotate=90] at (-18,0) {\scriptsize $\Arr$} ;
\node[rotate=90] at (18,0) {\scriptsize $\Arr$} ;
\node[rotate=-90] at (60,0) {\scriptsize $\Arr$} ;
\node[rotate=90] at (77.5,0) {\scriptsize $\Arr$} ;

\node[black] at (-20,0) [anchor=east] {\small $x$} ;
\node[black] at (-2,13) [anchor=west] {\small $a$} ;
\node[black] at (20,0) [anchor=west] {\small $b$} ;
\node[black] at (60,0) [anchor=east] {\small $a$} ;
\node[black] at (77.5,0) [anchor=west] {\small $x$} ;

\node at (-8,0)  {\small $\overline{\alpha}_i$} ;
\node at (8,0)  {\small $f$};

\node[black] at (40,0) {\small $\otimes$};

\node at (68,0) {\small $\alpha^i$};

\filldraw [black] (0,30) circle (1pt);
\filldraw [black] (0,-30) circle (1pt);
\filldraw [black] (60,30) circle (1pt);
\filldraw [black] (60,-30) circle (1pt);

\end{tikzpicture}
\;=\;
\begin{tikzpicture}[x=0.75pt,y=0.75pt,yscale=-1.5,xscale=1.5,baseline={([yshift=-.5ex]current bounding box.center)}]
\draw  [line width=1]  (0,-30) to (0,30)
(0,-30) to [in=225, out=135,distance=25] (0,30) 
(60,-30) to [in=225, out=135,distance=25] (60,30) 
(60,-30) to (60,30)
(60,-30) to [in=-45, out=45,distance=25] (60,30) ;

\node[rotate=90] at (0,0) {\scriptsize $\Arr$} ;
\node[rotate=90] at (-18,0) {\scriptsize $\Arr$} ;
\node[rotate=-90] at (42,0) {\scriptsize $\Arr$} ;
\node[rotate=90] at (60,0) {\scriptsize $\Arr$} ;
\node[rotate=90] at (77.5,0) {\scriptsize $\Arr$} ;

\node[black] at (-20,0) [anchor=east] {\small $x$} ;
\node[black] at (0,0) [anchor=west] {\small $b$} ;
\node[black] at (42,0) [anchor=east] {\small $a$} ;
\node[black] at (58,13) [anchor=west] {\small $b$} ;
\node[black] at (77.5,0) [anchor=west] {\small $x$} ;

\node at (-8,0)  {\small $\overline{\beta}_j$} ;
\node at (52,0)  {\small $f$};
\node[black] at (20,0) {\small $\otimes$};
\node at (68,0) {\small $\beta^j$};

\filldraw [black] (0,30) circle (1pt);
\filldraw [black] (0,-30) circle (1pt);
\filldraw [black] (60,30) circle (1pt);
\filldraw [black] (60,-30) circle (1pt);
\end{tikzpicture}
\end{equation}
where $(\alpha,\overline{\alpha})$ is a pair of dual bases of $\Hsp{ax}$ and $(\beta,\overline{\beta})$ dual bases of $\Hsp{\overline{b}x}$.
\end{lem}
Composition relative to an object is compatible with the pairing only up to a dimension factor. For a simple object $x\in\Irr\,\cx_{i j}$, define 
\begin{equation*}%\label{rcomp_x}
f\bullet_x g\coloneqq \sqrt{\rd_x}\;f\circ_x g\,.
\end{equation*}
For composable simple objects $x_1,\ldots,x_n$
\begin{equation*}
f\bullet_{x_1\cdots x_n} g\coloneqq \sqrt{\rd_{x_1}\cdots\rd_{x_n}}\;f\circ_{x_1\cdots x_n} g\,.
\end{equation*}

\begin{lem}[{\cite[Lemma 1.1]{BK1}}]\label{lem:pairing_properties}
Let $y\in\Irr\cx$ be a simple object, and $a,b \in\Iso\cx$ homogeneous objects.
\begin{enumerate}[\rm (i)]
    \item For $f\in{\Hsp{ay}}$ and $g\in{\Hsp{\overline{y}\,\overline{a}}}$, then 
    \begin{equation*}
        \rd_{y}\,f \circ_a g=\langle f,g\rangle \Id_y\;.
    \end{equation*}
    \item Dominance relation: unit morphisms can be decomposed in terms of sums of dual bases over simple objects.
\def\len{1.25cm}
\begin{equation}\label{dominance}
\begin{tikzpicture}[baseline={([yshift=-.5ex]current bounding box.center)}]
    \draw[line width=1] (-0.75*\len,0) -- (0,\len) node [midway, rotate=45] {\scriptsize $\Arr$}
    (-0.75*\len,0) -- (0,-\len) node [midway, rotate=135] {\scriptsize $\Arr$}
    (0.75*\len,0) -- (0,-\len) node [midway, rotate=45] {\scriptsize $\Arr$}
    (0.75*\len,0) -- (0,\len) node [midway, rotate=135] {\scriptsize $\Arr$};

\filldraw [black] (0,-\len) circle (1pt);
\filldraw [black] (0,\len) circle (1pt);
\filldraw [black] (0.75*\len,0) circle (1pt);
\filldraw [black] (-0.75*\len,0) circle (1pt);
 
\node at ($(-0.75*\len,0)!0.5!(0,\len)$) [anchor=south east]{$b$};
\node at ($(-0.75*\len,0)!0.5!(0,-\len)$) [anchor=north east]{$a$};
\node at ($(0.75*\len,0)!0.5!(0,\len)$) [anchor=south west]{$b$};
\node at ($(0.75*\len,0)!0.5!(0,-\len)$) [anchor=north west]{$a$};
\node  at (0,0) 
{$\Id_{ab}$} ;
\end{tikzpicture}
\;=\;\sum_{t\in \Irr \cx}\;\rd_t\,
\begin{tikzpicture}[baseline={([yshift=-.5ex]current bounding box.center)}]
    \draw[line width=1] (-0.75*\len,0) -- (0,\len) node [midway, rotate=45] {\scriptsize $\Arr$}
    (-0.75*\len,0) -- (0,-\len) node [midway, rotate=135] {\scriptsize $\Arr$}
    (0.75*\len,0) -- (0,-\len) node [midway, rotate=45] {\scriptsize $\Arr$}
    (0.75*\len,0) -- (0,\len) node [midway, rotate=135] {\scriptsize $\Arr$}
    (0,\len) -- (0,-\len) node [midway, rotate=90] {\scriptsize $\Arr$};

\filldraw [black] (0,-\len) circle (1pt);
\filldraw [black] (0,\len) circle (1pt);
\filldraw [black] (0.75*\len,0) circle (1pt);
\filldraw [black] (-0.75*\len,0) circle (1pt);
 
\node at ($(-0.75*\len,0)!0.5!(0,\len)$) [anchor=south east]{$b$};
\node at ($(-0.75*\len,0)!0.5!(0,-\len)$) [anchor=north east]{$a$};
\node at ($(0.75*\len,0)!0.5!(0,\len)$) [anchor=south west]{$b$};
\node at ($(0.75*\len,0)!0.5!(0,-\len)$) [anchor=north west]{$a$};
\node  at (-0.35*\len,0) 
{\small $\overline{\alpha}_{t,i}$} ;
\node  at (0.35*\len,0) 
{\small $\alpha_t^i$} ;
\node at (0.1*\len,0) [anchor=south] {\small $t$};
\end{tikzpicture}
\end{equation}
    % \item 
    \item For $f\in{\Hsp{ay}},\;g\in{\Hsp{\overline{y}\,b}},\;h\in{\Hsp{\overline{b}\,y}},~\text{and}~ k\in{\Hsp{\overline{y}\,\overline{a}}}$
    \begin{equation*}
    \langle f \bullet_y g, h\bullet_y k\rangle=\langle f ,k\rangle\cdot \langle g,h\rangle\;.
    \end{equation*}
\end{enumerate}
\end{lem}
\begin{proof}
Since $y$ is simple, $\Hsp{y\overline{y}}$ is one dimensional and thus $f\circ_a g=\lambda\,\Id_y$ for a scalar $\lambda\in\bk$. It follows that $\langle f,g\rangle=\tr_y(f\circ_a g)=\lambda\, \rd_y$ which implies (i). To prove (ii) first assume that $b=\un$ and $a=y$ is simple. Notice that, from Schur's lemma, $\Hsp{t\overline{y}}$ is one-dimensional only if $t\cong y$ and zero otherwise. It follows that the sum on the left hand side of~\eqref{dominance} reduces to
    \begin{equation*}
        \sum_{t\in \Irr \cx}\,\rd_t\,\overline{\alpha}_{t,i}\circ_t\alpha_t^i=\rd_y\,\overline{\alpha}_{y}\circ_t\alpha_y=\langle\overline{\alpha}_{y},\alpha_y\rangle\,\Id_y=\Id_y
    \end{equation*}
    where the second to last equality follows from (i) and the last one holds since $\overline{\alpha}_y$ and $\alpha_y$ are dual bases. Due to semisimplicty of $\cx$, extending the argument to direct sums ensures that the whole statement holds. Lastly, notice that
    \begin{equation*}
        \langle f\bullet_y g,h\bullet_y k\rangle=\rd_y\,\tr_a(f\circ_y g\circ_b h\circ_y k)\overset{\rm (i)}{=}\rd_y\,\frac{\langle g,h\rangle}{\rd_y}\tr_a(f\circ_y\Id_y\circ_y k)=\langle f ,k\rangle\cdot \langle g,h\rangle
    \end{equation*}
    which proves (iii).
\end{proof}

\begin{lem}\label{bases_comp}
Let $a,b,c$ be homogeneous objects in $\Iso\cx$. 
\begin{enumerate}[$($i$)$]
    \item For each $t\in\Irr\cx$ consider a pair of dual bases  $(\psi_t,\overline{\psi}_t)$ of $\Hsp{atc}$, and $(\beta_t,\overline{\beta}_t)$ a pair of bases of $\Hsp{b\overline{t}}$. Then 
    $\{\beta_t^i\bullet_t\psi_t^j\}_{t\in\Irr\cx}$ and $\{\overline{\beta}_{t,i}\bullet_t\overline{\psi}_{t,j}\}_{t\in\Irr\cx}$
    is a pair of dual bases of $\Hsp{abc}$. 
    \item Given a pair of dual bases  $(\varphi_t,\overline{\varphi}_t)$ of $\Hsp{btc}$, of each $t\in\Irr\cx$, and $(\alpha_t,\overline{\alpha}_t)$ of $\Hsp{a\overline{t}}$. Then $\{\alpha_t^i\bullet_t\varphi_t^j\}_{t\in\Irr\cx}$ and $\{\overline{\alpha}_{t,i}\bullet_t\overline{\varphi}_{t,j}\}_{t\in\Irr\cx}$ is a pair of dual bases of $\Hsp{abc}$. 
\end{enumerate}
Thus, graphically we have in the vector space $\Hsp{abc}\otimes \Hsp{\overline{abc}}$ that
\begin{equation*}
\sum_{t\in\Irr\cx}\quad
\begin{tikzpicture}[x=0.75pt,y=0.75pt,yscale=-1.5,xscale=1.5,baseline={([yshift=-.5ex]current bounding box.center)}]

\draw  [line width=1]  (400.3,270.5) -- (350.1,219.7) -- (350.1,270.5) -- cycle ;
\draw [line width=1]    (350.1,270.5) .. controls (360.1,250.1) and (381.2,249.8) .. (400.3,270.5) ;
\draw  [line width=1]  (420.6,250.2) -- (369.8,200) -- (420.6,200) -- cycle ;
\draw [line width=1]    (420.6,200) .. controls (400.2,210) and (399.9,231.1) .. (420.6,250.2) ;
\node[rotate=180] at (372.59,270.66) {\scriptsize $\Arr$} ;
\node[rotate=90] at (349.77,246.85) {\scriptsize $\Arr$} ;
\node[rotate=-45] at (375.01,245.37) {\scriptsize $\Arr$} ;
\node[rotate=-45] at (395.01,225.37) {\scriptsize $\Arr$} ;
\node[rotate=180] at (396.59,200.5) {\scriptsize $\Arr$} ;
\node[rotate=90] at (420.27,221.85) {\scriptsize $\Arr$} ;

\draw (335,239) node [anchor=north west]   {$b$};
\draw (389,185) node [anchor=north west]    {$b$};
\draw (422,220) node [anchor=north west]    {$a$};
\draw (370,275) node [anchor=north west]    {$a$};
\draw (370,230) node [anchor=north west]    {$c$};
\draw (383,233) node [anchor=north west]    {$\otimes$};
\draw (380,220) node [anchor=north west]     {$c$};
\draw (355,252) node [anchor=north west]   {$\bullet_{t}$};
\draw (402,206) node [anchor=north west]    {$\bullet_{t}$};
\draw (350,235) node [anchor=north west]   [font=\small ]  {$\overline{\varphi}_{t,j}$};
\draw (385,202) node [anchor=north west]   [font=\small ]  {$\varphi_t^j$};
\draw (370,255) node [anchor=north west]   [font=\small]  {$\overline{\alpha}_{t,i}$};
\draw (405,219) node [anchor=north west]   [font=\small]  {$\alpha_t^i$};

\filldraw [black] (400.3,270.5) circle (1pt);
\filldraw [black] (350.1,270.5) circle (1pt);
\filldraw [black] (350.1,219.7) circle (1pt);
\filldraw [black]  (369.8,200) circle (1pt);
\filldraw [black] (420.6,250.2)  circle (1pt);
\filldraw [black] (420.6,200) circle (1pt);
\end{tikzpicture}
\;=\sum_{t\in\Irr\cx}\quad
\begin{tikzpicture}[x=0.75pt,y=0.75pt,yscale=-1.5,xscale=1.5,baseline={([yshift=-.5ex]current bounding box.center)}]

\draw  [line width=1]  (371.25,80.57) -- (421.45,131.37) -- (421.45,80.57) -- cycle ;
\draw [line width=1]    (421.45,80.57) .. controls (411.45,100.97) and (390.35,101.27) .. (371.25,80.57) ;
\draw  [line width=1]  (350.95,100.87) -- (401.75,151.07) -- (350.95,151.07) -- cycle ;
\draw [line width=1]    (350.95,151.07) .. controls (371.35,141.07) and (371.65,119.97) .. (350.95,100.87) ;
\node[rotate=180] at (398.59,80.91) {\scriptsize $\Arr$} ;
\node[rotate=90] at (421.77,105.85) {\scriptsize $\Arr$} ;
\node[rotate=-45] at (400.01,109.37) {\scriptsize $\Arr$} ;
\node[rotate=135] at (380.01,129.37) {\scriptsize $\Arr$} ;
\node[rotate=180] at (374.59,151.04) {\scriptsize $\Arr$} ;
\node[rotate=90] at (350.77,127.85) {\scriptsize $\Arr$} ;

\draw (337,122) node [anchor=north west]    {$b$};
\draw (390,67) node [anchor=north west]    {$b$};
\draw (425,100) node [anchor=north west]    {$a$};
\draw (370,155) node [anchor=north west]   {$a$};
\draw (370,110) node [anchor=north west]    {$c$};
\draw (385,115) node [anchor=north west]    {$\otimes$};
\draw (380,100) node [anchor=north west]    {$c$};
\draw (357,135) node [anchor=north west]   {$\bullet_{t}$};
\draw (404,88) node [anchor=north west]    {$\bullet_{t}$};
\draw (403,100) node [anchor=north west]  [font=\small ]  {$\psi_t^j$};
\draw (374,133) node [anchor=north west]  [font=\small ]  {$\overline{\psi}_{t,j}$};
\draw (350,117) node [anchor=north west]  [font=\small]  {$\overline{\beta}_{\!t\!,i}$};
\draw (388,80) node [anchor=north west]  [font=\small]  {$\beta_t^i$};

\filldraw [black]  (421.45,80.57) circle (1pt);
\filldraw [black]  (421.45,131.37) circle (1pt);
\filldraw [black]  (371.25,80.57) circle (1pt);
\filldraw [black]  (401.75,151.07) circle (1pt);
\filldraw [black]  (350.95,100.87)  circle (1pt);
\filldraw [black]  (350.95,151.07) circle (1pt);
\end{tikzpicture}
\,=\,
\begin{tikzpicture}[x=0.75pt,y=0.75pt,yscale=-1.5,xscale=1.5,baseline={([yshift=-.5ex]current bounding box.center)}]

\draw  [line width=1]  (571.3,231.37) -- (521.1,180.57) -- (521.1,231.37) -- cycle ;
\draw  [line width=1]  (590.6,211.07) -- (539.8,160.87) -- (590.6,160.87) -- cycle ;
\node[rotate=180] at (542,232) {\scriptsize $\Arr$} ;
\node[rotate=90] at (521,208) {\scriptsize $\Arr$} ;
\node[rotate=-45] at (546,206) {\scriptsize $\Arr$} ;
\node[rotate=-45] at (565,186) {\scriptsize $\Arr$} ;
\node[rotate=180] at (566,161) {\scriptsize $\Arr$} ;
\node[rotate=90] at (591,182) {\scriptsize $\Arr$} ;

\draw (505,205) node [anchor=north west]  {$b$};
\draw (562,145) node [anchor=north west]  {$b$};
\draw (594,178) node [anchor=north west]  {$a$};
\draw (540,235) node [anchor=north west]  {$a$};
\draw (540,190) node [anchor=north west]  {$c$};
\draw (553,195) node [anchor=north west]  {$\otimes$};
\draw (555,185) node [anchor=north west]  {$c$};
\draw (525,205) node [anchor=north west]  [font=\small]  {$\overline{\gamma}_k$};
\draw (570,170) node [anchor=north west]  [font=\small]  {$\gamma^k$};

\filldraw [black]  (521,231) circle (1pt);
\filldraw [black]  (521,181) circle (1pt);
\filldraw [black]  (571,231) circle (1pt);
\filldraw [black]   (590,211) circle (1pt);
\filldraw [black]   (540,161)  circle (1pt);
\filldraw [black]   (590,161) circle (1pt);
\end{tikzpicture}
\end{equation*}
for an arbitrary pair of dual bases $(\gamma,\overline{\gamma})$ of $\Hsp{abc}$.
\end{lem}
\begin{proof}
From Lemma~\ref{lem:pairing_properties} (iii) follows that
\begin{equation*}
    \langle\, 
    \overline{\beta}_{t,i}\bullet_t\overline{\psi}_{t,j},
    \beta_s^k\bullet_s\psi_s^l
    \,\rangle
    =\langle\overline{\beta}_{t,i}, \beta_s^k\rangle\cdot \langle\overline{\psi}_{t,j}, \psi_s^l\rangle
    =\delta_{t,s}\delta_{i,k}\delta_{j,l}
\end{equation*}
where the $\delta_{t,s}$ appears after taking into account that $\overline{\beta}_{t,i}\circ_b\beta_s^k$ belongs to $\Hsp{t\overline{s}}\cong \delta_{t,s}\bk$. The additional delta factors come since we are pairing pairs of dual bases. This proves (i). Statement (ii) follows analogously.
\end{proof}

In the case that a graph has two edges with labels dual to each other, we can resolve them over simple objects, via the composition:
\begin{equation*}%\label{simple_resolved}
\star_x \Colon   \Hsp{axb\overline{x}c}
\xrightarrow{\eqref{eq:dual_bases}}
\bigoplus_{t\in \Irr\cx} \Hsp{axb\overline{x}c}\otimes\Hsp{x\overline{t}}\otimes \Hsp{\overline{x}t}\xrightarrow{\bullet_x\;\otimes\;\bullet_{\,\overline{x}}}
\bigoplus_{t\in \Irr \cx}\Hsp{atb\overline{t}c}
\end{equation*}
This operation can be graphically represented as follows. For $f\in \Hsp{axb\overline{x}c}$, the homogeneous component of $\star_x(f)$ associated to $t\in \Irr \cx$ is given by
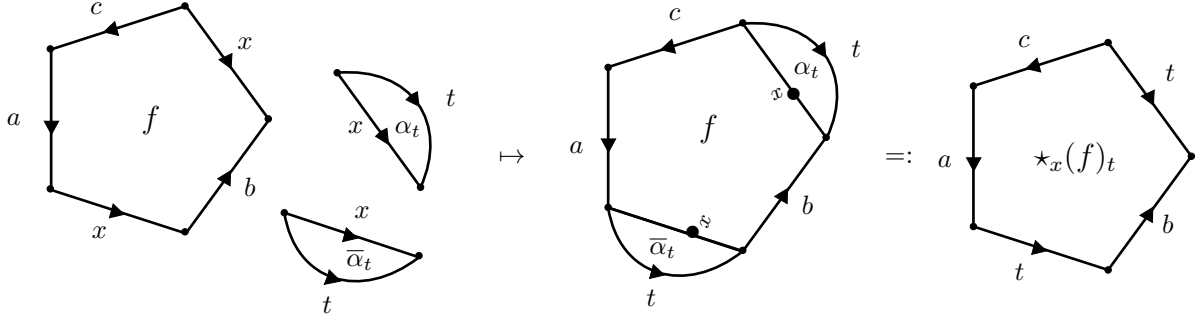
\begin{equation*}
\begin{tikzpicture}[x=1pt,y=1pt,yscale=-0.6,xscale=0.6,baseline={([yshift=-.5ex]current bounding box.center)}]
\draw [line width=1]  (177.5,159.5) -- (125.56,230.99) -- (41.52,203.68) -- (41.52,115.32) -- (125.56,88.01) -- cycle ;
\filldraw[black] (46.15,160.48) -- (41.71,168.73) -- (37.27,160.48) -- (41.71,160.48) -- cycle ;
\filldraw[black] (80.21,212.58) -- (87.03,219.01) -- (77.91,221.16) -- (79.06,216.87) -- cycle ;
\filldraw[black] (83.51,106.49) -- (74.39,104.34) -- (81.21,97.92) -- (82.36,102.2) -- cycle ;
\filldraw[black] (154.32,117.99) -- (153.78,127.34) -- (146.28,121.74) -- (150.3,119.87) -- cycle ;
\filldraw[black] (145.47,196.62) -- (153.44,191.7) -- (153.16,201.06) -- (149.31,198.84) -- cycle ;
\filldraw[black] (225.71,227.58) -- (232.53,234.01) -- (223.41,236.16) -- (224.56,231.87) -- cycle ;

\draw [line width=1] (187,218) -- (271,246) ;
\draw [line width=1] (220,130) -- (272,202) ;
\filldraw[black] (251,165) -- (252,175) -- (244,170) -- (247,167) -- cycle;
\draw [line width=1]    (187,218) .. controls (194,262) and (235,274) .. (271.06,245.99) ;
\filldraw[black] (215,254) -- (222,261) -- (212,263) -- (214,258) -- cycle;
\draw [line width=1] (272,202) .. controls (290,162) and (267,126) .. (220,130);
\filldraw[black] (270,142) -- (271,151) -- (263,147) -- (266,144) -- cycle;

\draw (92,145) node [anchor=north west] {\large${f}$};
\draw (8,150) node [anchor=north west]   {$a$};
\draw (61,221) node [anchor=north west] {$x$};
\draw (60,80) node [anchor=north west]   {$c$};
\draw (152,102) node [anchor=north west]     {$x$};
\draw (156,190) node [anchor=north west]     {$b$};
\draw (225,208) node [anchor=north west]     {$x$};
\draw (221,155) node [anchor=north west]     {$x$};
\draw (282,135) node [anchor=north west]     {$t$};
\draw (205,265) node [anchor=north west]     {$t$};
\draw (250,155) node [anchor=north west] {$ {\alpha_t}$};
\draw (220,235) node [anchor=north west] {$ {\overline{\alpha}_t}$};

\filldraw [black] (177,159)  circle (2pt);
\filldraw [black] (125,230)  circle (2pt);
\filldraw [black] (41,203)  circle (2pt);
\filldraw [black] (41,115) circle (2pt);
\filldraw [black] (125,88) circle (2pt);
\filldraw [black] (187,218) circle (2pt);
\filldraw [black] (271,245)  circle (2pt);
\filldraw [black] (220,130)  circle (2pt);
\filldraw [black] (272,202) circle (2pt);
\end{tikzpicture}
 \quad\mapsto\quad
\begin{tikzpicture}[x=1pt,y=1pt,yscale=-0.6,xscale=0.6,baseline={([yshift=-.5ex]current bounding box.center)}]
\draw  [line width=1]  (526,103) -- (474.06,174.49) -- (390.02,147.18) -- (390.02,58.82) -- (474.06,31.51) -- cycle ;
\filldraw[black] (394.65,103.98) -- (390.21,112.23) -- (385.77,103.98) -- (390.21,103.98) -- cycle ;
\filldraw[black] (432.01,49.99) -- (422.89,47.84) -- (429.71,41.42) -- (430.86,45.7) -- cycle ;
\filldraw[black] (493.97,140.12) -- (501.94,135.2) -- (501.66,144.56) -- (497.81,142.34) -- cycle;
\draw [line width=1]   (390.02,147.18) -- (473.06,174.49);
\draw [line width=1]    (390.02,147.18) .. controls (397,191) and (438.5,203) .. (474.06,174.49);
\filldraw[black] (418.21,183.08) -- (425.03,189.51) -- (415.91,191.66) -- (417.06,187.37) -- cycle;
\draw [line width=1]    (526.5,103) .. controls (545,63.5) and (521.56,27.51) .. (474.56,31.51);
\filldraw[black] (525.04,43.23) -- (525.78,52.57) -- (517.58,48.04) -- (521.31,45.64) -- cycle;

\draw (441,83) node [anchor=north west] [font=\large ]  {$ {f}$};
\draw (360,100) node [anchor=north west]   {$a$};
\draw (422,35) node [anchor=south west]   {$c$};
\draw (505,134) node [anchor=north west]   {$b$};
\draw (536,36) node [anchor=north west]  {$t$};
\draw (408,193) node [anchor=north west] {$t$};
\node[rotate=70] at (435,155) [anchor=north] {\large $\bullet_x$};
\node[rotate=210] at (495,66) [anchor=south] {\large $\bullet_x$};
\draw (500,50) node [anchor=north west] {$ {\alpha_t}$};
\draw (410,160) node [anchor=north west]   {$ {\overline{\alpha}_t}$};

\filldraw [black] (526,103)  circle (2pt);
\filldraw [black] (474,174) circle (2pt);
\filldraw [black] (390,147) circle (2pt);
\filldraw [black] (390,59) circle (2pt);
\filldraw [black] (474,31) circle (2pt);
\end{tikzpicture}
\;=:\;
\begin{tikzpicture}[x=1pt,y=1pt,yscale=-0.6,xscale=0.6,baseline={([yshift=-.5ex]current bounding box.center)}]

\draw  [line width=1]  (224,113.5) -- (172.06,184.99) -- (88.02,157.68) -- (88.02,69.32) -- (172.06,42.01) -- cycle ;
\filldraw[black] (92.65,114.48) -- (88.21,122.73) -- (83.77,114.48) -- (88.21,114.48) -- cycle ;
\filldraw[black] (126.71,166.58) -- (133.53,173.01) -- (124.41,175.16) -- (125.56,170.87) -- cycle ;
\filldraw[black] (130.01,60.49) -- (120.89,58.34) -- (127.71,51.92) -- (128.86,56.2) -- cycle ;
\filldraw[black] (200.82,71.99) -- (200.28,81.34) -- (192.78,75.74) -- (196.8,73.87) -- cycle ;
\filldraw[black] (191.97,150.62) -- (199.94,145.7) -- (199.66,155.06) -- (195.81,152.84) -- cycle ;

\filldraw [black] (224,113.5) circle (2pt);
\filldraw [black] (172.06,184.99) circle (2pt);
\filldraw [black] (88.02,157.68) circle (2pt);
\filldraw [black] (88.02,69.32) circle (2pt);
\filldraw [black] (172.06,42.01) circle (2pt);

\draw (120,99) node [anchor=north west]  [font=\large]  {$ {\star_x(f)_t}$};
\draw (60,107) node [anchor=north west]    {$a$};
\draw (108,175) node [anchor=north west]   {$t$};
\draw (110,32) node [anchor=north west]    {$c$};
\draw (202,50) node [anchor=north west]    {$t$};
\draw (200,144) node [anchor=north west]   {$b$};
\end{tikzpicture}
\end{equation*}
where $(\alpha_t,\overline{\alpha}_t)$ is a pair of dual bases of $\Hsp{x\overline{t}}$.

\begin{lem}\label{star_prop}
The maps $\star_x$ satisfy the following consistency conditions.
\begin{enumerate}[$($i$)$]
    \item Given $f\in\Hsp{\overline{y}xyz}$ and $g\in\Hsp{\overline{z}\,w}$, it holds in $\bigoplus_{t\in \Irr \cx}\Hsp{xtw\overline{t}}$ that 
    \begin{equation}\label{distr_star}
        \star_y (f\circ_z g)=\star_y(f)\circ_z g,\quad\text{ and }\quad 
        \star_y(f\bullet_z g)=\star_y(f)\bullet_z g\,.
    \end{equation}
    \item Given $f\in\Hsp{axyzb\overline{z}\,\overline{y}\,\overline{x}}$, it holds that
    \begin{equation}\label{star_triprod}
    \sum_{t\in \Irr \cx}\,\star_{xt}(\star_{yz}(f)_t)\,=\,\star_{xyz}(f)
    \,=\,\sum_{t\in \Irr \cx}\,\star_{tz}(\star_{xy}(f)_t)
    \end{equation}    
    \item Given $f\in\Hsp{xyzw\overline{z}\,\overline{y}}$, it holds  that 
    \begin{equation*}
        \star_z(\star_y (f))=\star_y(\star_z(f))\,,
    \end{equation*}
    \begin{equation*}%\label{star_biprod}
    \star_{yz}(f)=\sum_{s,t\in \Irr \cx}\, \star_{st}(\star_z(\star_y (f)_s)_t)
    =\sum_{s,t\in \Irr \cx}\, \star_{st}(\star_y(\star_z(f)_t)_s)\,.
    \end{equation*}
    \item For $f\in\Hsp{ax\,\overline{b}\,\overline{x}}$ and $x\in\Irr\cx$, we have that $\star_x(f)=f$.
    \item Given $f\in\Hsp{\overline{z}\,\overline{y}\,\overline{x}}$, with $z\in\Irr\cx$, and families of pairs of dual bases $\alpha$, $\beta$, $\gamma$ and $\eta$ it holds
\def\len{1.4cm}
\begin{equation*}
\sum_{t\in\Irr\cx} \rd_t\, \star_{xy}\left[\!\!
\begin{tikzpicture}[xscale=0.8,baseline={([yshift=-.5ex]current bounding box.center)}]
\draw[line width=1] 
    (-0.5*\len,0) to (0.5*\len,0)
    (-0.5*\len,0) to (-0.5*\len,0.5*\len)
    (0.5*\len,0) to (0.5*\len,0.5*\len)
    (-0.5*\len,\len) to (-0.5*\len,0.5*\len)
    (0.5*\len,\len) to (0.5*\len,0.5*\len)
    (-0.5*\len,\len) to (0.5*\len,\len)
    (-0.5*\len,0.5*\len) to (0.5*\len,0.5*\len);

\node[rotate=-90] at (-0.5*\len,0.25*\len) {\scriptsize $\Arr$};
\node[rotate=-90] at (0.5*\len,0.25*\len) {\scriptsize $\Arr$};
\node[rotate=-90] at (-0.5*\len,0.75*\len) {\scriptsize $\Arr$};
\node[rotate=-90] at (0.5*\len,0.75*\len) {\scriptsize $\Arr$};
\node[rotate=180] at (0,0) {\scriptsize $\Arr$};
\node[rotate=180] at (0,\len) {\scriptsize $\Arr$};
\node[rotate=180] at (0,0.5*\len) {\scriptsize $\Arr$};

\filldraw [black] (-0.5*\len,0) circle (1pt);
\filldraw [black] (0.5*\len,0) circle (1pt);
\filldraw [black] (-0.5*\len,\len) circle (1pt);
\filldraw [black] (0.5*\len,\len) circle (1pt);
\filldraw [black] (-0.5*\len,0.5*\len) circle (1pt);
\filldraw [black] (0.5*\len,0.5*\len) circle (1pt);

% Text Node
\node at (-0.5*\len,0.25*\len) [anchor=east] {$y$};
\node at (0.5*\len,0.25*\len) [anchor=west] {$y$};
\node at (-0.5*\len,0.75*\len) [anchor=east] {$x$};
\node at (0.5*\len,0.75*\len) [anchor=west] {$x$};
\node at (0,0) [anchor=north] {$b$};
\node at (-0.2*\len,0.45*\len) [anchor=south,font=\scriptsize] {$t$};
\node at (0,\len) [anchor=south] {$a$};

\draw (0.1*\len,0.75*\len) node    {${\overline{\alpha}_i}$};
\draw (0.1*\len,0.25*\len) node   [font=\normalsize]  {${\overline{\beta}_j}$};
\end{tikzpicture}
\right]\otimes\!\!\!
\begin{tikzpicture}[xscale=0.9,baseline={([yshift=-.5ex]current bounding box.center)}]
    \draw[line width=1]
(0.5*\len,0) .. controls (1.2*\len,0.25*\len) and (1.2*\len,0.75*\len).. (0.5*\len,\len);
\draw[line width=1] 
    (-0.5*\len,0) to (0.5*\len,0)
    (-0.5*\len,0) to (-0.5*\len,0.5*\len)
    (0.5*\len,0) to (0.5*\len,0.5*\len)
    (-0.5*\len,\len) to (-0.5*\len,0.5*\len)
    (0.5*\len,\len) to (0.5*\len,0.5*\len)
    (-0.5*\len,\len) to (0.5*\len,\len)
    (-0.5*\len,0.5*\len) to (0.5*\len,0.5*\len);

\node[rotate=90] at (-0.5*\len,0.25*\len) {\scriptsize $\Arr$};
\node[rotate=90] at (0.5*\len,0.25*\len) {\scriptsize $\Arr$};
\node[rotate=90] at (-0.5*\len,0.75*\len) {\scriptsize $\Arr$};
\node[rotate=90] at (0.5*\len,0.75*\len) {\scriptsize $\Arr$};
\node[rotate=-90] at (1.025*\len,0.5*\len) {\scriptsize $\Arr$};
\node[rotate=180] at (0,0) {\scriptsize $\Arr$};
\node[rotate=180] at (0,\len) {\scriptsize $\Arr$};
\node[rotate=180] at (0,0.5*\len) {\scriptsize $\Arr$};

\filldraw [black] (-0.5*\len,0) circle (1pt);
\filldraw [black] (0.5*\len,0) circle (1pt);
\filldraw [black] (-0.5*\len,\len) circle (1pt);
\filldraw [black] (0.5*\len,\len) circle (1pt);
\filldraw [black] (-0.5*\len,0.5*\len) circle (1pt);
\filldraw [black] (0.5*\len,0.5*\len) circle (1pt);

% Text Node
\node at (1.0*\len,0.5*\len) [anchor=west] {$z$};
\node at (-0.5*\len,0.25*\len) [anchor=east] {$x$};
\node at (0.5*\len,0.25*\len) [anchor=west,font=\footnotesize] {$x$};
\node at (-0.5*\len,0.75*\len) [anchor=east] {$y$};
\node at (0.5*\len,0.75*\len) [anchor=west,font=\footnotesize] {$y$};
\node at (0,0) [anchor=north] {$a$};
\node at (-0.2*\len,0.45*\len) [anchor=south,font=\scriptsize] {$t$};
\node at (0,\len) [anchor=south] {$b$};

\draw (0.8*\len,0.5*\len) node  {$f$};
\draw (0.1*\len,0.25*\len) node  {${\alpha^i}$};
\draw (0.1*\len,0.75*\len) node  {${\beta^j}$};
\end{tikzpicture}
=\frac{\rd_x\rd_y}{\rd_z}
\langle \gamma^k,f\rangle
\begin{tikzpicture}[xscale=0.7,baseline={([yshift=-.5ex]current bounding box.center)}]
\draw[line width=1] 
    (-0.5*\len,0) to (0.5*\len,0)
    (-0.5*\len,0) to (-0.5*\len,0.5*\len)
    (0.5*\len,0) to (0.5*\len,0.5*\len)
    (-0.5*\len,0.5*\len) to (0.5*\len,0.5*\len);

\node[rotate=-90] at (-0.5*\len,0.25*\len) {\scriptsize $\Arr$};
\node[rotate=-90] at (0.5*\len,0.25*\len) {\scriptsize $\Arr$};
\node[rotate=180] at (0,0) {\scriptsize $\Arr$};
\node[rotate=180] at (0,0.5*\len) {\scriptsize $\Arr$};

\filldraw [black] (-0.5*\len,0) circle (1pt);
\filldraw [black] (0.5*\len,0) circle (1pt);
\filldraw [black] (-0.5*\len,0.5*\len) circle (1pt);
\filldraw [black] (0.5*\len,0.5*\len) circle (1pt);

% Text Node
\node at (-0.5*\len,0.25*\len) [anchor=east] {$z$};
\node at (0.5*\len,0.25*\len) [anchor=west] {$z$};
\node at (0,0) [anchor=north] {$b$};
\node at (0,0.5*\len) [anchor=south] {$a$};

\draw (0,0.25*\len) node    {${\overline{\eta}_i}$};
\end{tikzpicture}
\otimes 
\begin{tikzpicture}[yscale=1.5,xscale=0.7,baseline={([yshift=-.5ex]current bounding box.center)}]
\draw[line width=1]
(-0.5*\len,0) .. controls (-1.35*\len,0.15*\len) and (-1.35*\len,0.35*\len).. (-0.5*\len,0.5*\len);
\draw[line width=1] 
    (-0.5*\len,0) to (0.5*\len,0)
    (-0.5*\len,0) to (-0.5*\len,0.5*\len)
    (0.5*\len,0) to (0.5*\len,0.5*\len)
    (-0.5*\len,0.5*\len) to (0.5*\len,0.5*\len);
\node[rotate=90] at (-0.5*\len,0.25*\len) {\scriptsize $\Arr$};
\node[rotate=90] at (0.5*\len,0.25*\len) {\scriptsize $\Arr$};
\node[rotate=90] at (-1.15*\len,0.25*\len) {\scriptsize $\Arr$};
\node[rotate=180] at (0,0) {\scriptsize $\Arr$};
\node[rotate=180] at (0,0.5*\len) {\scriptsize $\Arr$};
\filldraw [black] (-0.5*\len,0) circle (1pt);
\filldraw [black] (0.5*\len,0) circle (1pt);
\filldraw [black] (-0.5*\len,0.5*\len) circle (1pt);
\filldraw [black] (0.5*\len,0.5*\len) circle (1pt);

% Text Node
\node at (-1.15*\len,0.25*\len) [anchor=east] {$xy$};
\node at (-0.52*\len,0.25*\len) [anchor=west,font=\footnotesize] {$z$};
\node at (0.5*\len,0.25*\len) [anchor=west] {$z$.};
\node at (0,0) [anchor=north] {$a$};
\node at (0,0.5*\len) [anchor=south] {$b$};
\draw (-0.8*\len,0.25*\len) node    {${\overline{\gamma}_k}$};
\draw (0.1*\len,0.25*\len) node    {${\eta^i}$};
\end{tikzpicture}
\end{equation*}
\end{enumerate}
\end{lem}
\begin{proof}
The first statement follows from Lemma~\ref{comp_prop}, indeed
\begin{equation*}
    \star_y (f\circ_z g)=\sum_{t\in \Irr \cx}\alpha_t\circ_b (f\circ_z g)\circ_{\overline{b}}\overline{\alpha}_t=\sum_{t\in \Irr \cx}\alpha_t\circ_b (f\circ_{\overline{b}}\overline{\alpha}_t)\circ_z g=\star_y(f)\circ_z g\,.
\end{equation*}
The second equality in (i) follows by considering the additional factor of $\sqrt{\rd_Z}$.
To prove the second assertion, first notice that by (i) resolving a label over simple objects and gluing along a different label are procedures that commute. Now, by Lemma~\ref{bases_comp}, the two possible configurations of gluing pairs of dual bases agree. More explicitly, we have graphically for each simple object $s\in\Irr\cx$ that
\def\len{1.75cm}
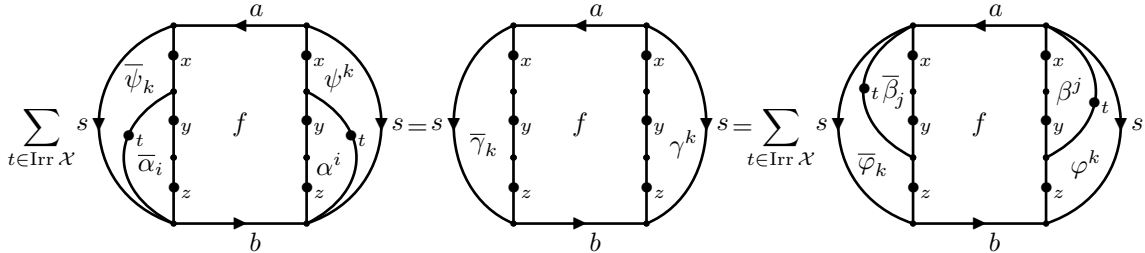
\begin{equation*}
\sum_{t\in\Irr\cx}\!\!\!
\begin{tikzpicture}[baseline={([yshift=-.5ex]current bounding box.center)}]
\draw[line width=1]
(-0.5*\len,0) .. controls (-1*\len,0.25*\len) and (-1*\len,0.75*\len).. (-0.5*\len,\len);
\draw[line width=1]
(-0.5*\len,0) .. controls (-1.25*\len,0.25*\len) and (-1.25*\len,1.25*\len).. (-0.5*\len,1.5*\len)
node[midway,rotate=-90] {\scriptsize $\Arr$}
    node[midway,left] {$s$};
\draw[line width=1]
(0.5*\len,0) .. controls (1*\len,0.25*\len) and (1*\len,0.75*\len).. (0.5*\len,\len);
\draw[line width=1]
(0.5*\len,0) .. controls (1.25*\len,0.25*\len) and (1.25*\len,1.25*\len).. (0.5*\len,1.5*\len)
node[midway,rotate=-90] {\scriptsize $\Arr$}
node[midway,right] {$s$};

\draw (0.75*\len,1.1*\len) node  {\normalsize ${\psi^k}$};
\draw (-0.75*\len,1.1*\len) node  {\normalsize ${\overline{\psi}_k}$};
\draw (0.67*\len,0.45*\len) node  {\normalsize $
{\alpha^i}$};
\draw (-0.67*\len,0.45*\len) node  {\normalsize ${\overline{\alpha}_i}$};

\node at (-0.8*\len,0.65*\len) {$\bullet_t$};
\node at (0.87*\len,0.65*\len) {$\bullet_t$};
%first element
\draw[line width=1] 
    (-0.5*\len,0) to (0.5*\len,0)
    (-0.5*\len,0) to (-0.5*\len,0.5*\len)
    (0.5*\len,0) to (0.5*\len,0.5*\len);

\node at (-0.45*\len,0.25*\len) {\normalsize $\bullet_z$};
\node[rotate=0] at (0.55*\len,0.25*\len) {\normalsize $\bullet_z$};
\node[rotate=0] at (0,0) {\scriptsize $\Arr$};

\filldraw [black] (-0.5*\len,0) circle (1pt);
\filldraw [black] (0.5*\len,0) circle (1pt);
\filldraw [black] (-0.5*\len,0.5*\len) circle (1pt);
\filldraw [black] (0.5*\len,0.5*\len) circle (1pt);
% Text Node
\node at (0,0) [anchor=north west] {$b$};

\begin{scope}[yshift=0.5*\len]
\draw[line width=1] 
    (-0.5*\len,0) to (-0.5*\len,0.5*\len)
    (0.5*\len,0) to (0.5*\len,0.5*\len);
\node at (-0.45*\len,0.25*\len) {\normalsize $\bullet_y$};
\node[rotate=0] at (0.55*\len,0.25*\len) {\normalsize $\bullet_y$};

\filldraw [black] (-0.5*\len,0) circle (1pt);
\filldraw [black] (0.5*\len,0) circle (1pt);
\filldraw [black] (-0.5*\len,0.5*\len) circle (1pt);
\filldraw [black] (0.5*\len,0.5*\len) circle (1pt);

\draw (0,0.25*\len) node    {$f$};
\end{scope}
\begin{scope}[yshift=\len]
\draw[line width=1] 
    (-0.5*\len,0) to (-0.5*\len,0.5*\len)
    (0.5*\len,0) to (0.5*\len,0.5*\len)
    (-0.5*\len,0.5*\len) to (0.5*\len,0.5*\len);

\node at (-0.45*\len,0.25*\len) {\normalsize $\bullet_x$};
\node at (0.55*\len,0.25*\len) {\normalsize $\bullet_x$};
\node[rotate=180] at (0,0.5*\len) {\scriptsize $\Arr$};

\filldraw [black] (-0.5*\len,0) circle (1pt);
\filldraw [black] (0.5*\len,0) circle (1pt);
\filldraw [black] (-0.5*\len,0.5*\len) circle (1pt);
\filldraw [black] (0.5*\len,0.5*\len) circle (1pt);
\node at (0,0.5*\len) [anchor=south west] {$a$};
\end{scope}
\end{tikzpicture}
\!\!\!=\!\!\!
\begin{tikzpicture}[baseline={([yshift=-.5ex]current bounding box.center)}]
\draw[line width=1]
(-0.5*\len,0) .. controls (-1.1*\len,0.25*\len) and (-1.1*\len,1.25*\len).. (-0.5*\len,1.5*\len)
node[midway,rotate=-90] {\scriptsize $\Arr$}
    node[midway,left] {$s$};
\draw[line width=1]
(0.5*\len,0) .. controls (1.1*\len,0.25*\len) and (1.1*\len,1.25*\len).. (0.5*\len,1.5*\len)
node[midway,rotate=-90] {\scriptsize $\Arr$}
node[midway,right] {$s$};

\draw (0.78*\len,0.60*\len) node   {$\gamma^k$};
\draw (-0.72*\len,0.60*\len) node   {$\overline{\gamma}_k$};
%first element
\draw[line width=1] 
    (-0.5*\len,0) to (0.5*\len,0)
    (-0.5*\len,0) to (-0.5*\len,0.5*\len)
    (0.5*\len,0) to (0.5*\len,0.5*\len);

\node at (-0.45*\len,0.25*\len) {\normalsize $\bullet_z$};
\node[rotate=0] at (0.55*\len,0.25*\len) {\normalsize $\bullet_z$};
\node[rotate=0] at (0,0) {\scriptsize $\Arr$};

\filldraw [black] (-0.5*\len,0) circle (1pt);
\filldraw [black] (0.5*\len,0) circle (1pt);
\filldraw [black] (-0.5*\len,0.5*\len) circle (1pt);
\filldraw [black] (0.5*\len,0.5*\len) circle (1pt);

% Text Node
\node at (0,0) [anchor=north west] {$b$};

\begin{scope}[yshift=0.5*\len]
\draw[line width=1] 
    (-0.5*\len,0) to (-0.5*\len,0.5*\len)
    (0.5*\len,0) to (0.5*\len,0.5*\len);
\node at (-0.45*\len,0.25*\len) {\normalsize $\bullet_y$};
\node[rotate=0] at (0.55*\len,0.25*\len) {\normalsize $\bullet_y$};

\filldraw [black] (-0.5*\len,0) circle (1pt);
\filldraw [black] (0.5*\len,0) circle (1pt);
\filldraw [black] (-0.5*\len,0.5*\len) circle (1pt);
\filldraw [black] (0.5*\len,0.5*\len) circle (1pt);

\draw (0,0.25*\len) node    {$f$};
\end{scope}
\begin{scope}[yshift=\len]
\draw[line width=1] 
    (-0.5*\len,0) to (-0.5*\len,0.5*\len)
    (0.5*\len,0) to (0.5*\len,0.5*\len)
    (-0.5*\len,0.5*\len) to (0.5*\len,0.5*\len);

\node at (-0.45*\len,0.25*\len) {\normalsize $\bullet_x$};
\node[rotate=0] at (0.55*\len,0.25*\len) {\normalsize $\bullet_x$};
\node[rotate=180] at (0,0.5*\len) {\scriptsize $\Arr$};

\filldraw [black] (-0.5*\len,0) circle (1pt);
\filldraw [black] (0.5*\len,0) circle (1pt);
\filldraw [black] (-0.5*\len,0.5*\len) circle (1pt);
\filldraw [black] (0.5*\len,0.5*\len) circle (1pt);

% Text Node
\node at (0,0.5*\len) [anchor=south west] {$a$};
\end{scope}
\end{tikzpicture}
\!\!\!=\!\!\!
\sum_{t\in\Irr\cx}\!\!\!
\begin{tikzpicture}[baseline={([yshift=-.5ex]current bounding box.center)}]
\draw[line width=1] 
(-0.5*\len,0.5*\len) .. controls (-1*\len,0.75*\len) and (-1*\len,1.25*\len).. (-0.5*\len,1.5*\len);
\draw[line width=1] 
(-0.5*\len,0) .. controls (-1.25*\len,0.25*\len) and (-1.25*\len,1.25*\len).. (-0.5*\len,1.5*\len)
node[midway,rotate=-90] {\scriptsize $\Arr$}
    node[midway,left] {$s$};
\draw[line width=1] 
(0.5*\len,0.5*\len) .. controls (1*\len,0.75*\len) and (1*\len,1.25*\len).. (0.5*\len,1.5*\len);
\draw[line width=1] 
(0.5*\len,0) .. controls (1.25*\len,0.25*\len) and (1.25*\len,1.25*\len).. (0.5*\len,1.5*\len)
node[midway,rotate=-90] {\scriptsize $\Arr$}
node[midway,right] {$s$};

\draw (0.8*\len,0.45*\len) node   {$\varphi^k$};
\draw (-0.8*\len,0.45*\len) node   {$\overline{\varphi}_k$};
\draw (0.68*\len,1*\len) node   {$\beta^j$};
\draw (-0.64*\len,1*\len) node   {$\overline{\beta}_{\!j}$};

\node at (-0.83*\len,1*\len) {$\bullet_t$};
\node at (0.9*\len,0.9*\len) {$\bullet_t$};
%first element
\draw[line width=1] 
    (-0.5*\len,0) to (0.5*\len,0)
    (-0.5*\len,0) to (-0.5*\len,0.5*\len)
    (0.5*\len,0) to (0.5*\len,0.5*\len);
\node at (-0.45*\len,0.25*\len) {\normalsize $\bullet_z$};
\node[rotate=0] at (0.55*\len,0.25*\len) {\normalsize $\bullet_z$};
\node[rotate=0] at (0,0) {\scriptsize $\Arr$};

\filldraw [black] (-0.5*\len,0) circle (1pt);
\filldraw [black] (0.5*\len,0) circle (1pt);
\filldraw [black] (-0.5*\len,0.5*\len) circle (1pt);
\filldraw [black] (0.5*\len,0.5*\len) circle (1pt);

% Text Node
\node at (0,0) [anchor=north west] {$b$};
\begin{scope}[yshift=0.5*\len]
\draw[line width=1] 
    (-0.5*\len,0) to (-0.5*\len,0.5*\len)
    (0.5*\len,0) to (0.5*\len,0.5*\len);
\node at (-0.45*\len,0.25*\len) {\normalsize $\bullet_y$};
\node[rotate=0] at (0.55*\len,0.25*\len) {\normalsize $\bullet_y$};

\filldraw [black] (-0.5*\len,0) circle (1pt);
\filldraw [black] (0.5*\len,0) circle (1pt);
\filldraw [black] (-0.5*\len,0.5*\len) circle (1pt);
\filldraw [black] (0.5*\len,0.5*\len) circle (1pt);

\draw (0,0.25*\len) node    {$f$};
\end{scope}
\begin{scope}[yshift=\len]
\draw[line width=1] 
    (-0.5*\len,0) to (-0.5*\len,0.5*\len)
    (0.5*\len,0) to (0.5*\len,0.5*\len)
    (-0.5*\len,0.5*\len) to (0.5*\len,0.5*\len);

\node at (-0.45*\len,0.25*\len) {\normalsize $\bullet_x$};
\node[rotate=0] at (0.55*\len,0.25*\len) {\normalsize $\bullet_x$};
\node[rotate=180] at (0,0.5*\len) {\scriptsize $\Arr$};

\filldraw [black] (-0.5*\len,0) circle (1pt);
\filldraw [black] (0.5*\len,0) circle (1pt);
\filldraw [black] (-0.5*\len,0.5*\len) circle (1pt);
\filldraw [black] (0.5*\len,0.5*\len) circle (1pt);

% Text Node
\node at (0,0.5*\len) [anchor=south west] {$a$};
\end{scope}
\end{tikzpicture}
\end{equation*}
which translates into equation~\eqref{star_triprod}. The first part of (iii) simply states that resolving different labels into simple objects are independent procedures and therefore commute. The second part of (iii) follows from a similar argument to the one that proves (ii).
Item (iv) is immediate from the definition, considering $x$ is a simple object. The fifth assertion follows by applying Lemmas~\ref{bases_comp},~\ref{lem:base_change} and~\ref{lem:pairing_properties} (i).
\end{proof}
\subsection{Spaces of tubes and the Tube algebra}
The rotational invariance of the graphical calculus allows for representing certain graphs as annuli or tubes:
two non-consecutive edges of a graph with the same label up to orientation reversal are merged. This leads to a loop labeled by the product of the labels of the edges in between the merged edges.
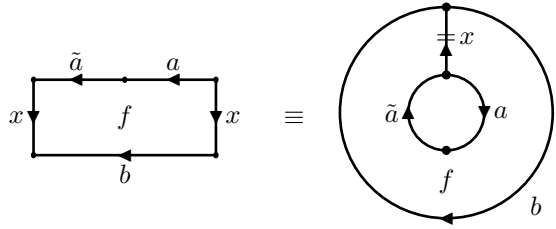
\def\len{2cm}
\def\radL{1.4cm}
\def\radS{0.5cm}
\begin{equation*}
\begin{tikzpicture}[xscale=0.8,baseline={([yshift=-.5ex]current bounding box.center)}]
    \draw[line width=1] 
    (-0.75*\len,0) to (0.75*\len,0)
    (-0.75*\len,0) to (-0.75*\len,0.5*\len)
    (0.75*\len,0) to (0.75*\len,0.5*\len)
    (-0.75*\len,0.5*\len) to (0.75*\len,0.5*\len);

\node[rotate=-90] at (-0.75*\len,0.25*\len) {\scriptsize $\Arr$};
\node[rotate=-90] at (0.75*\len,0.25*\len) {\scriptsize $\Arr$};
\node[rotate=180] at (0,0) {\scriptsize $\Arr$};
\node[rotate=180] at (-0.4*\len,0.5*\len) {\scriptsize $\Arr$};
\node[rotate=180] at (0.4*\len,0.5*\len) {\scriptsize $\Arr$};

\filldraw [black] (-0.75*\len,0) circle (1pt);
\filldraw [black] (0.75*\len,0) circle (1pt);
\filldraw [black] (0,0.5*\len) circle (1pt);
\filldraw [black] (-0.75*\len,0.5*\len) circle (1pt);
\filldraw [black] (0.75*\len,0.5*\len) circle (1pt);
% Text Node
\node at (-0.75*\len,0.25*\len) [anchor=east] {$x$};
\node at (0.75*\len,0.25*\len) [anchor=west] {$x$};
\node at (0,0) [anchor=north] {$b$};
\node at (0.4*\len,0.5*\len) [anchor=south] {$a$};
\node at (-0.4*\len,0.5*\len) [anchor=south] {$\tilde{a}$};

\draw (0,0.25*\len) node    {$ {f}$};
\end{tikzpicture}
\quad\equiv\quad
\begin{tikzpicture}[baseline={([yshift=-.5ex]current bounding box.center)}]
\draw[line width=1] (0,0) circle (\radL)
(0,0) circle (\radS);
\draw[line width=1] 
    (0,\radS) to (0,\radL);
\node[rotate=90] at (0,0.6*\radL) {\scriptsize $\Arr$};
\node[rotate=180] at (270:\radL) {\scriptsize $\Arr$};
\node[rotate=-90] at (0:\radS) {\scriptsize $\Arr$};
\node[rotate=90] at (180:\radS) {\scriptsize $\Arr$};
\filldraw [black] (0,\radS) circle (1.5pt);
\filldraw [black] (270:\radS) circle (1.5pt);
\filldraw [black] (0,\radL) circle (1.5pt);
% Labels
\node at (0,0.7*\radL)  {$=$};
\node at (0.05,0.7*\radL) [anchor=west] {$x$};
\node at (-40:1.1*\radL) [anchor=north] {$b$};
\node at (0:\radS) [anchor=west] {$a$};
\node at (180:\radS) [anchor=east] {$\tilde{a}$};
\node at (-90:0.45*\radL) [anchor=north]  {$ {f}$};
\end{tikzpicture}
\end{equation*}

\begin{defi}\label{def:data-for-tube-alg}
Let $\cx$ be a spherical semisimple multitensor category, and let $\rI = \coprod_{i \in \Gamma} \rI_{i}$ and $\rJ = \coprod_{i, j\in \Gamma}\rJ_{i j}$ be collections of isomorphism classes $\Irr\cx_{i } \subset \IndexSet_{i} \subset \Iso\cx_{i }$ and $\mathrm{J}_{i j} \subset \Iso\cx_{i j}$.
\begin{enumerate}[(i)]
    \item Given a pair of objects $a,b\in\IndexSet$, define the \textit{space of tubes from $b$ to $a$} as the vector space    
\begin{equation*}
    \mathrm{T}_{a;b}^{\mathrm{J}} \coloneqq\bigoplus_{x\in \mathrm{J}}\,\Hsp{\overline{b}\,\overline{x}\,a\,x}
\end{equation*}
We refer to a vector $f\in \Hsp{\overline{b}\,\overline{x}\,a\,x}$ as a \textit{tube element} and simply write $\mathrm{T}_{a;b}$ for the choice $\mathrm{J}_{i j} = \Irr\cx_{i j}$.
\item Given $a,b,c\in\IndexSet$, the \textit{welding} of tube elements is defined by the map
\begin{equation}\label{eq:tube_weld}
\mathrm{w}\Colon\mathrm{T}_{a;b}^{\mathrm{J}}\otimes \mathrm{T}_{b;c} ^{\mathrm{J}}\to\mathrm{T}_{a;c}^{\mathrm{J}}
\end{equation}
given for $f\in\Hsp{\overline{b}\overline{x}ax}$ and $g\in\Hsp{\overline{c}\,\overline{y}by}$ by
\begin{equation*}
% \mathrm{w}(f\otimes g)\coloneqq
\star_{xy}(f\circ_b g)=\sum_{t\in\Irr\cx}\!\!
\def\len{1.8cm}
\begin{tikzpicture}[baseline={([yshift=-.5ex]current bounding box.center)}]

\draw[line width=1]
(-0.5*\len,0) .. controls (-1*\len,0.25*\len) and (-1*\len,0.75*\len).. (-0.5*\len,\len);
\draw[line width=1]
(0.5*\len,0) .. controls (1*\len,0.25*\len) and (1*\len,0.75*\len).. (0.5*\len,\len);

\draw (0.7*\len,0.5*\len) node {${\eta^i}$};
\draw (-0.7*\len,0.5*\len) node {${\overline{{\eta }}_i}$};
\node at (-0.87*\len,0.5*\len) [anchor=east] {$t$};
\node at (0.87*\len,0.5*\len) [anchor=west] {$t$};
\node[rotate=-90] at (-0.87*\len,0.5*\len) {\scriptsize $\Arr$};
\node[rotate=-90] at (0.87*\len,0.5*\len) {\scriptsize $\Arr$};
\draw[line width=1] 
    (-0.5*\len,0) to (0.5*\len,0)
    (-0.5*\len,0) to (-0.5*\len,0.5*\len)
    (0.5*\len,0) to (0.5*\len,0.5*\len)
    (-0.5*\len,0.5*\len) to (0.5*\len,0.5*\len);

\node at (-0.45*\len,0.25*\len) {\normalsize $\bullet_y$};
\node at (0.55*\len,0.25*\len) {\normalsize $\bullet_y$};
\node[rotate=180] at (0,0.5*\len) {\scriptsize $\Arr$};
\node[rotate=180] at (0,0) {\scriptsize $\Arr$};

\filldraw [black] (-0.5*\len,0) circle (1pt);
\filldraw [black] (0.5*\len,0) circle (1pt);
\filldraw [black] (-0.5*\len,0.5*\len) circle (1pt);
\filldraw [black] (0.5*\len,0.5*\len) circle (1pt);

\node at (0.2*\len,0.5*\len) [anchor=south] {$b$};
\node at (0.2*\len,0) [anchor=north] {$c$};

\draw (-0.1*\len,0.23*\len) node  {${g}$};

\begin{scope}[yshift=0.5*\len]
\draw[line width=1] 
    (-0.5*\len,0) to (0.5*\len,0)
    (-0.5*\len,0) to (-0.5*\len,0.5*\len)
    (0.5*\len,0) to (0.5*\len,0.5*\len)
    (-0.5*\len,0.5*\len) to (0.5*\len,0.5*\len);

\node at (-0.45*\len,0.25*\len) {\normalsize $\bullet_x$};
\node at (0.55*\len,0.25*\len) {\normalsize $\bullet_x$};
\node[rotate=180] at (0,0.5*\len) {\scriptsize $\Arr$};

\filldraw [black] (-0.5*\len,0) circle (1pt);
\filldraw [black] (0.5*\len,0) circle (1pt);
\filldraw [black] (-0.5*\len,0.5*\len) circle (1pt);
\filldraw [black] (0.5*\len,0.5*\len) circle (1pt);

% Text Node
\node at (0,0.5*\len) [anchor=south west] {$a$};
\draw (-0.1*\len,0.25*\len) node  {${f}$};
\end{scope}
\end{tikzpicture}
\,.
\end{equation*}
\end{enumerate}
\end{defi}

We note that $\Hsp{\overline{b} \overline{x} a x}$ is nontrivial only if $x \in \rJ_{i j}$ holds for $i, j \in \Gamma$ $a \in \rI_{i i}$, $b \in \rI_{j j}$, so we ought to take the direct sum over this small choice of indexes, but supress it to avoid cumbersome notation.
In the following we will write $\Irr \cx \subset \IndexSet$ to mean $\Irr\cx_{i } \subset \IndexSet_{i}$ for $i \in \Gamma$, and similarly for $\rJ$.

\begin{prop}\label{prop:associativity}
Welding of tube elements is associative, i.e., the diagram
\begin{equation*}
    \begin{tikzcd}[column sep=normal, row sep=normal]
        \mathrm{T}_{a;b}^{\mathrm{J}}\otimes\mathrm{T}_{b;c}^{\mathrm{J}}\otimes\mathrm{T}_{c;d} ^{\mathrm{J}}\ar[r,"\id\otimes\mathrm{w}"]\ar[d,swap,"\mathrm{w}\otimes\id"]
        &\mathrm{T}_{a;b}^{\mathrm{J}}\otimes\mathrm{T}_{b;d}^{\mathrm{J}}
        \ar[d,"\mathrm{w}"]\\
        \mathrm{T}_{a;c}^{\mathrm{J}}\otimes\mathrm{T}_{c;d}^{\mathrm{J}}
        \ar[r,swap,"\mathrm{w}"]& 
        \mathrm{T}_{a;d}^{\mathrm{J}}
    \end{tikzcd}
\end{equation*}
commutes for all $a,b,c,d\in\Iso\cx$.
\end{prop}
\begin{proof}
Associativity follows from Lemma~\ref{star_prop}. Indeed, given $f\in(\mathrm{T}_{a;b}^{\mathrm{J}})_x$, $g\in(\mathrm{T}_{b;c}^{\mathrm{J}})_y$ and $h\in(\mathrm{T}_{c;d}^{\mathrm{J}})_z$ we have on one hand that
\begin{equation}\label{eq:W_asso_1}
    \mathrm{w}(f\otimes \mathrm{w}(g\otimes h))=\mathrm{w}(f\otimes\star_{yz}( g\circ_{c} h))
    =\sum_{t\in\Irr\cx}\star_{xt}(f\circ_b\star_{yz}( g\circ_{c} h)_t)
    \overset{\eqref{distr_star}}{=}
    \sum_{t\in\Irr\cx}\star_{xt}(\star_{yz}(f\circ_b g\circ_{c} h)_t)
\end{equation}
On the other hand, it holds that
\begin{equation}\label{eq:W_asso_2}
    \mathrm{w}(\mathrm{w}(f\otimes g)\otimes h)=\mathrm{w}(\star_{xy}(f\circ_b g)\otimes h)
    =\sum_{t\in\Irr\cx}\star_{tz}(\star_{xy}(f\circ_b g)_t\circ_{c}h)
    \overset{\eqref{distr_star}}
    {=}\sum_{t\in\Irr\cx}\star_{tz}(\star_{xy}(f\circ_b g\circ_{c}h)_t)
\end{equation}
The last steps in~\eqref{eq:W_asso_1} and~\eqref{eq:W_asso_2} are graphically depicted as 
\begin{equation*}
\sum_{s,t\in\Irr\cx}\!\!\!
\def\len{2cm}
\begin{tikzpicture}[baseline={([yshift=-.5ex]current bounding box.center)}]
\draw[line width=1]
(-0.5*\len,0) .. controls (-1*\len,0.25*\len) and (-1*\len,0.75*\len).. (-0.5*\len,\len);
\draw[line width=1]
(-0.5*\len,0) .. controls (-1.25*\len,0.25*\len) and (-1.25*\len,1.25*\len).. (-0.5*\len,1.5*\len);
\draw[line width=1] 
(0.5*\len,0) .. controls (1*\len,0.25*\len) and (1*\len,0.75*\len).. (0.5*\len,\len);
\draw[line width=1]
(0.5*\len,0) .. controls (1.25*\len,0.25*\len) and (1.25*\len,1.25*\len).. (0.5*\len,1.5*\len);

\draw (0.75*\len,1.1*\len) node {\normalsize $ {\psi^k}$};
\draw (-0.75*\len,1.1*\len) node {\normalsize $ {\overline{\psi}_k}$};
\draw (0.67*\len,0.45*\len) node {\normalsize $ {\alpha^i}$};
\draw (-0.67*\len,0.45*\len) node {\normalsize $ {\overline{\alpha}_i}$};

\node at (-1.1*\len,0.75*\len) [anchor=east] {$s$};
\node at (1.1*\len,0.75*\len) [anchor=west] {$s$};
\node at (-0.8*\len,0.65*\len) {$\bullet_t$};
\node at (0.87*\len,0.65*\len) {$\bullet_t$};
\node[rotate=-90] at (-1.07*\len,0.75*\len) {\scriptsize $\Arr$};
\node[rotate=-90] at (1.07*\len,0.75*\len) {\scriptsize $\Arr$};
%first element
\draw[line width=1] 
    (-0.5*\len,0) to (0.5*\len,0)
    (-0.5*\len,0) to (-0.5*\len,0.5*\len)
    (0.5*\len,0) to (0.5*\len,0.5*\len)
    (-0.5*\len,0.5*\len) to (0.5*\len,0.5*\len);

\node at (-0.45*\len,0.25*\len) {\normalsize $\bullet_z$};
\node at (0.55*\len,0.25*\len) {\normalsize $\bullet_z$};
\node[rotate=180] at (0,0) {\scriptsize $\Arr$};
\node[rotate=180] at (0,0.5*\len) {\scriptsize $\Arr$};

\filldraw [black] (-0.5*\len,0) circle (1pt);
\filldraw [black] (0.5*\len,0) circle (1pt);
\filldraw [black] (-0.5*\len,0.5*\len) circle (1pt);
\filldraw [black] (0.5*\len,0.5*\len) circle (1pt);

% Text Node
\node at (0,0) [anchor=north west] {$d$};
\node at (0.2*\len,0.5*\len) [anchor=south] {$c$};
\draw (-0.1*\len,0.25*\len) node {${h}$};

\begin{scope}[yshift=0.5*\len]
\draw[line width=1] 
    (-0.5*\len,0) to (0.5*\len,0)
    (-0.5*\len,0) to (-0.5*\len,0.5*\len)
    (0.5*\len,0) to (0.5*\len,0.5*\len)
    (-0.5*\len,0.5*\len) to (0.5*\len,0.5*\len);

\node at (-0.45*\len,0.25*\len) {\normalsize $\bullet_y$};
\node at (0.55*\len,0.25*\len) {\normalsize $\bullet_y$};
\node[rotate=180] at (0,0.5*\len) {\scriptsize $\Arr$};

\filldraw [black] (-0.5*\len,0) circle (1pt);
\filldraw [black] (0.5*\len,0) circle (1pt);
\filldraw [black] (-0.5*\len,0.5*\len) circle (1pt);
\filldraw [black] (0.5*\len,0.5*\len) circle (1pt);

\node at (0.2*\len,0.5*\len) [anchor=south] {$b$};

\draw (-0.1*\len,0.25*\len) node    {$ {g}$};
\end{scope}
\begin{scope}[yshift=\len]
\draw[line width=1] 
    (-0.5*\len,0) to (0.5*\len,0)
    (-0.5*\len,0) to (-0.5*\len,0.5*\len)
    (0.5*\len,0) to (0.5*\len,0.5*\len)
    (-0.5*\len,0.5*\len) to (0.5*\len,0.5*\len);

\node at (-0.45*\len,0.25*\len) {\normalsize $\bullet_x$};
\node at (0.55*\len,0.25*\len) {\normalsize $\bullet_x$};
\node[rotate=180] at (0,0.5*\len) {\scriptsize $\Arr$};

\filldraw [black] (-0.5*\len,0) circle (1pt);
\filldraw [black] (0.5*\len,0) circle (1pt);
\filldraw [black] (-0.5*\len,0.5*\len) circle (1pt);
\filldraw [black] (0.5*\len,0.5*\len) circle (1pt);
\node at (0,0.5*\len) [anchor=south west] {$a$};

\draw (-0.1*\len,0.25*\len) node    {$ {f}$};
\end{scope}
\end{tikzpicture}
\!\overset{\eqref{star_triprod}}{=}\!
\sum_{s,t\in\Irr\cx}\!\!\!
\def\len{2cm}
\begin{tikzpicture}[baseline={([yshift=-.5ex]current bounding box.center)}]

\draw[line width=1]
(-0.5*\len,0.5*\len) .. controls (-1*\len,0.75*\len) and (-1*\len,1.25*\len).. (-0.5*\len,1.5*\len);
\draw[line width=1]
(-0.5*\len,0) .. controls (-1.25*\len,0.25*\len) and (-1.25*\len,1.25*\len).. (-0.5*\len,1.5*\len);
\draw[line width=1] 
(0.5*\len,0.5*\len) .. controls (1*\len,0.75*\len) and (1*\len,1.25*\len).. (0.5*\len,1.5*\len);
\draw[line width=1]
(0.5*\len,0) .. controls (1.25*\len,0.25*\len) and (1.25*\len,1.25*\len).. (0.5*\len,1.5*\len);

\draw (0.75*\len,0.4*\len) node {\normalsize $\ {\varphi^k}$};
\draw (-0.75*\len,0.4*\len) node {\normalsize $\ {\overline{\varphi}_k}$};
\draw (0.68*\len,1*\len) node {\normalsize $\ {\beta^j}$};
\draw (-0.64*\len,1*\len) node {\normalsize $\ {\overline{\beta}_{\!j}}$};

\node at (-1.1*\len,0.75*\len) [anchor=east] {$s$};
\node at (1.1*\len,0.75*\len) [anchor=west] {$s$};
\node at (-0.83*\len,1*\len) {$\bullet_t$};
\node at (0.9*\len,0.9*\len) {$\bullet_t$};
\node[rotate=-90] at (-1.07*\len,0.75*\len) {\scriptsize $\Arr$};
\node[rotate=-90] at (1.07*\len,0.75*\len) {\scriptsize $\Arr$};
%first element
\draw[line width=1] 
    (-0.5*\len,0) to (0.5*\len,0)
    (-0.5*\len,0) to (-0.5*\len,0.5*\len)
    (0.5*\len,0) to (0.5*\len,0.5*\len)
    (-0.5*\len,0.5*\len) to (0.5*\len,0.5*\len);

\node at (-0.45*\len,0.25*\len) {\normalsize $\bullet_z$};
\node at (0.55*\len,0.25*\len) {\normalsize $\bullet_z$};
\node[rotate=180] at (0,0) {\scriptsize $\Arr$};
\node[rotate=180] at (0,0.5*\len) {\scriptsize $\Arr$};

\filldraw [black] (-0.5*\len,0) circle (1pt);
\filldraw [black] (0.5*\len,0) circle (1pt);
\filldraw [black] (-0.5*\len,0.5*\len) circle (1pt);
\filldraw [black] (0.5*\len,0.5*\len) circle (1pt);

% Text Node
\node at (0,0) [anchor=north west] {$d$};
\node at (0.2*\len,0.5*\len) [anchor=south] {$c$};

\draw (-0.1*\len,0.25*\len) node    {$  {h}$};

\begin{scope}[yshift=0.5*\len]
\draw[line width=1] 
    (-0.5*\len,0) to (0.5*\len,0)
    (-0.5*\len,0) to (-0.5*\len,0.5*\len)
    (0.5*\len,0) to (0.5*\len,0.5*\len)
    (-0.5*\len,0.5*\len) to (0.5*\len,0.5*\len);

\node at (-0.45*\len,0.25*\len) {\normalsize $\bullet_y$};
\node at (0.55*\len,0.25*\len) {\normalsize $\bullet_y$};
\node[rotate=180] at (0,0.5*\len) {\scriptsize $\Arr$};

\filldraw [black] (-0.5*\len,0) circle (1pt);
\filldraw [black] (0.5*\len,0) circle (1pt);
\filldraw [black] (-0.5*\len,0.5*\len) circle (1pt);
\filldraw [black] (0.5*\len,0.5*\len) circle (1pt);

\node at (0.2*\len,0.5*\len) [anchor=south] {$b$};

\draw (-0.1*\len,0.25*\len) node    {$  {g}$};
\end{scope}
\begin{scope}[yshift=\len]
\draw[line width=1] 
    (-0.5*\len,0) to (0.5*\len,0)
    (-0.5*\len,0) to (-0.5*\len,0.5*\len)
    (0.5*\len,0) to (0.5*\len,0.5*\len)
    (-0.5*\len,0.5*\len) to (0.5*\len,0.5*\len);

\node at (-0.45*\len,0.25*\len) {\normalsize $\bullet_x$};
\node at (0.55*\len,0.25*\len) {\normalsize $\bullet_x$};
\node[rotate=180] at (0,0.5*\len) {\scriptsize $\Arr$};

\filldraw [black] (-0.5*\len,0) circle (1pt);
\filldraw [black] (0.5*\len,0) circle (1pt);
\filldraw [black] (-0.5*\len,0.5*\len) circle (1pt);
\filldraw [black] (0.5*\len,0.5*\len) circle (1pt);

% Text Node
\node at (0,0.5*\len) [anchor=south west] {$a$};

\draw (-0.1*\len,0.25*\len) node    {$  {f}$};
\end{scope}
\end{tikzpicture}
\end{equation*}
which both equal to $\star_{xyz}(f\circ_b g\circ_{c}h)$, leading to the desired result.
\end{proof}

\begin{rem}
We will fix the label set to be $\mathrm{J}_{i j}=\Irr\cx_{i j}$.
This is intrinsically related to the following construction, see also Section~\ref{sec:drinfeld_center}.
Given $a \in \cx_{i }$ and $j \in \Gamma$, we set
\begin{equation*}
    \mathrm{Z}_j(a)=\bigoplus_{x\in\Irr\cx_{i j}}\,\overline{x}\otimes a\otimes x
\end{equation*}
in the ind-category of $\cx_{j j}$.
Since $\cx$ is semisimple, this can be interpreted as a coend construction.
Therefore, the vector spaces of tubes can be described as $\mathrm{T}_{a;b}=\Hsp{\overline{b}\; \mathrm{Z}_j(a)}$ for $b \in \cx_{j j}$, as long as the above assumption on $\rJ$ stands.
\end{rem}

When $\cx$ is a spherical fusion category, the above specializes to the \emph{tube algebra} $\Tube$, which is a symmetric Frobenius algebra whose representations realize the Drinfeld center of $\cx$.
The corresponding claim for multifusion setting is straightforward, but if we allow for $\cx$ to have infinitely many isomorphism classes of simple objects, $\Tube$ is no longer a symmetric Frobenius algebra, since it is not unital.

Let $\cx$ be a spherical semisimple multitensor category, and let $\Irr\cx\subset \IndexSet\subset\Iso\cx$ be in the diagonal part as in Definition~\ref{def:data-for-tube-alg}.
The vector space 
\begin{equation*}
\Tubecx\coloneqq \bigoplus_{a,b\in\IndexSet}\mathrm{T}_{a;b}=\bigoplus_{a,b\in\IndexSet}\bigoplus_{x\in \Irr \cx}\,\Hsp{\overline{b}\,\overline{x}\,a\,x}
\end{equation*}
is endowed with the structure of an associative $\bk$-algebra with local units. The product is given by welding of tube elements, i.e., for $f\in\Hsp{\overline{b}\overline{x}ax}$ and $g\in\Hsp{\overline{c}\,\overline{y}by}$ 
\begin{equation*}
f\cdot g\coloneqq \mathrm{w}(f\otimes g)=\star_{xy}(f\circ_b g)
\end{equation*}
which is associative by Proposition~\ref{prop:associativity}.
Now, given a finite set $F\subset\IndexSet$, multiplication by the idempotent
\begin{equation*}
\Id_F\coloneqq\;\sum_{a\in F}\,\Id_a\in \bigoplus_{a\in F}\,\Hsp{\overline{a}\,a}\,.
\end{equation*}
from both sides projects $\Tubecx$ to $\bigoplus_{a,b\in F}\mathrm{T}_{a;b}$.

\begin{defi}
Let $\cx$ be a spherical semisimple multitensor category and $\Irr\cx\subset \IndexSet\subset\Iso\cx$ be in the diagonal part as in Definition~\ref{def:data-for-tube-alg}. The associative $\bk$-algebra with local units $\Tubecx$ is called the \textit{tube algebra of $\cx$}.
\end{defi}

In particular, a homogeneous element of the tube algebra $f\in\Hsp{\overline{b}\overline{x}ax}$ can be depicted as
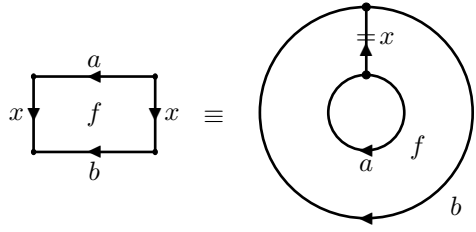
\def\len{2cm}
\def\radL{1.4cm}
\def\radS{0.5cm}
\begin{equation*}
\tubelem{a}{b}{x}{f}
\;\equiv\quad
\begin{tikzpicture}[baseline={([yshift=-.5ex]current bounding box.center)}]
\draw[line width=1] (0,0) circle (\radL)
(0,0) circle (\radS);
\draw[line width=1] 
    (0,\radS) to (0,\radL);
\node[rotate=90] at (0,0.6*\radL) {\scriptsize $\Arr$};
\node[rotate=180] at (270:\radL) {\scriptsize $\Arr$};
\node[rotate=180] at (270:\radS) {\scriptsize $\Arr$};
\filldraw [black] (0,\radS) circle (1.5pt);
\filldraw [black] (0,\radL) circle (1.5pt);
% Labels
\node at (0,0.7*\radL)  {$=$};
\node at (0.05,0.7*\radL) [anchor=west] {$x$};
\node at (-40:1.1*\radL) [anchor=north] {$b$};
\node at (270:\radS) [anchor=north] {$a$};
\node at (-15:0.5*\radL) [anchor=north]  {$ {f}$};
\end{tikzpicture}
\end{equation*}
which justifies the terminology.

The tube algebra comes with a symmetric linear form given by the trace coming from the spherical structure of $\cx$, defined by
\begin{equation*}
    \epsilon\Colon\Tubecx\to\bk,\qquad \epsilon(f)=
    \begin{cases}
    \tr_a(f) & \text{if $f\in\Hsp{\overline{a}\,\un_i a \un_i}$}, \\
    0 & \text{if $f \in\Hsp{\overline{b} \overline{x} a x}$, $x \neq \un_i$ or $a \neq b$,}
    \end{cases}
\end{equation*}
where $a \in \cx_{i }$ and we interpret $\tr_a$ as a scalar by~\eqref{scalar_id}.
In the case that $\cx$ is fusion and $\IndexSet$ is finite, $\Tubecx$ is a Frobenius algebra
\begin{equation*}
    \Tubecx^*\cong \bigoplus_{a,b\in \IndexSet}\bigoplus_{x\in \Irr \cx}\,\Hsp{a\,x\,\overline{b}\,\overline{x}}^*\cong 
    \bigoplus_{a,b\in \IndexSet}\bigoplus_{x\in \Irr \cx}\,\Hsp{x\,b\,\overline{x}\,\overline{a}}\cong \Tubecx\,.
\end{equation*}

\section{Tube algebra representations}\label{sec:tub-alg-rep}

The goal of this section is to explicitly define a braided monoidal structure on the category $\Rep\Tubecx$ of tube algebra representations of a spherical semisimple multitensor category $\cx$. In Section~\ref{sec:drinfeld_center} this structure is related to the naturally induced on the Drinfeld center of the ind-completion of $\cx$.

\subsection{Tensor product of representations}\label{sec:tensor_prod}

Let $\cx$ be a spherical semisimple multitensor category and suppose we are given $\Irr\cx\subset \IndexSet\subset\Iso\cx$ in the diagonal part as in Definition~\ref{def:data-for-tube-alg}.
Every tube algebra representation $M\in\Rep\Tubecx$ admits a canonical decomposition
\begin{equation*}
    M\cong\bigoplus_{a\in\IndexSet}\,M_a,\quad \text{ where}\quad M_a\coloneqq M.\Id_a  
\end{equation*}
determined by the action of the local units $\Id_a$.

In order to define the tensor product of two tube algebra representations, we will consider the following auxiliary space
\begin{equation}\label{eq:prod-mult-space}
\auxprod\coloneqq\mathrm{T}_{ab;c}=\bigoplus_{x\in \Irr\cx}\,\Hsp{\overline{c}\,\overline{x}\,a\,b\,x}
\end{equation}
where $a$ and $b$ represent the degree of the vectors of the representations and $c$ the label where the tube algebra acts.

\begin{prop}\label{prop:tubeaction_onprod}
The tube algebra $\Tubecx$ acts from the right on $\auxprod$ via
\begin{equation}\label{eq:tubeaction_onprod}
    \pi.f\coloneqq\mathrm{w}(\pi\otimes f)=\star_{yx}(\pi\circ_c f)
\end{equation}
for $f\in\Hsp{\overline{d}\,\overline{y}\,c\,y}$ and $\pi\in\Hsp{\overline{c}\,\overline{x}\,a\,b\,x}$.
\end{prop}
\begin{proof}
The module associativity follows from Proposition~\ref{prop:associativity}. Locality is ensured by taking the sum of the idempotents $\Id_c$ for each basis element $\pi\in\Hsp{\overline{c}\,\overline{x}\,a\,b\,x}$ in a finite-dimensional $W\subset\auxprod$.
\end{proof}
Now, given tube algebra representations $M,N\in\Rep\Tubecx$ and $c\in\IndexSet$ define the vector space
\begin{equation*}
(M\overset{\sim}{\Boxt} N)_c\coloneqq   \bigoplus_{a,b\in \IndexSet}\,M_a\otimes N_b\otimes \auxprod
\end{equation*}
A homogeneous vector $m\otimes n \otimes \pi\in M_a\otimes N_b\otimes \auxprod$ is graphically depicted as
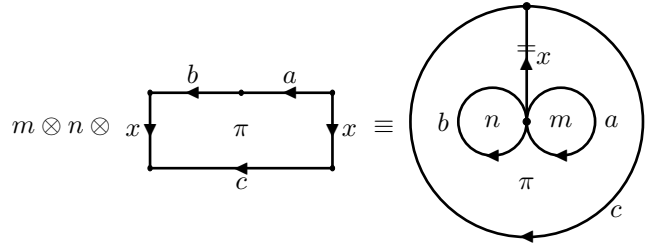
\begin{equation*}
m\otimes n \otimes \pielem{a}{b}{c}{x}{\pi}
\equiv\; \tenEl{\pi}{c}{x}{a}{b}{m}{ n }
\end{equation*}

The tube algebra action on $\auxprod$ defines an action on $M\overset{\sim}{\Boxt} N$. However, we need to account for the actions on the representations $M$ and $N$. To this end, we define an equivalence relation on $M\overset{\sim}{\Boxt} N$:
\begin{enumerate}[(i)]
    \item Given a tube algebra element $f\in\Hsp{\overline{a}\,\overline{u}a'u}$ acting on $m\in M_{a'}$, and vectors $n\in N_b$ and $\pi\in\auxprod$ we consider the following relation%\vspace{-1em}
\def\len{1.8cm}
\begin{equation}
\label{eq:Rel1}
m.f\otimes n \otimes\pielem{a}{b}{c}{x}{\pi}
\;\sim\,
\sum_{y,\tilde{b}\in\Irr\cx}\rd_{\tilde{b}}\;m\otimes n.\,\frac{\overline{\beta}_j}{\rd_u}\otimes\!\! \def\len{2cm}
\begin{tikzpicture}[xscale=0.8,baseline={([yshift=-.5ex]current bounding box.center)}]
\draw[line width=1]
(-0.75*\len,0) .. controls (-1.25*\len,0.25*\len) and (-1.25*\len,0.75*\len).. (-0.75*\len,\len);
\draw[line width=1]
(0.75*\len,0) .. controls (1.25*\len,0.25*\len) and (1.25*\len,0.75*\len).. (0.75*\len,\len);

\node at (-0.7*\len,0.25*\len) {\normalsize $\bullet_x$};
\node at (0.8*\len,0.25*\len) {\normalsize $\bullet_x$};
\node at (-0.7*\len,0.75*\len) {\normalsize $\bullet_u$};
\node at (0.8*\len,0.75*\len) {\normalsize $\bullet_u$};

\draw[line width=1] 
    (-0.75*\len,0) to (0.75*\len,0)
    (-0.75*\len,0) to (-0.75*\len,0.5*\len)
    (0.75*\len,0) to (0.75*\len,0.5*\len)
    (-0.75*\len,0.5*\len) to (0.75*\len,0.5*\len)
    (0,0.5*\len) to (0,\len)
    (0,\len) to (0.75*\len,\len)
    (0.75*\len,0.5*\len) to (0.75*\len,\len)
    (-0.75*\len,0.5*\len) to (-0.75*\len,\len)
    (0,\len) to (-0.75*\len,\len);

\node[rotate=-90] at (-1.125*\len,0.5*\len) {\scriptsize $\Arr$};
\node[rotate=-90] at (1.125*\len,0.5*\len) {\scriptsize $\Arr$};
\node[rotate=180] at (0,0) {\scriptsize $\Arr$};
\node[rotate=-90] at (0,0.75*\len) {\scriptsize $\Arr$};

\node[rotate=180] at (-0.4*\len,\len) {\scriptsize $\Arr$};
\node[rotate=180] at (-0.4*\len,0.5*\len) {\scriptsize $\Arr$};
\node[rotate=180] at (0.4*\len,0.5*\len) {\scriptsize $\Arr$};
\node[rotate=180] at (0.4*\len,\len) {\scriptsize $\Arr$};

\filldraw [black] (0,\len) circle (1pt);
\filldraw [black] (0.75*\len,\len) circle (1pt);
\filldraw [black] (-0.75*\len,\len) circle (1pt);
\filldraw [black] (-0.75*\len,0) circle (1pt);
\filldraw [black] (0.75*\len,0) circle (1pt);
\filldraw [black] (0,0.5*\len) circle (1pt);
\filldraw [black] (-0.75*\len,0.5*\len) circle (1pt);
\filldraw [black] (0.75*\len,0.5*\len) circle (1pt);

% Text Node
% \node at (-0.75*\len,0.25*\len) [anchor=east] {$x$};
% \node at (0.75*\len,0.25*\len) [anchor=west] {$x$};
% \node at (0.75*\len,0.75*\len) [anchor=west] {$u$};
% \node at (-0.75*\len,0.75*\len) [anchor=east] {$u$};
\node at (1.125*\len,0.5*\len) [anchor=west] {$y$};
\node at (-1.125*\len,0.5*\len) [anchor=east] {$y$};
\node at (0,0.75*\len) [anchor=east] {$u$};
\node at (0,0) [anchor=north] {$c$};
\node at (0.4*\len,0.5*\len) [anchor=north] {$a$};
\node at (0.4*\len,\len) [anchor=south] {$a'$};
\node at (-0.4*\len,0.5*\len) [anchor=north] {$b$};
\node at (-0.4*\len,\len) [anchor=south] {$\Tilde{b}$};

\node at (0.4*\len,0.75*\len) {{$f$}};
\node at (-0.4*\len,0.75*\len) { {$\beta^j$}};

\node at (-0.9*\len,0.5*\len) { {$\overline{\varphi}_{\!k}$}};
\node at (0.95*\len,0.5*\len) { {$\varphi^k$}};

\draw (0,0.25*\len) node {${\pi}$};
\end{tikzpicture}
\end{equation}%\vspace{-1em}
where $(\beta,\overline{\beta})$ is a pair of dual bases of $\Hsp{\overline{b}\,\overline{u}\tilde{b}u}$.
\item  Similarly, accounting for the tube algebra action on $N$, for $g\in\Hsp{\overline{b}\,\overline{v}b'v}$, and $m\in M_{a}$, $n\in N_{b'}$ and $\pi\in\auxprod$ we set the relation%\vspace{-1em}
\def\len{1.8cm}
\begin{equation}\label{eq:Rel2}
m\otimes n.g\otimes\pielem{a}{b}{c}{x}{\pi}
\sim
\sum_{y,\tilde{a}\in\Irr\cx}\rd_{\tilde{a}}\;m.\,\frac{\overline{\alpha}_{i}}{\rd_v}\otimes n \otimes \star_{vx}\left[
\!\!
\def\len{2cm}
\begin{tikzpicture}[xscale=0.8,baseline={([yshift=-.5ex]current bounding box.center)}]

\draw[line width=1] 
    (-0.75*\len,0) to (0.75*\len,0)
    (-0.75*\len,0) to (-0.75*\len,0.5*\len)
    (0.75*\len,0) to (0.75*\len,0.5*\len)
    (-0.75*\len,0.5*\len) to (0.75*\len,0.5*\len)
    (0,0.5*\len) to (0,\len)
    (0,\len) to (0.75*\len,\len)
    (0.75*\len,0.5*\len) to (0.75*\len,\len)
    (-0.75*\len,0.5*\len) to (-0.75*\len,\len)
    (0,\len) to (-0.75*\len,\len);

\node[rotate=180] at (0,0) {\scriptsize $\Arr$};
\node[rotate=-90] at (0,0.75*\len) {\scriptsize $\Arr$};
\node[rotate=-90] at (-0.75*\len,0.75*\len) {\scriptsize $\Arr$};
\node[rotate=-90] at (0.75*\len,0.75*\len) {\scriptsize $\Arr$};
\node[rotate=-90] at (-0.75*\len,0.25*\len) {\scriptsize $\Arr$};
\node[rotate=-90] at (0.75*\len,0.25*\len) {\scriptsize $\Arr$};

\node[rotate=180] at (-0.4*\len,\len) {\scriptsize $\Arr$};
\node[rotate=180] at (-0.4*\len,0.5*\len) {\scriptsize $\Arr$};
\node[rotate=180] at (0.4*\len,0.5*\len) {\scriptsize $\Arr$};
\node[rotate=180] at (0.4*\len,\len) {\scriptsize $\Arr$};

\filldraw [black] (0,\len) circle (1pt);
\filldraw [black] (0.75*\len,\len) circle (1pt);
\filldraw [black] (-0.75*\len,\len) circle (1pt);
\filldraw [black] (-0.75*\len,0) circle (1pt);
\filldraw [black] (0.75*\len,0) circle (1pt);
\filldraw [black] (0,0.5*\len) circle (1pt);
\filldraw [black] (-0.75*\len,0.5*\len) circle (1pt);
\filldraw [black] (0.75*\len,0.5*\len) circle (1pt);

% Text Node
\node at (-0.75*\len,0.25*\len) [anchor=east] {$x$};
\node at (0.75*\len,0.25*\len) [anchor=west] {$x$};
\node at (0.75*\len,0.75*\len) [anchor=west] {$v$};
\node at (-0.75*\len,0.75*\len) [anchor=east] {$v$};
\node at (0,0.75*\len) [anchor=east] {$v$};
\node at (0,0) [anchor=north] {$c$};
\node at (0.4*\len,0.5*\len) [anchor=north] {$a$};
\node at (0.4*\len,\len) [anchor=south] {$\tilde{a}$};
\node at (-0.4*\len,0.5*\len) [anchor=north] {$b$};
\node at (-0.4*\len,\len) [anchor=south] {$b'$};

\node at (0.4*\len,0.75*\len) { {$\alpha^i$}};
\node at (-0.4*\len,0.75*\len) {${g}$};

\draw (0,0.25*\len) node  {${\pi}$};
\end{tikzpicture}
\right]
\end{equation}    
for $(\alpha,\overline{\alpha})$ a pair of dual bases of $\Hsp{\overline{a}\,\overline{v}\tilde{a}v}$.
\end{enumerate}
\begin{defi}\label{def:tensor_prod}
Given tube algebra representations $M,N\in \Rep\Tubecx$, their \textit{tensor product} is defined as the quotient
\begin{equation}\label{tensor_prod}
    M\Boxt N\coloneqq\faktor{M\overset{\sim}{\Boxt} N}{\sim}
\end{equation}
where the equivalence relation $\sim$ is defined by~\eqref{eq:Rel1} and~\eqref{eq:Rel2}.
\end{defi}

\begin{rem}
In the setting of tensor categories, the summand of $x = \un$ in~\eqref{eq:prod-mult-space} is often used to describe the monoidal product, which can be interpreted as a variation of the Day convolution product, see for example~\cite[Section 3]{levinwen2024}.
While it is possible to reduce the computation to the case $x = \un_i$ for some $i$ (see Proposition~\ref{prop:rel_unit_simple} (iii)), we chose to work with this more flexible presentation because there is no preferred choice of the $0$-cell $i \in \Gamma$, and the braiding will inevitably introduce nontrivial labels $x$.
\end{rem}

\begin{rem}
The previously defined relations can be also graphically expressed in terms of tube-like notation: the relation starting from an action on $M$ as
\begin{equation*}
\tenEl{\pi}{c}{x}{a}{b}{m.f}{n}
\qquad\sim \qquad
\sum_{\tilde{b}\in\Irr\cx}\;\rd_{\tilde{b}}\rd_u^{-1}\;\star_{ux}\left(\tenEl{\! {\beta^j}\circ_u \!f\!\circ_{ab}\!\pi}{c}{ux}{a'}{\tilde{b}}{m}{ n . {\overline{\beta}_{\!j}}}\right)
\end{equation*}
and the relation involving the tube algebra action on $N$ is graphically presented as
\begin{equation*}
\tenEl{\pi}{c}{x}{a}{b}{m}{n.g}
\qquad\sim \qquad
\sum_{\tilde{a}\in\Irr\cx}\,\rd_{\tilde{a}}\rd_v^{-1}\;\star_{vx}\left(\tenEl{ {\alpha^i}\!\circ_v \!g\!\circ_{ab}\!\pi}{c}{vx}{\tilde{a}}{b'}{m.{\overline{\alpha}_{\!i}}}{n}\right)\;.
\end{equation*}
\end{rem}

\begin{prop}\label{prop:action_tensor_prod}
Let $M,N\in\Rep\Tubecx$ be tube algebra representations.
The tube algebra action~\eqref{eq:tubeaction_onprod} on $\auxprod$ descends to an action on the tensor product $M\Boxt N$.
\end{prop}
\begin{proof}
We need to verify that the action is compatible with the relations defining the quotient $M\Boxt N$. 
\def\len{1.75cm}
For the first relation~\eqref{eq:Rel1}, notice that $(m.g\otimes n \otimes\pi).f$ is given by
\begin{equation*}%\label{act_rel1}
m.g\otimes n \otimes\star_{xy}\left(
\begin{tikzpicture}[baseline={([yshift=-.5ex]current bounding box.center)},xscale=0.8]
\begin{scope}[yshift=-0.5*\len]
\draw[line width=1] 
    (-0.75*\len,0) to (0.75*\len,0)
    (-0.75*\len,0) to (-0.75*\len,0.5*\len)
    (0.75*\len,0) to (0.75*\len,0.5*\len)
    (-0.75*\len,0.5*\len) to (0.75*\len,0.5*\len);
\filldraw [black] (-0.75*\len,0) circle (1pt);
\filldraw [black] (0.75*\len,0) circle (1pt);
\node[rotate=180] at (0,0) {\scriptsize $\Arr$};
\node[rotate=-90] at (-0.75*\len,0.25*\len) {\scriptsize $\Arr$};
\node[rotate=-90] at (0.75*\len,0.25*\len) {\scriptsize $\Arr$};
\node at (0,0) [anchor=north] {$d$};
\node at (-0.75*\len,0.25*\len)[anchor=east] {$y$};
\node at (0.75*\len,0.25*\len)[anchor=west] {$y$};    
% 
% \node at (-0.70*\len,0.25*\len) { $\bullet_y$};
% \node at (0.80*\len,0.25*\len) { $\bullet_y$};    
\draw (0.1*\len,0.25*\len) node {$ {f}$};
% \node at (-0.9*\len,0.5*\len) { {$\overline{\psi}_{\!k}$}};
% \node at (0.95*\len,0.5*\len) { {$\psi^k$}};
 \end{scope}
    \draw[line width=1] 
    (-0.75*\len,0) to (0.75*\len,0)
    (-0.75*\len,0) to (-0.75*\len,0.5*\len)
    (0.75*\len,0) to (0.75*\len,0.5*\len)
    (-0.75*\len,0.5*\len) to (0.75*\len,0.5*\len);
% \node at (-0.70*\len,0.25*\len) { $\bullet_x$};
% \node at (0.80*\len,0.25*\len) { $\bullet_x$};
\node[rotate=-90] at (-0.75*\len,0.25*\len) {\scriptsize $\Arr$};
\node[rotate=-90] at (0.75*\len,0.25*\len) {\scriptsize $\Arr$};
\node[rotate=180] at (0,0) {\scriptsize $\Arr$};
\node[rotate=180] at (-0.4*\len,0.5*\len) {\scriptsize $\Arr$};
\node[rotate=180] at (0.4*\len,0.5*\len) {\scriptsize $\Arr$};
\filldraw [black] (-0.75*\len,0) circle (1pt);
\filldraw [black] (0.75*\len,0) circle (1pt);
\filldraw [black] (0,0.5*\len) circle (1pt);
\filldraw [black] (-0.75*\len,0.5*\len) circle (1pt);
\filldraw [black] (0.75*\len,0.5*\len) circle (1pt);
% Text Node
\node at (-0.75*\len,0.25*\len)[anchor=east] {$x$};
\node at (0.75*\len,0.25*\len)[anchor=west] {$x$};
\node at (-0.1*\len,0) [anchor=north] {$c$};
\node at (0.4*\len,0.5*\len) [anchor=south] {$a$};
\node at (-0.4*\len,0.5*\len) [anchor=south] {$b$};
\draw (0,0.25*\len) node      {$ {\pi}$};
\end{tikzpicture}
\right)
\overset{\eqref{eq:Rel1}}{\sim}\sum_{t,\tilde{b},z\in\Irr\cx}m\otimes n .\overline{\beta}_{\!j}\otimes
\star_{uz}\circ(\star_{xy})_z\left(\begin{tikzpicture}[baseline={([yshift=-.5ex]current bounding box.center)},xscale=0.9]
\draw[line width=1] 
    (-0.75*\len,0) to (0.75*\len,0)
    (-0.75*\len,0) to (-0.75*\len,0.5*\len)
    (0.75*\len,0) to (0.75*\len,0.5*\len)
    (-0.75*\len,0.5*\len) to (0.75*\len,0.5*\len)
    (0,0.5*\len) to (0,\len)
    (0,\len) to (0.75*\len,\len)
    (0.75*\len,0.5*\len) to (0.75*\len,\len)
    (-0.75*\len,0.5*\len) to (-0.75*\len,\len)
    (0,\len) to (-0.75*\len,\len)
    ;
\node[rotate=180] at (0,0) {\scriptsize $\Arr$};
\node[rotate=-90] at (0,0.75*\len) {\scriptsize $\Arr$};
\node[rotate=180] at (-0.4*\len,\len) {\scriptsize $\Arr$};
\node[rotate=180] at (-0.4*\len,0.5*\len) {\scriptsize $\Arr$};
\node[rotate=180] at (0.4*\len,0.5*\len) {\scriptsize $\Arr$};
\node[rotate=180] at (0.4*\len,\len) {\scriptsize $\Arr$};
\node[rotate=-90] at (0.75*\len,0.25*\len) {\scriptsize $\Arr$};
\node[rotate=-90] at (-0.75*\len,0.25*\len) {\scriptsize $\Arr$};
\node[rotate=-90] at (0.75*\len,0.75*\len) {\scriptsize $\Arr$};
\node[rotate=-90] at (-0.75*\len,0.75*\len) {\scriptsize $\Arr$};
\filldraw [black] (0,\len) circle (1pt);
\filldraw [black] (0.75*\len,\len) circle (1pt);
\filldraw [black] (-0.75*\len,\len) circle (1pt);
\filldraw [black] (-0.75*\len,0) circle (1pt);
\filldraw [black] (0.75*\len,0) circle (1pt);
\filldraw [black] (0,0.5*\len) circle (1pt);
\filldraw [black] (-0.75*\len,0.5*\len) circle (1pt);
\filldraw [black] (0.75*\len,0.5*\len) circle (1pt);
%
% Text Node
% \node at (1.3*\len,0.25*\len) [anchor=west] {$t$};
% \node at (-1.3*\len,0.25*\len) [anchor=east] {$t$};
\node at (0,0.75*\len) [anchor=west] {\footnotesize $u$};
\node at (-0.75*\len,0.75*\len) [anchor=east] {$u$};
\node at (0.75*\len,0.75*\len) [anchor=west] {$u$};
\node at (-0.75*\len,0.25*\len) [anchor=east] {$x$};
\node at (0.75*\len,0.25*\len) [anchor=west] {$x$};
\node at (0,0) [anchor=north] {$c$};
\node at (0.4*\len,0.5*\len) [anchor=north] {$a$};
\node at (0.4*\len,\len) [anchor=south] {$\tilde{a}$};
\node at (-0.4*\len,0.5*\len) [anchor=north] {$b$};
\node at (-0.4*\len,\len) [anchor=south] {$\tilde{b}$};
\node at (0.4*\len,0.75*\len) { {$g$}};
\node at (-0.4*\len,0.75*\len) { {$\beta^j$}};
% \node at (-1*\len,0.6*\len) { {$\overline{\varphi}_{\!i}$}};
% \node at (1.05*\len,0.6*\len) { {$\varphi^i$}};
\draw (0,0.25*\len) node      {$ {\pi}$};
%f
\begin{scope}[yshift=-0.5*\len]
% \node at (-1*\len,0.7*\len) {$\bullet_z$};
% \node at (1.1*\len,0.7*\len) {$\bullet_z$};
% \draw
% (-0.75*\len,0) .. controls (-1.25*\len,0.25*\len) and (-1.25*\len,0.75*\len).. (-0.75*\len,\len);
% \draw
% (0.75*\len,0) .. controls (1.25*\len,0.25*\len) and (1.25*\len,0.75*\len).. (0.75*\len,\len);
\draw[line width=1] 
    (-0.75*\len,0) to (0.75*\len,0)
    (-0.75*\len,0) to (-0.75*\len,0.5*\len)
    (0.75*\len,0) to (0.75*\len,0.5*\len)
    (-0.75*\len,0.5*\len) to (0.75*\len,0.5*\len);
\filldraw [black] (-0.75*\len,0) circle (1pt);
\filldraw [black] (0.75*\len,0) circle (1pt);
\node[rotate=180] at (0,0) {\scriptsize $\Arr$};
\node[rotate=-90] at (0.75*\len,0.25*\len) {\scriptsize $\Arr$};
\node[rotate=-90] at (-0.75*\len,0.25*\len) {\scriptsize $\Arr$};
\node at (0,0) [anchor=north] {$d$};
\node at (-0.75*\len,0.25*\len) [anchor=east] {$y$};
\node at (0.75*\len,0.25*\len) [anchor=west] {$y$};
% \node at (-0.70*\len,0.25*\len) { $\bullet_y$};
% \node at (0.80*\len,0.25*\len) { $\bullet_y$};    
    \draw (0.1*\len,0.20*\len) node      {$ {f}$};
% \node at (-0.9*\len,0.45*\len) { {$\overline{\psi}_{\!k}$}};
% \node at (0.95*\len,0.45*\len) { {$\psi^k$}};
\end{scope}
\end{tikzpicture}
\right).
\end{equation*}
Applying Lemma~\ref{star_prop} (ii) we obtain the expression
\begin{equation*}
\sum_{t,\tilde{b},z\in\Irr\cx}m\otimes n.\overline{\beta}_j\otimes
\star_{zy}\circ(\star_{ux})_z\left[(\beta^j\circ_u g)\circ_{ab}\pi\circ_c f\right]=\sum_{\tilde{b}\in\Irr\cx}m\otimes n.\overline{\beta}_{\!j}\otimes\star_{ux}\left[(\beta^j\circ_u g)\circ_{ab}\pi\right].f    
\end{equation*}
proving the desired result. A similar argument shows the compatibility of the tube algebra action with the relation~\eqref{eq:Rel2}.
\end{proof}
\begin{prop}\label{prop:tensor_prod}
The tensor product from Definition~\ref{def:tensor_prod} extends to a functor
\begin{equation*}
    \Boxt \Colon \Rep\Tubecx\times \Rep\Tubecx\to\Rep\Tubecx
\end{equation*}
given for homomorphisms $F\colon M\to M'$ and $G\colon N\to N'$ between tube algebra representations by the assignment
\begin{equation*}
\begin{aligned}
    F\Boxt G\colon &\qquad M\Boxt N &\to&\; M'\Boxt N'\,,\\
    &m\otimes n\otimes \pi &\mapsto&\; F(m)\otimes G(n)\otimes \pi
\end{aligned} 
\end{equation*}
that, in particular, is a well-defined homomorphism of tube algebra representations.
\end{prop}
\begin{proof}
A routine check shows that $F\Boxt G$ preserves the relations~\eqref{eq:Rel1} and~\eqref{eq:Rel2}.
Since the action~\eqref{eq:tubeaction_onprod} is solely defined on $\pi$ it is clear that $F\Boxt G$ is an intertwiner.
\end{proof}

\begin{prop}\label{prop:rel_unit_simple}
Let $M,N$ be tube algebra representations. 
\begin{enumerate}[\rm (i)]
\item The following relation, obtained by decomposing the label $a$ in simple objects, 
\begin{equation}\label{eq:rel_simple_decomp1}
\begin{aligned}
m\otimes n \otimes\pielem{a}{b}{c}{x}{\pi}
&\quad\sim&
\sum_{t\in\Irr\cx} \rd_t \;m.\overline{\alpha}_i\otimes n \otimes \!\!\begin{tikzpicture}[baseline={([yshift=-.5ex]current bounding box.center)},xscale=0.8]
    \draw[line width=1] 
    (-0.75*\len,0) to (0.75*\len,0)
    (-0.75*\len,0) to (-0.75*\len,0.5*\len)
    (0.75*\len,0) to (0.75*\len,0.5*\len)
    (-0.75*\len,0.5*\len) to (0.75*\len,0.5*\len);

\node[rotate=-90] at (-0.75*\len,0.25*\len) {\scriptsize $\Arr$};
\node[rotate=-90] at (0.75*\len,0.25*\len) {\scriptsize $\Arr$};
\node[rotate=180] at (0,0) {\scriptsize $\Arr$};
\node[rotate=180] at (-0.4*\len, 0.5*\len) {\scriptsize $\Arr$};
\node[rotate=180] at (0.4*\len, 0.5*\len) {\scriptsize $\Arr$};
\node[rotate=180] at (0.4*\len, 0.875*\len) {\scriptsize $\Arr$};

\filldraw [black] (-0.75*\len,0) circle (1pt);
\filldraw [black] (0.75*\len,0) circle (1pt);
\filldraw [black] (0,0.5*\len) circle (1pt);
\filldraw [black] (-0.75*\len,0.5*\len) circle (1pt);
\filldraw [black] (0.75*\len,0.5*\len) circle (1pt);

% Text Node
\node at (-0.75*\len,0.25*\len) [anchor=east] {$x$};
\node at (0.75*\len,0.25*\len) [anchor=west] {$x$};
\node at (0.4*\len,0.5*\len) [anchor=north] {$a$};
\node at (-0.4*\len,0.5*\len) [anchor=south] {$b$};
\node at (0,0) [anchor=north] {$c$};
\node at (0.4*\len,0.875*\len) [anchor=south] {$t$};

\node at (0.4*\len,0.7*\len) {{$\alpha^i$}};

\draw (0,0.25*\len) node {${\pi}$};

\draw[line width=1] (0.75*\len,0.5*\len) arc (0:180:0.375*\len);

\end{tikzpicture}
\end{aligned}
\end{equation}
holds for any $m\otimes n \otimes\pi\in M_a\otimes N_b\otimes \auxprod$.
\item Similarly, the relation
\begin{equation}\label{eq:rel_simple_decomp2}
\begin{aligned}
m\otimes n \otimes\pielem{a}{b}{c}{x}{\pi}
&\quad\sim&
\sum_{s\in\Irr\cx} \rd_s \;m\otimes n.\overline{\beta}_j \otimes \!\!
\begin{tikzpicture}[xscale=0.8,baseline={([yshift=-.5ex]current bounding box.center)}]
    \draw[line width=1] 
    (-0.75*\len,0) to (0.75*\len,0)
    (-0.75*\len,0) to (-0.75*\len,0.5*\len)
    (0.75*\len,0) to (0.75*\len,0.5*\len)
    (-0.75*\len,0.5*\len) to (0.75*\len,0.5*\len);

\node[rotate=-90] at (-0.75*\len,0.25*\len) {\scriptsize $\Arr$};
\node[rotate=-90] at (0.75*\len,0.25*\len) {\scriptsize $\Arr$};
\node[rotate=180] at (0,0) {\scriptsize $\Arr$};
\node[rotate=180] at (-0.4*\len, 0.5*\len) {\scriptsize $\Arr$};
\node[rotate=180] at (0.4*\len, 0.5*\len) {\scriptsize $\Arr$};
\node[rotate=180] at (-0.4*\len, 0.875*\len) {\scriptsize $\Arr$};

\filldraw [black] (-0.75*\len,0) circle (1pt);
\filldraw [black] (0.75*\len,0) circle (1pt);
\filldraw [black] (0,0.5*\len) circle (1pt);
\filldraw [black] (-0.75*\len,0.5*\len) circle (1pt);
\filldraw [black] (0.75*\len,0.5*\len) circle (1pt);

% Text Node
\node at (-0.75*\len,0.25*\len) [anchor=east] {$x$};
\node at (0.75*\len,0.25*\len) [anchor=west] {$x$};
\node at (0.4*\len,0.5*\len) [anchor=south] {$a$};
\node at (-0.4*\len,0.5*\len) [anchor=north] {$b$};
\node at (0,0) [anchor=north] {$c$};
\node at (-0.4*\len,0.875*\len) [anchor=south] {$s$};

\node at (-0.4*\len,0.65*\len) {{$\beta^j$}};

\draw (0,0.25*\len) node  {${\pi}$};

\draw[line width=1]  (0,0.5*\len) arc (0:180:0.375*\len);

\end{tikzpicture}
\end{aligned}
\end{equation}
holds for every $m\otimes n \otimes\pi\in M_a\otimes N_b\otimes \auxprod$.
\item Every vector in $(M\Boxt N)_c$ can be expressed in terms of vectors in  $M_t\otimes N_s \otimes (\Auxprod{t}{s}{c})_\un$ where ${t,s\in\Irr\cx}$ are simple objects. Explicitly, the following relation
\begin{equation}\label{eq:rel_unit}
\rd_x\; m\otimes n\otimes 
\pielem{a}{b}{c}{x}{\pi}
\sim
\sum_{t,s\in\Irr\cx}\rd_{t}\rd_{s}\,m.\overline{\mu}_{i}\otimes n.\overline{\nu}_{j} \otimes 
\begin{tikzpicture}[baseline={([yshift=-.5ex]current bounding box.center)}]

\draw[line width=1]
(-0.75*\len,0.5*\len) .. controls (-0.75*\len,0.75*\len) and (-0.5*\len,\len).. (0,\len);
\draw[line width=1]
(0.75*\len,0.5*\len) .. controls (0.75*\len,0.75*\len) and (0.5*\len,\len).. (0,\len);

\draw[line width=1,dotted] 
    (-0.75*\len,0) to (-0.75*\len,0.5*\len)
    (0.75*\len,0) to (0.75*\len,0.5*\len);
\draw[line width=1] 
    (-0.75*\len,0) to (0.75*\len,0)    
    (-0.75*\len,0.5*\len) to (0.75*\len,0.5*\len)
    (0,0.5*\len) to (0,\len);

\node[rotate=180] at (0,0) {\scriptsize $\Arr$};
\node[rotate=90] at (0,0.75*\len) {\scriptsize $\Arr$};
\node[rotate=180] at (-0.6*\len,0.5*\len) {\scriptsize $\Arr$};
\node[rotate=180] at (-0.2*\len,0.5*\len) {\scriptsize $\Arr$};
\node[rotate=0] at (0.6*\len,0.5*\len) {\scriptsize $\Arr$};
\node[rotate=180] at (0.2*\len,0.5*\len) {\scriptsize $\Arr$};
\node[rotate=215] at (-0.45*\len,0.9*\len) {\scriptsize $\Arr$};
\node[rotate=145] at (0.45*\len,0.9*\len) {\scriptsize $\Arr$};

\filldraw [black] (0,\len) circle (1pt);
\filldraw [black] (-0.75*\len,0) circle (1pt);
\filldraw [black] (0.75*\len,0) circle (1pt);
\filldraw [black] (0,0.5*\len) circle (1pt);
\filldraw [black] (-0.75*\len,0.5*\len) circle (1pt);
\filldraw [black] (0.75*\len,0.5*\len) circle (1pt);
\filldraw [black] (-0.4*\len,0.5*\len) circle (1pt);
\filldraw [black] (0.4*\len,0.5*\len) circle (1pt);

\node at (0,0.75*\len) [anchor=east,font=\footnotesize] {$x$};
\node at (0,0) [anchor=north] {$c$};
\node at (0.6*\len,0.5*\len) [anchor=north] {$x$};
\node at (0.2*\len,0.5*\len) [anchor=north] {$a$};
\node at (0.4*\len,\len) [anchor=south] {$t$};
\node at (-0.2*\len,0.5*\len) [anchor=north] {$b$};
\node at (-0.6*\len,0.5*\len) [anchor=north] {$x$};
\node at (-0.4*\len,\len) [anchor=south] {$s$};

\node at (0.3*\len,0.75*\len) {$\mu^i$};
\node at (-0.35*\len,0.75*\len) {$\nu^j$};

\draw (0,0.25*\len) node   [font=\normalsize]  {${\pi}$};
\end{tikzpicture}
\end{equation}
holds for any $m\otimes n\otimes \pi\in M_a\otimes N_b\otimes \auxprod$.
\end{enumerate}

\end{prop}

\begin{proof}
Relation (i) follows from applying~\eqref{eq:Rel2} to $g=\Id_b$. Similarly, (ii) corresponds to relation~\eqref{eq:Rel1} for $f=\Id_a$. To show (iii), notice that by~\eqref{eq:Rel2} the right hand side of~\eqref{eq:rel_unit} is related to 
\begin{equation*}
\sum_{t,s,\tilde{a}\in\Irr\cx}\rd_{t}\rd_{s}\rd_{\tilde{a}}\,m.\left(\overline{\mu}_{i}\cdot\frac{\overline{\alpha}_{k}}{\rd_x}\right)\otimes n\otimes \;
\star_{x}\left[
\begin{tikzpicture}[baseline={([yshift=-.5ex]current bounding box.center)},yscale=0.75]

\draw[line width=1]
(-0.75*\len,0.5*\len) .. controls (-0.75*\len,0.75*\len) and (-0.5*\len,\len).. (0,\len);
\draw[line width=1]
(0.75*\len,0.5*\len) .. controls (0.75*\len,0.75*\len) and (0.5*\len,\len).. (0,\len);

\draw[line width=1,dotted] 
    (-0.75*\len,0) to (-0.75*\len,0.5*\len)
    (0.75*\len,0) to (0.75*\len,0.5*\len);
\draw[line width=1] 
    (-0.75*\len,1.5*\len) to (-0.75*\len,0.5*\len)
    (0.75*\len,1.5*\len) to (0.75*\len,0.5*\len)
    (-0.75*\len,0) to (0.75*\len,0)    
    (-0.75*\len,0.5*\len) to (0.75*\len,0.5*\len)
    (-0.75*\len,1.5*\len) to (0.75*\len,1.5*\len)
    (0,0.5*\len) to (0,1.5*\len);

\node[rotate=-90] at (-0.75*\len,1*\len) {\scriptsize $\Arr$};
\node[rotate=-90] at (0.75*\len,1*\len) {\scriptsize $\Arr$};

\node[rotate=180] at (0,0) {\scriptsize $\Arr$};
\node[rotate=90] at (0,0.75*\len) {\scriptsize $\Arr$};
\node[rotate=-90] at (0,1.25*\len) {\scriptsize $\Arr$};
\node[rotate=180] at (-0.6*\len,0.5*\len) {\scriptsize $\Arr$};
\node[rotate=180] at (-0.2*\len,0.5*\len) {\scriptsize $\Arr$};
\node[rotate=0] at (0.6*\len,0.5*\len) {\scriptsize $\Arr$};
\node[rotate=180] at (0.2*\len,0.5*\len) {\scriptsize $\Arr$};
\node[rotate=180] at (0.4*\len,1.5*\len) {\scriptsize $\Arr$};
\node[rotate=180] at (-0.4*\len,1.5*\len) {\scriptsize $\Arr$};
\node[rotate=215] at (-0.45*\len,0.9*\len) {\scriptsize $\Arr$};
\node[rotate=145] at (0.45*\len,0.9*\len) {\scriptsize $\Arr$};

\filldraw [black] (0,\len) circle (1pt);
\filldraw [black] (-0.75*\len,0) circle (1pt);
\filldraw [black] (0.75*\len,0) circle (1pt);
\filldraw [black] (0,0.5*\len) circle (1pt);
\filldraw [black] (-0.75*\len,0.5*\len) circle (1pt);
\filldraw [black] (0.75*\len,0.5*\len) circle (1pt);
\filldraw [black] (-0.4*\len,0.5*\len) circle (1pt);
\filldraw [black] (0.4*\len,0.5*\len) circle (1pt);

\node at (0.75*\len,1*\len) [anchor=west] {$x$};
\node at (-0.75*\len,1*\len) [anchor=east] {$x$};
\node at (0,0.75*\len) [anchor=east,font=\scriptsize] {$x$};
\node at (0,1.25*\len) [anchor=east,font=\scriptsize] {$x$};
\node at (0,0) [anchor=north] {$c$};
\node at (0.6*\len,0.5*\len) [anchor=north] {$x$};
\node at (0.2*\len,0.5*\len) [anchor=north] {$a$};
\node at (0.4*\len,1.5*\len) [anchor=south] {$\tilde{a}$};
\node at (-0.4*\len,1.5*\len) [anchor=south] {$b$};
\node at (0.55*\len,0.85*\len) [anchor=south] {$t$};
\node at (-0.2*\len,0.5*\len) [anchor=north] {$b$};
\node at (-0.6*\len,0.5*\len) [anchor=north] {$x$};
\node at (-0.55*\len,0.85*\len) [anchor=south] {$s$};

\node at (0.3*\len,0.75*\len) {$\mu^i$};
\node at (-0.35*\len,0.75*\len) {$\nu^j$};
\node at (0.35*\len,1.25*\len) {$\alpha^k$};
\node at (-0.3*\len,1.25*\len) {$\overline{\nu}_j$};

\draw (0,0.2*\len) node   [font=\normalsize]  {${\pi}$};
\end{tikzpicture}
\right]
\end{equation*}
which in turn, by using Lemma~\ref{star_prop} (iv) and Lemma~\ref{lem:pairing_properties} (ii), gives the expression
\begin{equation}\label{eq:inter_step}
\rd_x^{-1}\sum_{t,\tilde{a}\in\Irr\cx}\rd_{t}\rd_{\tilde{a}}\,m.\left(\overline{\mu}_{i}\cdot\overline{\alpha}_{k}\right)\otimes n\otimes\,
\begin{tikzpicture}[baseline={([yshift=-.5ex]current bounding box.center)},yscale=0.75]

\draw[line width=1]
(0.75*\len,0.5*\len) .. controls (0.75*\len,0.75*\len) and (0.5*\len,\len).. (0,\len);

\draw[line width=1,dotted] 
    (-0.75*\len,0) to (-0.75*\len,0.5*\len)
    (0.75*\len,0) to (0.75*\len,0.5*\len);
\draw[line width=1] 
    (0.75*\len,1.5*\len) to (0.75*\len,0.5*\len)
    (-0.75*\len,0) to (0.75*\len,0)    
    (-0.75*\len,0.5*\len) to (0.75*\len,0.5*\len)
    (0,1.5*\len) to (0.75*\len,1.5*\len)
    (0,0.5*\len) to (0,1.5*\len);

\node[rotate=-90] at (0.75*\len,1*\len) {\scriptsize $\Arr$};
\node[rotate=180] at (0,0) {\scriptsize $\Arr$};
\node[rotate=90] at (0,0.75*\len) {\scriptsize $\Arr$};
\node[rotate=-90] at (0,1.25*\len) {\scriptsize $\Arr$};
\node[rotate=180] at (-0.6*\len,0.5*\len) {\scriptsize $\Arr$};
\node[rotate=180] at (-0.2*\len,0.5*\len) {\scriptsize $\Arr$};
\node[rotate=0] at (0.6*\len,0.5*\len) {\scriptsize $\Arr$};
\node[rotate=180] at (0.2*\len,0.5*\len) {\scriptsize $\Arr$};
\node[rotate=180] at (0.4*\len,1.5*\len) {\scriptsize $\Arr$};
\node[rotate=145] at (0.45*\len,0.9*\len) {\scriptsize $\Arr$};

\filldraw [black] (0,\len) circle (1pt);
\filldraw [black] (-0.75*\len,0) circle (1pt);
\filldraw [black] (0.75*\len,0) circle (1pt);
\filldraw [black] (0,0.5*\len) circle (1pt);
\filldraw [black] (-0.75*\len,0.5*\len) circle (1pt);
\filldraw [black] (0.75*\len,0.5*\len) circle (1pt);
\filldraw [black] (-0.4*\len,0.5*\len) circle (1pt);
\filldraw [black] (0.4*\len,0.5*\len) circle (1pt);

\node at (0.75*\len,1*\len) [anchor=west] {$x$};
\node at (0,0.75*\len) [anchor=east,font=\scriptsize] {$x$};
\node at (0,1.25*\len) [anchor=east,font=\scriptsize] {$x$};
\node at (0,0) [anchor=north] {$c$};
\node at (0.6*\len,0.5*\len) [anchor=north] {$x$};
\node at (0.2*\len,0.5*\len) [anchor=north] {$a$};
\node at (0.4*\len,1.5*\len) [anchor=south] {$\tilde{a}$};
\node at (0.55*\len,0.85*\len) [anchor=south] {$t$};
\node at (-0.2*\len,0.5*\len) [anchor=north] {$b$};
\node at (-0.6*\len,0.5*\len) [anchor=north] {$x$};

\node at (0.3*\len,0.72*\len) {$\mu^i$};
\node at (0.3*\len,1.25*\len) {$\alpha^k$};
\node at (-0.3*\len,1*\len) {$\Id_x$};

\draw (0,0.2*\len) node  {${\pi}$};

\draw[dotted] (0,1.5*\len) arc (-90:90:-0.5*\len);
\end{tikzpicture}
\;.
\end{equation}
Now, since the product of tube algebra elements is $\overline{\mu}_{i}\cdot\overline{\alpha}_{k}=\star_{\overline{x}x}(\overline{\mu}_{i}\circ_t\overline{\alpha}_{k})$, we obtain by Lemma~\ref{star_prop} (v) that
\begin{equation*}
\sum_{t\in\Irr\cx}\rd_{t}  \star_{\overline{x}x}(\overline{\mu}_{i}\circ_t\overline{\alpha}_{k}) \otimes \alpha^k\circ_t\mu^i\circ_{xx}\Id_x=\rd_x^2 \rd_x^{-1}\langle \Id_x,\Id_x\rangle\, \overline{\eta}_i\otimes \eta^i\circ_\un \Id_x
\end{equation*}
where $(\overline{\eta},\eta)$ is a pair of dual bases of $\Hsp{\overline{a}\,\tilde{a}}$.
Replacing this identity in~\eqref{eq:inter_step} we get the expression
%\vspace{-0.25cm}
\begin{equation*}
\rd_x\sum_{\tilde{a}\in\Irr\cx}\rd_{\tilde{a}}\,m.\overline{\eta}_{i}\otimes n\otimes
\begin{tikzpicture}[xscale=0.8,yscale=0.9,baseline={([yshift=-.5ex]current bounding box.center)}]
        \draw[line width=1] 
    (-0.75*\len,0) to (0.75*\len,0)
    (-0.75*\len,0) to (-0.75*\len,0.5*\len)
    (0.75*\len,0) to (0.75*\len,0.5*\len)
    (-0.75*\len,0.5*\len) to (0.75*\len,0.5*\len);

\node[rotate=-90] at (-0.75*\len,0.25*\len) {\scriptsize $\Arr$};
\node[rotate=-90] at (0.75*\len,0.25*\len) {\scriptsize $\Arr$};
\node[rotate=180] at (0,0) {\scriptsize $\Arr$};
\node[rotate=180] at (-0.4*\len, 0.5*\len) {\scriptsize $\Arr$};
\node[rotate=180] at (0.4*\len, 0.5*\len) {\scriptsize $\Arr$};
\node[rotate=180] at (0.4*\len, 0.875*\len) {\scriptsize $\Arr$};

\filldraw [black] (-0.75*\len,0) circle (1pt);
\filldraw [black] (0.75*\len,0) circle (1pt);
\filldraw [black] (0,0.5*\len) circle (1pt);
\filldraw [black] (-0.75*\len,0.5*\len) circle (1pt);
\filldraw [black] (0.75*\len,0.5*\len) circle (1pt);

% Text Node
\node at (-0.75*\len,0.25*\len) [anchor=east] {$x$};
\node at (0.75*\len,0.25*\len) [anchor=west] {$x$};
\node at (0.4*\len,0.5*\len) [anchor=north] {$a$};
\node at (-0.4*\len,0.5*\len) [anchor=south] {$b$};
\node at (0,0) [anchor=north] {$c$};
\node at (0.4*\len,0.875*\len) [anchor=south] {$\tilde{a}$};
\node at (0.4*\len,0.68*\len) {{$\eta^i$}};
\draw (0,0.25*\len) node   [font=\normalsize]  {${\pi}$};
\draw[line width=1]  (0.75*\len,0.5*\len) arc (0:180:0.375*\len);
\end{tikzpicture}
\end{equation*}
which is related to the left hand side of~\eqref{eq:rel_unit} by~\eqref{eq:rel_simple_decomp1}.
\end{proof}

\subsection{Associators and pentagon axiom}
The purpose of this section is to define the associators for the tensor product $M\Boxt N$ of tube algebra representations presented in Section~\ref{sec:tensor_prod}.
\begin{defi}\label{def:associator}
Given tube algebra representations $M,N,L\in\Rep\Tubecx$, the \textit{associator}
\begin{equation}\label{eq:associator}
    \mathrm{A}_{M,N,L}\Colon (M\Boxt N) \Boxt L \to M\Boxt (N \Boxt L)
\end{equation}
is defined as the assignment given by
\def\len{1.2cm}
\begin{equation*}
    (m\otimes n\otimes \pi)\otimes l\otimes\rho\mapsto
\sum_{\tilde{d}\in\Irr\cx}\rd_{\tilde{d}}\; m\otimes \left(n \otimes l.\,\frac{\overline{\gamma}_k}{\rd_y} \otimes \Id_{b\tilde{d}}\right)
\otimes \star_{yx}\left[
\begin{tikzpicture}[baseline={([yshift=-.5ex]current bounding box.center)}]
    \draw[line width=1] 
    (-1.2*\len,0) to (1.2*\len,0)
    (-1.2*\len,0) to (-1.2*\len,0.75*\len)
    (1.2*\len,0) to (1.2*\len,0.75*\len)
    (1.2*\len,0.75*\len) to (1.2*\len,1.5*\len)
    (-0.4*\len,0.75*\len) to (-0.4*\len,1.5*\len)
    (-1.2*\len,0.75*\len) to (-1.2*\len,1.5*\len)
    (-1.2*\len,1.5*\len) to (1.2*\len,1.5*\len)
    (-1.2*\len,0.75*\len) to (1.2*\len,0.75*\len);

\node[rotate=-90] at (-0.4*\len,1.125*\len) {\scriptsize $\Arr$};
\node[rotate=-90] at (-1.2*\len,1.125*\len) {\scriptsize $\Arr$};
\node[rotate=-90] at (1.2*\len,1.125*\len) {\scriptsize $\Arr$};
\node[rotate=-90] at (-1.2*\len,0.375*\len) {\scriptsize $\Arr$};
\node[rotate=-90] at (1.2*\len,0.375*\len) {\scriptsize $\Arr$};
\node[rotate=180] at (0,0) {\scriptsize $\Arr$};
\node[rotate=180] at (-0.8*\len,0.75*\len) {\scriptsize $\Arr$};
\node[rotate=180] at (0.4*\len,0.75*\len) {\scriptsize $\Arr$};
\node[rotate=180] at (0.8*\len,1.5*\len) {\scriptsize $\Arr$};
\node[rotate=180] at (0,1.5*\len) {\scriptsize $\Arr$};
\node[rotate=180] at (-0.8*\len,1.5*\len) {\scriptsize $\Arr$};

\filldraw [black] (-1.2*\len,0) circle (1pt);
\filldraw [black] (1.2*\len,0) circle (1pt);
\filldraw [black] (-0.4*\len,0.75*\len) circle (1pt);
\filldraw [black] (-1.2*\len,0.75*\len) circle (1pt);
\filldraw [black] (1.2*\len,0.75*\len) circle (1pt);
\filldraw [black] (1.2*\len,1.5*\len) circle (1pt);
\filldraw [black] (0.4*\len,1.5*\len) circle (1pt);
\filldraw [black] (-1.2*\len,1.5*\len) circle (1pt);
\filldraw [black] (-0.4*\len,1.5*\len) circle (1pt);

\node at (-0.4*\len,1.125*\len) [anchor=west] {$y$};

% Text Node
\node at (-1.2*\len,1.125*\len) [anchor=east] {$y$};
\node at (1.2*\len,1.125*\len)  [anchor=west]{$y$};
\node at (-1.2*\len,0.375*\len) [anchor=east] {$x$};
\node at (1.2*\len,0.375*\len)  [anchor=west]{$x$};
\node at (0,0) [anchor=north] {$r$};
\node at (0.4*\len,0.75*\len) [anchor=north] {$c$};
\node at (0.8*\len,1.5*\len) [anchor=south] {$a$};
\node at (0,1.5*\len) [anchor=south] {$b$};
\node at (-0.8*\len,1.5*\len) [anchor=south] {$\tilde{d}$};
\node at (-0.8*\len,0.75*\len) [anchor=north] {$d$};
\draw (0*\len,0.4*\len) node {${\rho}$};
\draw (-0.8*\len,1.125*\len) node  {{$\gamma^k$}};
\node at (0.4*\len,1.125*\len) {{$\pi$}};
%Identity
% \draw[black] (0.4*\len,1.5*\len) arc (0:180:0.8*\len);
% \draw (-0.4*\len,1.9*\len) node {${\Id_{b\tilde{d}}}$};
% \node at (-0.4*\len,2.3*\len) [anchor=south] {$b\tilde{d}$};
% \node[rotate=180] at (-0.4*\len,2.3*\len) {\scriptsize $\Arr$};
\end{tikzpicture}
\right]
\end{equation*}
for $\pi\in \auxprod$ and $\rho\in \Auxprod{c}{d}{r}$.
% \otimes L_d\otimes.m\otimes n\otimes )M_a\otimes N_b\otimes \otimes l\otimes
\end{defi}

\begin{prop}\label{prop:associators_well-defined}
$\mathrm{A}_{M,N,L}$ is a well-defined natural isomorphism of $\Tubecx$-representations.
\end{prop}
\begin{proof}
Naturality is immediate from the definition. 
It follows from Proposition~\ref{prop:associativity} and the definition of the tube algebra action~\eqref{eq:tubeaction_onprod} that $\mathrm{A}_{M,N,L}$ is an intertwiner. To check that the associator is well-defined under relation~\eqref{eq:Rel1}, notice that on one hand
\((m\otimes n \otimes \pi.f)\otimes l\otimes \rho \in M_a\otimes N_b \otimes \Auxprod{a}{b}{c}\otimes L_d\otimes \Auxprod{c}{d}{r}\)
gets assigned by $\mathrm{A}_{M,N,L}$ the value
\def\len{1.4cm}
\begin{equation}\label{eq:omega1}
\sum_{\tilde{d},u\in\Irr\cx}\rd_{\tilde{d}}\rd_{yz}\, m\otimes \left(n \otimes l.\,\frac{\overline{\gamma}_k}{\rd_u} \otimes \Id_{b\tilde{d}}\right)
\otimes \star_{ux}\left[
\begin{tikzpicture}[baseline={([yshift=-.5ex]current bounding box.center)},xscale=1.5]
    \draw[line width=1] 
    (-1.2*\len,0) to (1.2*\len,0)
    (-1.2*\len,0) to (-1.2*\len,0.5*\len)
    (1.2*\len,0) to (1.2*\len,0.5*\len)
    (1.2*\len,0.5*\len) to (1.2*\len,1.5*\len)
    (-0.4*\len,0.5*\len) to (-0.4*\len,1.5*\len)
    (-1.2*\len,0.5*\len) to (-1.2*\len,1.5*\len)
    (-1.2*\len,1.5*\len) to (1.2*\len,1.5*\len)
    (-1.2*\len,0.5*\len) to (1.2*\len,0.5*\len)
    (0,\len) to (-0.4*\len,1.5*\len)
    (0,\len) to (-0.4*\len,0.5*\len)
    (1.2*\len,1.5*\len) to (0.8*\len,\len)
    (1.2*\len,0.5*\len) to (0.8*\len,\len)
    (-0*\len,\len) to (0.8*\len,\len);
    
\node[rotate=-90] at (-0.4*\len,\len) {\scriptsize $\Arr$};
\node[rotate=-45] at (-0.2*\len,1.25*\len) {\scriptsize $\Arr$};
\node[rotate=-135] at (\len,1.25*\len) {\scriptsize $\Arr$};
\node[rotate=-135] at (-0.2*\len,0.75*\len) {\scriptsize $\Arr$};
\node[rotate=-45] at (\len,0.75*\len) {\scriptsize $\Arr$};
\node[rotate=-90] at (-1.2*\len,\len) {\scriptsize $\Arr$};
\node[rotate=-90] at (1.2*\len,\len) {\scriptsize $\Arr$};
\node[rotate=-90] at (-1.2*\len,0.25*\len) {\scriptsize $\Arr$};
\node[rotate=-90] at (1.2*\len,0.25*\len) {\scriptsize $\Arr$};
\node[rotate=180] at (0,0) {\scriptsize $\Arr$};
\node[rotate=180] at (-0.8*\len,0.5*\len) {\scriptsize $\Arr$};
\node[rotate=180] at (0.4*\len,0.5*\len) {\scriptsize $\Arr$};
\node[rotate=180] at (0.4*\len,\len) {\scriptsize $\Arr$};
\node[rotate=180] at (0.8*\len,1.5*\len) {\scriptsize $\Arr$};
\node[rotate=180] at (0,1.5*\len) {\scriptsize $\Arr$};
\node[rotate=180] at (-0.8*\len,1.5*\len) {\scriptsize $\Arr$};

\filldraw [black] (-1.2*\len,0) circle (1pt);
\filldraw [black] (1.2*\len,0) circle (1pt);
\filldraw [black] (0,\len) circle (1pt);
\filldraw [black] (0.8*\len,\len) circle (1pt);
\filldraw [black] (-0.4*\len,0.5*\len) circle (1pt);
\filldraw [black] (-1.2*\len,0.5*\len) circle (1pt);
\filldraw [black] (1.2*\len,0.5*\len) circle (1pt);
\filldraw [black] (1.2*\len,1.5*\len) circle (1pt);
\filldraw [black] (0.4*\len,1.5*\len) circle (1pt);
\filldraw [black] (-1.2*\len,1.5*\len) circle (1pt);
\filldraw [black] (-0.4*\len,1.5*\len) circle (1pt);

\node at (-0.4*\len,\len) [anchor=east] {$u$};

% % Text Node
\node at (-1.2*\len,\len) [anchor=east] {$u$};
\node at (1.2*\len,\len)  [anchor=west]{$u$};
\node at (-0.2*\len,1.25*\len) [anchor=west,font=\footnotesize] {$y$};
\node at (\len,1.25*\len)  [anchor=east,font=\footnotesize]{$y$};
\node at (-0.2*\len,0.75*\len) [anchor=west,font=\footnotesize] {$z$};
\node at (\len,0.75*\len)  [anchor=east,font=\footnotesize]{$z$};
\node at (-1.2*\len,0.25*\len) [anchor=east] {$x$};
\node at (1.2*\len,0.25*\len)  [anchor=west]{$x$};
\node at (0,0) [anchor=north] {$r$};
\node at (0.25*\len,0.5*\len) [anchor=north,font=\footnotesize] {$\tilde{c}$};
\node at (0.25*\len,\len) [anchor=north,font=\footnotesize] {$c$};
\node at (0.8*\len,1.5*\len) [anchor=south] {$a$};
\node at (0,1.5*\len) [anchor=south] {$b$};
\node at (-0.8*\len,1.5*\len) [anchor=south] {$\tilde{d}$};
\node at (-0.8*\len,0.5*\len) [anchor=south,font=\footnotesize] {$d$};
\draw (-0.15*\len,0.25*\len) node {${\rho}$};
\draw (-0.8*\len,1.1*\len) node  {{$\gamma^k$}};
\draw (1.05*\len,\len) node [font=\small]{{$\nu^i$}};
\draw (-0.25*\len,\len) node [font=\small]{{$\overline{\nu}_i$}};
\node at (0.4*\len,1.25*\len) {{$\pi$}};
\node at (0.5*\len,0.75*\len) {{$f$}};
\end{tikzpicture}
\right]\,.
\end{equation}
On the other hand, the right-hand side of relation~\eqref{eq:Rel1}
% $
% \sum_{{t}\in\Irr\cx}\rd_{{t}}\,(m\otimes n\otimes \pi)\otimes l.\overline{\beta}_j\otimes  \star_{zx}\left[\left(\,\beta^j\circ_z f\right)\circ_{\tilde{c}d}\rho\right]
% $
% \end{equation*}
is mapped by $\mathrm{A}_{M,N,L}$ to
\begin{equation}\label{eq:omega2}
\sum_{\tilde{d},{t}\in\Irr\cx}\rd_{\tilde{d}}\rd_{{t}}\, m\otimes\left( n \otimes l.\frac{\overline{\beta}_j\cdot \overline{\alpha}_i}{\rd_z\rd_y}\otimes\Id_{b\tilde{d}}\right)
\otimes \star_{yzx}\left[
\begin{tikzpicture}[baseline={([yshift=-.5ex]current bounding box.center)},xscale=1.2,yscale=1.2]
    \draw[line width=1] 
    (-1.2*\len,0) to (1.2*\len,0)
    (-1.2*\len,0) to (-1.2*\len,0.5*\len)
    (1.2*\len,0) to (1.2*\len,0.5*\len)
    (1.2*\len,0.5*\len) to (1.2*\len,1.5*\len)
    (-0.4*\len,0.5*\len) to (-0.4*\len,1.5*\len)
    (-1.2*\len,0.5*\len) to (-1.2*\len,1.5*\len)
    (-1.2*\len,1.5*\len) to (1.2*\len,1.5*\len)
    (-1.2*\len,0.5*\len) to (1.2*\len,0.5*\len)
    (-1.2*\len,\len) to (-0.4*\len,\len)
    (-0.4*\len,\len) to (1.2*\len,\len);

\node[rotate=-90] at (-0.4*\len,1.25*\len) {\scriptsize $\Arr$};
\node[rotate=-90] at (-0.4*\len,0.75*\len) {\scriptsize $\Arr$};
\node[rotate=-90] at (-1.2*\len,0.75*\len) {\scriptsize $\Arr$};
\node[rotate=-90] at (-1.2*\len,1.25*\len) {\scriptsize $\Arr$};
\node[rotate=-90] at (1.2*\len,0.75*\len) {\scriptsize $\Arr$};
\node[rotate=-90] at (1.2*\len,1.25*\len) {\scriptsize $\Arr$};

\node[rotate=-90] at (-1.2*\len,0.25*\len) {\scriptsize $\Arr$};
\node[rotate=-90] at (1.2*\len,0.25*\len) {\scriptsize $\Arr$};
\node[rotate=180] at (0,0) {\scriptsize $\Arr$};
\node[rotate=180] at (-0.8*\len,0.5*\len) {\scriptsize $\Arr$};
\node[rotate=180] at (-0.8*\len,\len) {\scriptsize $\Arr$};
\node[rotate=180] at (0.4*\len,0.5*\len) {\scriptsize $\Arr$};
\node[rotate=180] at (0.4*\len,\len) {\scriptsize $\Arr$};
\node[rotate=180] at (0.8*\len,1.5*\len) {\scriptsize $\Arr$};
\node[rotate=180] at (0,1.5*\len) {\scriptsize $\Arr$};
\node[rotate=180] at (-0.8*\len,1.5*\len) {\scriptsize $\Arr$};

\filldraw [black] (-1.2*\len,0) circle (1pt);
\filldraw [black] (1.2*\len,0) circle (1pt);
\filldraw [black] (-0.4*\len,\len) circle (1pt);
\filldraw [black] (-1.2*\len,\len) circle (1pt);
\filldraw [black] (1.2*\len,\len) circle (1pt);
\filldraw [black] (-0.4*\len,0.5*\len) circle (1pt);
\filldraw [black] (-1.2*\len,0.5*\len) circle (1pt);
\filldraw [black] (1.2*\len,0.5*\len) circle (1pt);
\filldraw [black] (1.2*\len,1.5*\len) circle (1pt);
\filldraw [black] (0.4*\len,1.5*\len) circle (1pt);
\filldraw [black] (-1.2*\len,1.5*\len) circle (1pt);
\filldraw [black] (-0.4*\len,1.5*\len) circle (1pt);

% % Text Node
\node at (-1.2*\len,1.25*\len)  [anchor=east]{$y$};
\node at (-1.2*\len,0.75*\len) [anchor=east] {$z$};
\node at (-0.4*\len,1.25*\len) [anchor=west,font=\footnotesize] {$y$};
\node at (1.2*\len,1.25*\len)  [anchor=west]{$y$};
\node at (-0.4*\len,0.75*\len) [anchor=west,font=\footnotesize] {$z$};
\node at (1.2*\len,0.75*\len)  [anchor=west]{$z$};
\node at (-1.2*\len,0.25*\len) [anchor=east] {$x$};
\node at (1.2*\len,0.25*\len)  [anchor=west]{$x$};
\node at (0,0) [anchor=north] {$r$};
\node at (0.25*\len,0.5*\len) [anchor=north,font=\footnotesize] {$\tilde{c}$};
\node at (0.25*\len,\len) [anchor=north,font=\footnotesize] {$c$};
\node at (0.8*\len,1.5*\len) [anchor=south] {$a$};
\node at (0,1.5*\len) [anchor=south] {$b$};
\node at (-0.8*\len,1.5*\len) [anchor=south] {$\tilde{d}$};
\node at (-0.65*\len,\len) [anchor=south,font=\scriptsize] {${t}$};
\node at (-0.8*\len,0.5*\len) [anchor=north,font=\footnotesize] {$d$};
\draw (-0.15*\len,0.25*\len) node {${\rho}$};
\draw (-0.85*\len,1.3*\len) node  {{$\alpha^i$}};
\draw (-0.85*\len,0.75*\len) node {{$\beta^j$}};
\node at (0.4*\len,1.25*\len) {{$\pi$}};
\node at (0.5*\len,0.75*\len) {{$f$}};
\end{tikzpicture}
\right]\,.
\end{equation}
It follows from Lemma~\ref{bases_comp} and~\ref{lem:base_change} that
\begin{equation*}
\sum_{{t}\in\Irr\cx}\rd_{{t}}\,  (\overline{\beta}_j\cdot \overline{\alpha}_i) \otimes \alpha^i\circ_{{t}}\beta^j=\rd_z\rd_y\sum_{u\in\Irr\cx} \overline{\gamma}_k\otimes \overline{\nu_i}\circ_u\gamma^k \circ_{\overline{u}}\nu_i
\end{equation*}
where $(\overline{\gamma},\gamma)$ is a pair of dual bases of $\Hsp{\overline{d}\,\overline{u}\,\tilde{d} u}$ and $(\overline{\nu},\nu)$ of $\Hsp{\overline{z}\,\overline{y}\,u}$. Applying this formula and Lemma~\ref{star_prop} (ii) to the expression~\eqref{eq:omega2} we obtain~\eqref{eq:omega1}. That the associator is well-defined under relation~\eqref{eq:Rel2}  can be verified using analogous arguments.

Lastly, the assignment given for $\pi\in\Auxprod{b}{c}{d}$ and $\rho\in \Auxprod{a}{d}{r}$ by the value
\def\len{1.2cm}
\begin{equation*}
    m\otimes( n \otimes l\otimes \pi)\otimes\rho \mapsto\sum_{\tilde{a}\in\Irr\cx}\rd_{\tilde{a}}\, \left(m.\frac{\overline{\alpha}_i}{\rd_y}\otimes n\otimes \Id_{\tilde{a}b} \right)\otimes l\otimes
    \star_{yx}\left[
\begin{tikzpicture}[baseline={([yshift=-.5ex]current bounding box.center)}]
    \draw[line width=1] 
    (-1.2*\len,0) to (1.2*\len,0)
    (-1.2*\len,0) to (-1.2*\len,0.75*\len)
    (1.2*\len,0) to (1.2*\len,0.75*\len)
    (1.2*\len,0.75*\len) to (1.2*\len,1.5*\len)
    (0.4*\len,0.75*\len) to (0.4*\len,1.5*\len)
    (-1.2*\len,0.75*\len) to (-1.2*\len,1.5*\len)
    (-1.2*\len,1.5*\len) to (1.2*\len,1.5*\len)
    (-1.2*\len,0.75*\len) to (1.2*\len,0.75*\len);

\node[rotate=-90] at (0.4*\len,1.125*\len) {\scriptsize $\Arr$};
\node[rotate=-90] at (-1.2*\len,1.125*\len) {\scriptsize $\Arr$};
\node[rotate=-90] at (1.2*\len,1.125*\len) {\scriptsize $\Arr$};
\node[rotate=-90] at (-1.2*\len,0.375*\len) {\scriptsize $\Arr$};
\node[rotate=-90] at (1.2*\len,0.375*\len) {\scriptsize $\Arr$};
\node[rotate=180] at (0,0) {\scriptsize $\Arr$};
\node[rotate=180] at (-0.4*\len,0.75*\len) {\scriptsize $\Arr$};
\node[rotate=180] at (0.8*\len,0.75*\len) {\scriptsize $\Arr$};
\node[rotate=180] at (0.8*\len,1.5*\len) {\scriptsize $\Arr$};
\node[rotate=180] at (0,1.5*\len) {\scriptsize $\Arr$};
\node[rotate=180] at (-0.8*\len,1.5*\len) {\scriptsize $\Arr$};

\filldraw [black] (-1.2*\len,0) circle (1pt);
\filldraw [black] (1.2*\len,0) circle (1pt);
\filldraw [black] (0.4*\len,0.75*\len) circle (1pt);
\filldraw [black] (-1.2*\len,0.75*\len) circle (1pt);
\filldraw [black] (1.2*\len,0.75*\len) circle (1pt);
\filldraw [black] (1.2*\len,1.5*\len) circle (1pt);
\filldraw [black] (0.4*\len,1.5*\len) circle (1pt);
\filldraw [black] (-1.2*\len,1.5*\len) circle (1pt);
\filldraw [black] (-0.4*\len,1.5*\len) circle (1pt);

\node at (0.4*\len,1.125*\len) [anchor=east] {$y$};

% Text Node
\node at (-1.2*\len,1.125*\len) [anchor=east] {$y$};
\node at (1.2*\len,1.125*\len)  [anchor=west]{$y$};
\node at (-1.2*\len,0.375*\len) [anchor=east] {$x$};
\node at (1.2*\len,0.375*\len)  [anchor=west]{$x$};
\node at (0,0) [anchor=north] {$r$};
\node at (0.8*\len,0.75*\len) [anchor=north] {$a$};
\node at (0.8*\len,1.5*\len) [anchor=south] {$\tilde{a}$};
\node at (0,1.5*\len) [anchor=south] {$b$};
\node at (-0.8*\len,1.5*\len) [anchor=south] {$c$};
\node at (-0.4*\len,0.75*\len) [anchor=north] {$d$};

\draw (0*\len,0.3*\len) node {${\rho}$};
\draw (0.8*\len,1.125*\len) node  {{$\alpha^i$}};
\node at (-0.4*\len,1.125*\len) {{$\pi$}};
\end{tikzpicture}
\right]\,.
\end{equation*}
provides the inverse $\mathrm{A}^{-1}_{M,N,L}\colon M\Boxt (N \Boxt L)\to  (M\Boxt N) \Boxt L$ for the associator. Indeed, using relation~\eqref{eq:rel_simple_decomp1}, we obtain that the image of a vector $(m\otimes n\otimes \pi)\otimes l\otimes\rho$ under the composition  $\mathrm{A}^{-1}_{M,N,L}\circ \mathrm{A}_{M,N,L}$ is given by the expression
\begin{equation*}
\sum_{\tilde{a},\tilde{d}\in\Irr\cx}\rd_{\tilde{a}}\rd_{\tilde{d}}\; \left(m\otimes n \otimes \overline{\alpha_i}\right)\otimes l.\,\frac{\overline{\gamma}_k}{\rd_y} 
\otimes \star_{yx}\left[
\begin{tikzpicture}[baseline={([yshift=-.5ex]current bounding box.center)}]
\draw[line width=1] (1.2*\len,1.5*\len) arc (0:180:0.8*\len) node[midway,rotate=180]{\scriptsize $\Arr$} node[midway,above]{ $\tilde{a}$};
    \draw[line width=1] 
    (-1.2*\len,0) to (1.2*\len,0)
    (-1.2*\len,0) to (-1.2*\len,0.75*\len)
    (1.2*\len,0) to (1.2*\len,0.75*\len)
    (1.2*\len,0.75*\len) to (1.2*\len,1.5*\len)
    (-0.4*\len,0.75*\len) to (-0.4*\len,1.5*\len)
    (-1.2*\len,0.75*\len) to (-1.2*\len,1.5*\len)
    (-1.2*\len,1.5*\len) to (1.2*\len,1.5*\len)
    (-1.2*\len,0.75*\len) to (1.2*\len,0.75*\len);

\node[rotate=-90] at (-0.4*\len,1.125*\len) {\scriptsize $\Arr$};
\node[rotate=-90] at (-1.2*\len,1.125*\len) {\scriptsize $\Arr$};
\node[rotate=-90] at (1.2*\len,1.125*\len) {\scriptsize $\Arr$};
\node[rotate=-90] at (-1.2*\len,0.375*\len) {\scriptsize $\Arr$};
\node[rotate=-90] at (1.2*\len,0.375*\len) {\scriptsize $\Arr$};
\node[rotate=180] at (0,0) {\scriptsize $\Arr$};
\node[rotate=180] at (-0.8*\len,0.75*\len) {\scriptsize $\Arr$};
\node[rotate=180] at (0.4*\len,0.75*\len) {\scriptsize $\Arr$};
\node[rotate=180] at (0.8*\len,1.5*\len) {\scriptsize $\Arr$};
\node[rotate=180] at (0,1.5*\len) {\scriptsize $\Arr$};
\node[rotate=180] at (-0.8*\len,1.5*\len) {\scriptsize $\Arr$};

\filldraw [black] (-1.2*\len,0) circle (1pt);
\filldraw [black] (1.2*\len,0) circle (1pt);
\filldraw [black] (-0.4*\len,0.75*\len) circle (1pt);
\filldraw [black] (-1.2*\len,0.75*\len) circle (1pt);
\filldraw [black] (1.2*\len,0.75*\len) circle (1pt);
\filldraw [black] (1.2*\len,1.5*\len) circle (1pt);
\filldraw [black] (0.4*\len,1.5*\len) circle (1pt);
\filldraw [black] (-1.2*\len,1.5*\len) circle (1pt);
\filldraw [black] (-0.4*\len,1.5*\len) circle (1pt);

\node at (-0.4*\len,1.125*\len) [anchor=west] {$y$};

% Text Node
\node at (-1.2*\len,1.125*\len) [anchor=east] {$y$};
\node at (1.2*\len,1.125*\len)  [anchor=west]{$y$};
\node at (-1.2*\len,0.375*\len) [anchor=east] {$x$};
\node at (1.2*\len,0.375*\len)  [anchor=west]{$x$};
\node at (0,0) [anchor=north] {$r$};
\node at (0.4*\len,0.75*\len) [anchor=north] {$c$};
\node at (0.8*\len,1.5*\len) [anchor=south] {$a$};
\node at (0,1.5*\len) [anchor=south] {$b$};
\node at (-0.8*\len,1.5*\len) [anchor=south] {$\tilde{d}$};
\node at (-0.8*\len,0.75*\len) [anchor=north] {$d$};
\draw (0*\len,0.4*\len) node {${\rho}$};
\draw (-0.8*\len,1.125*\len) node  {{$\gamma^k$}};
\node at (0.4*\len,1.125*\len) {{$\pi$}};
\node at (0.4*\len,2*\len) {{$\alpha^i$}};
\end{tikzpicture}
\right]
\end{equation*}
which can be further simplified to  $(m\otimes n\otimes \pi)\otimes l\otimes\rho$ by transporting $\pi$ according to~\eqref{eq:base_change} and applying relation~\eqref{eq:Rel2} in the representation $(M\Boxt N)\Boxt L$. That
$\mathrm{A}_{M,N,L}\circ \mathrm{A}_{M,N,L}^{-1}= \id_{M\Boxt (N \Boxt L)}$ can be verified with similar arguments.
\end{proof}

\def\len{1.4cm}
\begin{lem}\label{lem:associators_unit} 
Let $M,N,L$ be tube algebra representations. Then, it holds that
\[
\mathrm{A}_{M,N,L}\left[\left(
m\otimes n\otimes
\begin{tikzpicture}[xscale=0.8,baseline={([yshift=-.5ex]current bounding box.center)}]
    \draw[line width=1] 
    (-0.75*\len,0) to (0.75*\len,0)
    (-0.75*\len,0.5*\len) to (0.75*\len,0.5*\len);
    \draw[line width=1,dotted] 
    (-0.75*\len,0) to (-0.75*\len,0.5*\len)
    (0.75*\len,0) to (0.75*\len,0.5*\len);

\node[rotate=180] at (0,0) {\scriptsize $\Arr$};
\node[rotate=180] at (-0.4*\len,0.5*\len) {\scriptsize $\Arr$};
\node[rotate=180] at (0.4*\len,0.5*\len) {\scriptsize $\Arr$};

\filldraw [black] (-0.75*\len,0) circle (1pt);
\filldraw [black] (0.75*\len,0) circle (1pt);
\filldraw [black] (0,0.5*\len) circle (1pt);
\filldraw [black] (-0.75*\len,0.5*\len) circle (1pt);
\filldraw [black] (0.75*\len,0.5*\len) circle (1pt);

\node at (0,0) [anchor=north] {$d$};
\node at (0.4*\len,0.5*\len) [anchor=south] {$a$};
\node at (-0.4*\len,0.5*\len) [anchor=south] {$b$};

\draw (0,0.25*\len) node {${\pi}$};
\end{tikzpicture}
\right)\otimes l\otimes \rho\right]
=
m\otimes\left(n\otimes l\otimes \Id_{bc}\right)\otimes \pi\circ_d\rho\]   
\end{lem}
\begin{proof}
Directly from the definition of the associator~\eqref{eq:associator}, we obtain an expression that can be further simplified, by means of the relation~\eqref{eq:rel_simple_decomp2}, into the desired result.
\end{proof}
\begin{prop}\label{prop:pentagon_axiom}
The associators $\mathrm{A}_{M,N,L}$ defined in~\eqref{eq:associator} satisfy the pentagon axiom
\begin{equation*}
\begin{tikzcd}
    &((M\Boxt N)\Boxt L)\Boxt K&\\
        (M\Boxt(N\Boxt L))\Boxt K&&(M\Boxt N)\Boxt (L\Boxt K)\\
        M\Boxt ((N\Boxt L)\Boxt K)&&M\Boxt (N\Boxt (L\Boxt K))
\arrow["\mathrm{A}_{M,N,L}\Boxt \id_K", from=1-2,to=2-1,swap]
\arrow["\mathrm{A}_{M\Boxt N,L,K}", from=1-2,to=2-3]
\arrow["\mathrm{A}_{M,N\Boxt L,K}", from=2-1,to=3-1,swap]
\arrow["\mathrm{A}_{M,N,L\Boxt K}", from=2-3,to=3-3]
\arrow["\id_M\Boxt\mathrm{A}_{N,L,K}", from=3-1,to=3-3,swap]
\end{tikzcd}
\end{equation*}
for all tube algebra representations $M,N,L,K$.
\end{prop}
\begin{proof}
In view of Proposition~\ref{prop:rel_unit_simple} (iii), it is enough to check that the pentagon axiom holds for elements of the form
\def\len{1.4cm}
\begin{equation}\label{eq:four_product_element}
\left(\left(m\otimes n\otimes\pielemUnit{a}{b}{u}{\pi}\right)\otimes
l\otimes \pielemUnit{u}{c}{v}{\rho}\right) \otimes k\otimes\pielemUnit{v}{d}{w}{\tau}
\end{equation}
where $a,b,c,d,u,v\in\Irr\cx$. Using Lemma~\ref{lem:associators_unit}, one obtains that both compositions
\[
\id_M\Boxt\mathrm{A}_{N,L,K}\circ  \mathrm{A}_{M,N\Boxt L,K}\circ \mathrm{A}_{M,N,L}\Boxt \id_K
\quad \text{and} \quad
\mathrm{A}_{M,N,L\Boxt K}\circ \mathrm{A}_{M\Boxt N,L,K}
\]
assign to~\eqref{eq:four_product_element} the value
\def\len{1.4cm}
\begin{equation*}
m\otimes \left(n\otimes
\left(l
\otimes k\otimes\Id_{cd}\right)\otimes\Id_{bcd}\right)\otimes
\begin{tikzpicture}[baseline={([yshift=-.5ex]current bounding box.center)}]
    \draw[line width=1] 
    (-\len,0) to (\len,0)
    (0,1.5*\len) to (\len,1.5*\len)
    (-0.5*\len,\len) to (\len,\len)
    (-\len,0.5*\len) to (\len,0.5*\len);
    
    \draw[line width=1,dotted] 
    (-\len,0) to (-\len,0.5*\len)
    (-0.5*\len,\len) to (-0.5*\len,0.5*\len)
    (0,1.5*\len) to (0,\len)
    (\len,0) to (\len,1.5*\len);

\node[rotate=180] at (0,0) {\scriptsize $\Arr$};
\node[rotate=180] at (-0.75*\len,0.5*\len) {\scriptsize $\Arr$};
\node[rotate=180] at (-0.25*\len,\len) {\scriptsize $\Arr$};
\node[rotate=180] at (0.5*\len,\len) {\scriptsize $\Arr$};
\node[rotate=180] at (0.25*\len,0.5*\len) {\scriptsize $\Arr$};
\node[rotate=180] at (0.75*\len,1.5*\len) {\scriptsize $\Arr$};
\node[rotate=180] at (0.25*\len,1.5*\len) {\scriptsize $\Arr$};

\filldraw [black] (-\len,0) circle (1pt);
\filldraw [black] (\len,0) circle (1pt);
\filldraw [black] (-0.5*\len,0.5*\len) circle (1pt);
\filldraw [black] (-0.5*\len,\len) circle (1pt);
\filldraw [black] (-\len,0.5*\len) circle (1pt);
\filldraw [black] (0,1.5*\len) circle (1pt);
\filldraw [black] (0,\len) circle (1pt);
\filldraw [black] (\len,1.5*\len) circle (1pt);
\filldraw [black] (0.5*\len,1.5*\len) circle (1pt);
\filldraw [black] (\len,\len) circle (1pt);
\filldraw [black] (\len,0.5*\len) circle (1pt);

\node at (0,0) [anchor=north] {$w$};
\node at (0.25*\len,0.5*\len) [anchor=north] {$v$};
\node at (-0.75*\len,0.5*\len) [anchor=south] {$d$};
\node at (-0.25*\len,\len) [anchor=south] {$c$};
\node at (0.5*\len,\len) [anchor=north] {$u$};
\node at (0.75*\len,1.5*\len) [anchor=south] {$a$};
\node at (0.25*\len,1.5*\len) [anchor=south] {$b$};

\draw (0.5*\len,1.25*\len) node {${\pi}$};
\draw (0.25*\len,0.75*\len) node {${\rho}$};
\draw (-0.25*\len,0.25*\len) node {${\tau}$};

\end{tikzpicture}
\end{equation*}
thereby proving the statement.
\end{proof}

\subsection{Monoidal unit and dualizable objects}

\begin{defi}
Define the \textit{trivial tube algebra representation} $\mathbb{I}$ as the vector space 
$\bigoplus_{a\in \IndexSet}\;\Hsp{\overline{a}\un}$
with tube algebra action given by the partial trace
\begin{equation*}
    l.f\coloneqq  \rd_x\,\,\ptr_x(l\circ_a f)
\end{equation*}
for $l\in\Hsp{\overline{a}\,\un}$ and $f\in\Hsp{\overline{b}\overline{x}ax}$.
\end{defi}
\begin{rem}
In case $\cx$ is a tensor category and we consider the tube algebra associated to the index set $\IndexSet=\Irr\cx$, the trivial representation is one-dimensional and the action is given by the full trace.
\end{rem}
The trivial representation $\mathbb{I}$ serves as monoidal unit for the tensor product~\eqref{tensor_prod}. In order to show this, define the \textit{right unitor} morphism by the assignment
\begin{equation*}
   \mathbf{r}_M\Colon M\to M\,\Boxt\,\mathbb{I}, \qquad m_a\mapsto m_a\otimes \Id_\un\otimes \Id_a
\end{equation*}
 and, similarly, the \textit{left unitor} morphism as
\begin{equation*}
   \mathbf{l}_M\Colon M\to \mathbb{I}\,\Boxt\,M, \qquad m_a\mapsto  \Id_\un\otimes m_a \otimes\Id_a\,.
\end{equation*}

\begin{prop}
For any tube algebra representation $M\in\Rep\Tubecx$, the unitors $\mathbf{r}_M$ and $\mathbf{l}_M$ are well-defined isomorphisms of tube algebra representations.
\end{prop}
\begin{proof}
Given $m\in M_a$, $l\in\Hsp{\overline{b}\,\un}$ and $\pi\in\auxprod$, we have that
\begin{equation*}%\label{eq:mon_unit_rel}
    m\otimes l\otimes \pi \overset{\eqref{eq:Rel2}}{\sim} m\otimes \Id_\un\otimes l\circ_b\pi \overset{\eqref{eq:Rel1}}{\sim}
    m.l\circ_b\pi \otimes \Id_\un\otimes \Id_c
\end{equation*}
and thus $(M\,\Boxt\,\mathbb{I})_c\cong M_c$, for every $c\in\IndexSet$. Moreover, relation~\eqref{eq:Rel1} ensures that $\mathbf{r}_M$ is an intertwiner. A similar argument shows the statement for $\mathbf{l}_M$.
\end{proof}

\begin{prop}
The triangle axiom holds, i.e., the diagram
\begin{equation*}
\begin{tikzcd}
    (M\Boxt \mathbb{I}) \Boxt N\ar[rr,"\mathrm{A}_{M,\mathbb{I},N}"]&& M\Boxt (\mathbb{I}\, \Boxt N)\\
    &M\Boxt N\ar[ul,"\mathbf{r}_M\Boxt \id_N"]\ar[ur,swap,"\id_M\Boxt\mathbf{l}_N"]&
\end{tikzcd}
\end{equation*}
commutes for all tube algebra representations $M$ and $N$.
\end{prop}
\begin{proof}
From Lemma~\ref{lem:associators_unit} it follows that
\begin{equation*}
\mathrm{A}_{M,\mathbb{I},N}\circ \mathbf{r}_M\Boxt \id_N(m\otimes n\otimes \pi)
    =m\otimes(\Id_\un\otimes n\otimes \Id_b)\otimes\pi
    =\id_M\Boxt\mathbf{l}_N(m\otimes n\otimes \pi)
\end{equation*}
for any $m\otimes n\otimes \pi\in M_a\otimes N_b\otimes \auxprod$, which proves the desired result.
\end{proof}
Now that we count with a monoidal unit in $\Rep\Tubecx$, we can discuss dualizability of representations. From the infinite nature of this tensor category we know that not every representation admits a dual object.

\begin{defi}
A tube algebra representation $M\in\Rep\Tubecx$ is called \textit{locally finite} when $M_a$ is finite-dimensional for every $a\in\IndexSet$ and the support $\{t\in\Irr\cx\,|\,M_t\neq0\}$ is finite.
\end{defi}

\begin{rem}
In the case that $\IndexSet = \Irr\cx$, the situation simplifies. We have that locally finite representations are actually finite-dimensional.
\end{rem}
To define the dual of a locally finite representation, we make use of the following involutive structure on the tube algebra: given $a,b\in\IndexSet$ the assignment%\vspace{-0.5em}
\begin{equation*}
    \left(-\right)^\#\colon \mathrm{T}_{a,b}\to\mathrm{T}_{\overline{b},\overline{a}},\qquad
\raisebox{0.3em}{\tubelem{a}{b}{x}{f}}\mapsto \raisebox{0.3em}{\tubelem{\overline{b}}{\overline{a}}{\overline{x}}{f}}%\vspace{-0.5em}
\end{equation*}
extends to the tube algebra $\Tubecx$ and obeys $(f\cdot g)^\#=g^\#\cdot f^\#$.
\begin{defi}
Let $M\in\Rep\Tubecx$ be a locally finite tube algebra representation. The \textit{dual tube algebra representation of $M$} is a triple consisting of:
\begin{enumerate}[$($i$)$]
    \item The representation $\overline{M}$ given for $a\in\IndexSet$ by the vector space $\overline{M}_a\coloneqq M_{\,\overline{a}}^*=\Hom_\bk(M_{\,\overline{a}},\bk)$ with tube algebra action
\begin{equation*}
    \varphi.f(m)\coloneqq \varphi\left(m.f^\#\right)
\end{equation*}
for $\varphi\in M_{\,\overline{a}}^*$, any tube element $f\in\mathrm{T}_{a,b}$ and $m\in M_{\,\overline{b}}$.
\item The evaluation morphism defined by the assignment%\vspace{-0.5em}
\begin{equation*}
    \ev_M\colon \overline{M}\Boxt M\to \mathbb{I},\qquad  \varphi\otimes m\otimes\!\!\raisebox{0.3em}{\pielem{a}{b}{c}{x}{\pi}}
    \mapsto\rd_x\;  \varphi.\overline{\eta}_k\left(m\right)\;\ptr_{x}\left(
\begin{tikzpicture}[baseline={([yshift=-.5ex]current bounding box.center)},xscale=0.8]
    \draw[line width=1,dotted]
    (-0.75*\len,0.5*\len) .. controls (-0.75*\len,1.15*\len) and (0.75*\len,1.15*\len).. (0.75*\len,0.5*\len);
    \draw[line width=1] 
    (-0.75*\len,0) to (0.75*\len,0)
    (-0.75*\len,0) to (-0.75*\len,0.5*\len)
    (0.75*\len,0) to (0.75*\len,0.5*\len)
    (-0.75*\len,0.5*\len) to (0.75*\len,0.5*\len);

\node[rotate=-90] at (-0.75*\len,0.25*\len) {\scriptsize $\Arr$};
\node[rotate=-90] at (0.75*\len,0.25*\len) {\scriptsize $\Arr$};
\node[rotate=180] at (0,0) {\scriptsize $\Arr$};
\node[rotate=180] at (-0.4*\len,0.5*\len) {\scriptsize $\Arr$};
\node[rotate=180] at (0.4*\len,0.5*\len) {\scriptsize $\Arr$};

\filldraw [black] (-0.75*\len,0) circle (1pt);
\filldraw [black] (0.75*\len,0) circle (1pt);
\filldraw [black] (0,0.5*\len) circle (1pt);
\filldraw [black] (-0.75*\len,0.5*\len) circle (1pt);
\filldraw [black] (0.75*\len,0.5*\len) circle (1pt);

% Text Node
\node at (-0.75*\len,0.25*\len) [anchor=east] {$x$};
\node at (0.75*\len,0.25*\len) [anchor=west] {$x$};
\node at (0,0) [anchor=north] {$c$};
\node at (0.4*\len,0.5*\len) [anchor=north] {$a$};
\node at (-0.4*\len,0.5*\len) [anchor=north] {$b$};

\draw (0,0.25*\len) node {${\pi}$};
\draw (0,0.7*\len) node {${\eta^k}$};
\end{tikzpicture}
\right)
\end{equation*}
where $\varphi\in\overline{M}_a$, $m\in M_b$ and $(\overline{\eta},\eta)$ is a pair of dual bases of $\Hsp{a\,b}$.
\item The coevaluation morphism is given by%\vspace{-0.5em}
\def\len{2cm}
\begin{equation*}
    \coev_M\colon \mathbb{I}\to M\Boxt \overline{M},\qquad{\unelem{c}{l}}  \mapsto\sum_{t\in\Irr\cx}\rd_t\;  \sum_{i}m_t^i\otimes \overline{m}_t^i\otimes 
\begin{tikzpicture}[xscale=1.25,baseline={([yshift=-.5ex]current bounding box.center)}]
\draw[line width=1,dotted]
    (-0.25*\len,0) .. controls (-0.2*\len,0.3*\len) and (0.2*\len,0.3*\len).. (0.25*\len,0);
\draw[line width=1,dotted]
    (-0.25*\len,0) to (-0.25*\len,0.5*\len)
    (0.25*\len,0) to (0.25*\len,0.5*\len);

\draw[line width=1] 
    (-0.25*\len,0) to (0.25*\len,0)(-0.25*\len,0.5*\len) to (0.25*\len,0.5*\len);

\node[rotate=180] at (0,0) {\scriptsize $\Arr$};
\node[rotate=180] at (-0.13*\len,0.5*\len) {\scriptsize $\Arr$};
\node[rotate=0] at (0.13*\len,0.5*\len) {\scriptsize $\Arr$};

\filldraw [black] (-0.25*\len,0) circle (1pt);
\filldraw [black] (0.25*\len,0) circle (1pt);
\filldraw [black] (0,0.5*\len) circle (1pt);
\filldraw [black] (0.25*\len,0.5*\len) circle (1pt);
\filldraw [black] (-0.25*\len,0.5*\len) circle (1pt);

% % Text Node
\node at (0,0) [anchor=north]{$c$};
\node at (-0.13*\len,0.5*\len) [anchor=south]{\small$t$};
\node at (0.13*\len,0.5*\len) [anchor=south]{\small$t$};
\draw (0,0.12*\len) node {${l}$};
\draw (0,0.35*\len) node {$\Id_t$};
\end{tikzpicture}
\end{equation*}
where $t$ is taken from $\Irr\cx_k$ with $c \in \rI_k$, and $\{m_t^i\}_i$ is a basis for $M_t$ and $\{\overline{m}_t^i\}_i$ its dual basis, that means that $\overline{m}_t^j(m_t^i)=\delta_{i,j}$.
\end{enumerate}
\end{defi}
\begin{prop}\label{prop:rigidity}
The full $\bk$-linear subcategory $\Rep^{\mathrm{l.f.}}\Tubecx\subset \Rep\Tubecx$ of locally finite tube algebra representations is rigid.
\end{prop}
\begin{proof}
That the subcategory $\Rep^{\mathrm{l.f.}}\Tubecx$ is closed under the tensor product $\Boxt$ follows from Proposition~\ref{prop:rel_unit_simple} (iii). To check rigidity, we start by proving that $\ev_M$ is well-defined. Consider a vector $\varphi.f\otimes m\otimes \pi$ in the tensor product $\overline{M}\Boxt M$. The image of the right-hand side of~\eqref{eq:Rel1} under $\ev_M$ equals
\begin{equation*}
    \rd_x\;\varphi\left(m.\overline{\beta}_j\right) 
\def\len{1.6cm}
\ptr_{yx}\left(
\begin{tikzpicture}[baseline={([yshift=-.5ex]current bounding box.center)},xscale=1]
 \draw[line width=1,dotted]
    (-0.75*\len,\len) .. controls (-0.75*\len,1.65*\len) and (0.75*\len,1.65*\len).. (0.75*\len,\len);
\draw[line width=1] 
    (-0.75*\len,0) to (0.75*\len,0)
    (-0.75*\len,0) to (-0.75*\len,0.5*\len)
    (0.75*\len,0) to (0.75*\len,0.5*\len)
    (-0.75*\len,0.5*\len) to (0.75*\len,0.5*\len)
    (0,0.5*\len) to (0,\len)
    (0,\len) to (0.75*\len,\len)
    (0.75*\len,0.5*\len) to (0.75*\len,\len)
    (-0.75*\len,0.5*\len) to (-0.75*\len,\len)
    (0,\len) to (-0.75*\len,\len)
    ;
\node[rotate=180] at (0,0) {\scriptsize $\Arr$};
\node[rotate=-90] at (0,0.75*\len) {\scriptsize $\Arr$};
\node[rotate=-90] at (-0.75*\len,0.75*\len) {\scriptsize $\Arr$};
\node[rotate=-90] at (0.75*\len,0.75*\len) {\scriptsize $\Arr$};
\node[rotate=-90] at (-0.75*\len,0.25*\len) {\scriptsize $\Arr$};
\node[rotate=-90] at (0.75*\len,0.25*\len) {\scriptsize $\Arr$};

\node[rotate=0] at (-0.4*\len,\len) {\scriptsize $\Arr$};
\node[rotate=180] at (-0.4*\len,0.5*\len) {\scriptsize $\Arr$};
\node[rotate=180] at (0.4*\len,0.5*\len) {\scriptsize $\Arr$};
\node[rotate=180] at (0.4*\len,\len) {\scriptsize $\Arr$};

\filldraw [black] (0,\len) circle (1pt);
\filldraw [black] (0.75*\len,\len) circle (1pt);
\filldraw [black] (-0.75*\len,\len) circle (1pt);
\filldraw [black] (-0.75*\len,0) circle (1pt);
\filldraw [black] (0.75*\len,0) circle (1pt);
\filldraw [black] (0,0.5*\len) circle (1pt);
\filldraw [black] (-0.75*\len,0.5*\len) circle (1pt);
\filldraw [black] (0.75*\len,0.5*\len) circle (1pt);

% Text Node
\node at (-0.75*\len,0.25*\len) [anchor=east] {$x$};
\node at (0.75*\len,0.25*\len) [anchor=west] {$x$};
\node at (0.75*\len,0.75*\len) [anchor=west] {$y$};
\node at (-0.75*\len,0.75*\len) [anchor=east] {$y$};
\node at (0,0.75*\len) [anchor=west] {$y$};
\node at (0,0) [anchor=north] {$c$};
\node at (0.4*\len,0.5*\len) [anchor=north] {$a$};
\node at (0.4*\len,\len) [anchor=south] {$\tilde{a}$};
\node at (-0.4*\len,0.5*\len) [anchor=north] {$b$};
\node at (-0.4*\len,\len) [anchor=south] {$\tilde{a}$};

\node at (0.45*\len,0.75*\len) { {$f$}};
\node at (-0.4*\len,0.75*\len) {${\beta^j}$};
\draw (0,1.25*\len) node  {${\Id_{\tilde{a}}}$};
\draw (0,0.25*\len) node  {${\pi}$};
\end{tikzpicture}
\right)
\end{equation*}
after applying the identity~\eqref{dominance}. We obtain by means of~\eqref{eq:base_change} the expression
\begin{equation*}
    \rd_x\;\varphi\left(m.\left(\overline{\eta}_k^\#\cdot f^\#\right)\right)\otimes\ptr_x\left(\eta^k\circ_{ab}\pi\right)=\ev_M\left(\varphi.f\otimes m\otimes \pi\right)
\end{equation*}
for $(\overline{\eta},\eta)$a pair of dual bases of $\Hsp{a\,b}$. Compatibility with relation~\eqref{eq:Rel2} follows analogously. A routine check shows that $\ev_M$ is an intertwiner of tube algebra representations. To prove that $\coev_M$ is compatible with the action requires a more involved calculation using~\eqref{eq:rel_unit}.
We check now the snake relation. By direct computation one can verify that the composition 
\begin{equation*}
M\xrightarrow{\;\mathbf{l}_M\;}\mathbb{I}\Boxt M\xrightarrow{\;\coev_M\Boxt\id\;}\left(M\Boxt \overline{M}\right) \Boxt M    \xrightarrow{\;\mathrm{A}_{M,\overline{M},M}\;} M\Boxt \left(\overline{M}\Boxt M\right)\xrightarrow{\;\id\Boxt\ev_M\;} M\Boxt \mathbb{I}   \xrightarrow{\;\mathbf{r}_M^{-1}\;}M
\end{equation*}
assigns to a vector $m\in M_a$ the value
\begin{equation*}%\label{eq:snake}
\sum_{t\in\Irr\cx}\rd_t\;  \sum_{i}
\left(\overline{m}_t^i.\eta^k\right)(m)
\;m_t^i.\overline{\eta}_k=\sum_{j}  
\overline{m}_a^j(m)
\;m_a^j=m
\end{equation*}
where $(\eta,\overline{\eta})$ is a pair of dual bases of $\Hsp{a\,t}$
and the first equality follows from Lemma~\ref{lem:pairing_properties} (i) and the independence of the coevaluation map from the choice of basis. The second equality follows by the definition of the dual basis to $\{m_a^j\}_j$. The remaining snake relation holds similarly.
\end{proof}

\subsection{Braiding and twist}
We turn now to the definition of the braiding in $\Rep\Tubecx$. To this end, we need first to look into the following special class of tube elements.
\def\len{2cm}
\begin{defi}
The \textit{twist element of $a\in\Iso\cx_i$} is defined as the tube element   
\begin{equation*}
\twist_a\coloneqq\rd_{a}^{-1}
\begin{tikzpicture}[xscale=0.75,baseline={([yshift=-.5ex]current bounding box.center)}]
\draw[line width=1] 
    (-0.5*\len,0) to (0.5*\len,0)
    (-0.5*\len,0) to (-0.5*\len,0.5*\len)
    (0.5*\len,0) to (0.5*\len,0.5*\len)
    (-0.5*\len,0.5*\len) to (0.5*\len,0.5*\len);

\draw[line width=1,dotted] 
    (0.5*\len,0.5*\len) to (-0.5*\len,0);

\node[rotate=90] at (-0.5*\len,0.25*\len) {\scriptsize $\Arr$};
\node[rotate=90] at (0.5*\len,0.25*\len) {\scriptsize $\Arr$};
\node[rotate=180] at (0,0) {\scriptsize $\Arr$};
\node[rotate=180] at (0,0.5*\len) {\scriptsize $\Arr$};

\filldraw [black] (-0.5*\len,0) circle (1pt);
\filldraw [black] (0.5*\len,0) circle (1pt);
\filldraw [black] (-0.5*\len,0.5*\len) circle (1pt);
\filldraw [black] (0.5*\len,0.5*\len) circle (1pt);

% Text Node
\node at (-0.5*\len,0.25*\len) [anchor=east] {$a$};
\node at (0.5*\len,0.25*\len) [anchor=west] {$a$};
\node at (0,0) [anchor=north] {$a$};
\node at (0,0.5*\len) [anchor=south] {$a$};

\draw (-0.25*\len,0.3*\len) node    {${\Id_a}$};
\draw (0.25*\len,0.2*\len) node    {${\Id_a}$};
\end{tikzpicture}
\; \in\mathrm{T}_{a;a}^{\Iso\cx_i}\;.
\end{equation*}
\end{defi}
In the case that  $a\notin\IndexSet$, the twist element does not properly belong to the tube algebra $\Tubecx$. However, tube algebra actions can be extended to tube elements $ f\in\mathrm{T}_{a;b}^{\Iso\cx}$ by resolving the $x$ label into simple objects, i.e., by considering $\star_x(f)\in \Tubecx$ and defining the action $m.f\coloneqq m.\star_x(f)$ for $m\in M$.
\begin{lem}Twist elements obey the following properties.
\begin{enumerate}[\rm (i)]
\item The twist element of $a\in\Iso\cx_i$ is invertible under~\eqref{eq:tube_weld} and its inverse is given by
\begin{equation*}
\twist_a^{-1}=\rd_{a}^{-1}
\begin{tikzpicture}[xscale=0.75,baseline={([yshift=-.5ex]current bounding box.center)}]
\draw[line width=1] 
    (-0.5*\len,0) to (0.5*\len,0)
    (-0.5*\len,0) to (-0.5*\len,0.5*\len)
    (0.5*\len,0) to (0.5*\len,0.5*\len)
    (-0.5*\len,0.5*\len) to (0.5*\len,0.5*\len);

\draw[line width=1,dotted] 
    (-0.5*\len,0.5*\len) to (0.5*\len,0);

\node[rotate=-90] at (-0.5*\len,0.25*\len) {\scriptsize $\Arr$};
\node[rotate=-90] at (0.5*\len,0.25*\len) {\scriptsize $\Arr$};
\node[rotate=180] at (0,0) {\scriptsize $\Arr$};
\node[rotate=180] at (0,0.5*\len) {\scriptsize $\Arr$};

\filldraw [black] (-0.5*\len,0) circle (1pt);
\filldraw [black] (0.5*\len,0) circle (1pt);
\filldraw [black] (-0.5*\len,0.5*\len) circle (1pt);
\filldraw [black] (0.5*\len,0.5*\len) circle (1pt);

% Text Node
\node at (-0.5*\len,0.25*\len) [anchor=east] {$a$};
\node at (0.5*\len,0.25*\len) [anchor=west] {$a$};
\node at (0,0) [anchor=north] {$a$};
\node at (0,0.5*\len) [anchor=south] {$a$};

\draw (0.25*\len,0.3*\len) node    {${\Id_a}$};
\draw (-0.25*\len,0.2*\len) node    {${\Id_a}$};
\end{tikzpicture}
\;.
\end{equation*}
\item For any $f\in  \mathrm{T}_{a;b}^{\Iso\cx_i}$ it holds that
\begin{equation}\label{eq:twist-central}
    \twist_a\cdot f=f\cdot \twist_b\;\;\in \mathrm{T}_{a;b}^{\Iso\cx_i}\;.
\end{equation}
\item For any $a\in\Iso\cx_i$, it holds that
\begin{equation*}
\star_a(\twist_a)=\sum_{t\in\Irr\cx_i}\rd_t\,\overline{\alpha}_k\cdot\twist_t\cdot \alpha^k\;,
\end{equation*}
where $(\alpha,\overline{\alpha})$ is a pair of dual bases of $\Hsp{\overline{a}t}$.
\end{enumerate}
\end{lem}

\begin{proof}
The first statement is an immediate consequence of the definition of welding of tube elements. To verify assertion (ii) notice that from the definition of twist elements we have that
\begin{equation*}
\twist_a^{-1}\cdot f=\sum_{z\in\Irr\cx}\hspace{-1.5em}
\begin{tikzpicture}[baseline={([yshift=-.5ex]current bounding box.center)},xscale=0.8]
\draw[line width=1]
    (-0.5*\len,0) .. controls (-1.45*\len,1*\len) and (0.25*\len,1*\len).. (0.5*\len,0.5*\len)
    node[midway,rotate=180] {\scriptsize $\Arr$}
    node[midway,above] {\footnotesize $z$};
\draw[line width=1]
    (0.5*\len,0) .. controls (1*\len,0*\len) and (1*\len,1*0.55*\len).. (1*\len,0.5*\len)
    node[midway,rotate=-115] {\scriptsize $\Arr$}
    node[midway,right] {\footnotesize $z$};
    
\draw[line width=1] 
    (-0.5*\len,0) to (0.5*\len,0)
    (-0.5*\len,0) to (-0.5*\len,0.5*\len)
    (0.5*\len,0) to (0.5*\len,0.5*\len)
    (-0.5*\len,0.5*\len) to (1*\len,0.5*\len);

\node[rotate=180] at (0,0) {\scriptsize $\Arr$};
\node[rotate=180] at (0,0.5*\len) {\scriptsize $\Arr$};
\node[rotate=180] at (0.75*\len,0.5*\len) {\scriptsize $\Arr$};

\filldraw [black] (-0.5*\len,0) circle (1pt);
\filldraw [black] (0.5*\len,0) circle (1pt);
\filldraw [black] (-0.5*\len,0.5*\len) circle (1pt);
\filldraw [black] (0.5*\len,0.5*\len) circle (1pt);
\filldraw [black] (\len,0.5*\len) circle (1pt);

% Text Node
\node at (-0.45*\len,0.25*\len) { $\bullet_x$};
\node at (0.55*\len,0.25*\len) { $\bullet_x$};
\node at (0,0) [anchor=north] {$b$};
\node at (0.15*\len,0.5*\len) [anchor=south] {$a$};
\node at (0.75*\len,0.5*\len) [anchor=south] {$a$};

\draw (0,0.25*\len) node  {${f}$};
\draw (0.8*\len,0.325*\len) node  {${\eta^j}$};
\draw (-0.25*\len,0.65*\len) node  {${\overline{\eta}_j}$};
\end{tikzpicture}
\overset{\eqref{eq:base_change}}{=}\sum_{z\in\Irr\cx}
\begin{tikzpicture}[baseline={([yshift=-.5ex]current bounding box.center)},xscale=0.8]
\draw[line width=1]
    (0.5*\len,0.5*\len) .. controls (1.45*\len,-0.5*\len) and (-0.25*\len,-0.5*\len).. (-0.5*\len,0)
    node[midway,rotate=180] {\scriptsize $\Arr$}
    node[midway,below] {\footnotesize $z$};
\draw[line width=1]
    (-0.5*\len,0.5*\len) .. controls (-1*\len,1*0.55*\len) and (-1*\len,0).. (-1*\len,0)
    node[midway,rotate=-115] {\scriptsize $\Arr$}
    node[midway,left] {\footnotesize $z$};
    
\draw[line width=1] 
    (-1*\len,0) to (0.5*\len,0)
    (-0.5*\len,0) to (-0.5*\len,0.5*\len)
    (0.5*\len,0) to (0.5*\len,0.5*\len)
    (0.5*\len,0.5*\len) to (-0.5*\len,0.5*\len);

\node at (-0.45*\len,0.25*\len) { $\bullet_x$};
\node at (0.55*\len,0.25*\len) { $\bullet_x$};
\node[rotate=180] at (0,0) {\scriptsize $\Arr$};
\node[rotate=180] at (0,0.5*\len) {\scriptsize $\Arr$};
\node[rotate=180] at (-0.75*\len,0) {\scriptsize $\Arr$};

\filldraw [black] (-0.5*\len,0) circle (1pt);
\filldraw [black] (0.5*\len,0) circle (1pt);
\filldraw [black] (-0.5*\len,0.5*\len) circle (1pt);
\filldraw [black] (0.5*\len,0.5*\len) circle (1pt);
\filldraw [black] (-\len,0) circle (1pt);

% Text Node
\node at (-0.1*\len,0.05*\len) [anchor=north] {$b$};
\node at (0,0.5*\len) [anchor=south] {$a$};
\node at (-0.75*\len,0) [anchor=north] {$b$};

\draw (0,0.25*\len) node  {${f}$};
\draw (-0.7*\len,0.2*\len) node {${\overline{\nu}_j}$};
\draw (0.325*\len,-0.15*\len) node   {${\nu^j}$};
\end{tikzpicture}
\hspace{-1.5em} =f\cdot \twist^{-1}_b
\end{equation*}
where $(\eta,\overline{\eta})$ is a pair of dual bases of $\Hsp{ax\overline{z}}$ and $(\nu,\overline{\nu})$ of $\Hsp{xb\overline{z}}$. The third statement follows by applying (ii) and relation~\eqref{dominance}.
\end{proof} 
Twist elements define a twist on $\Rep\Tubecx$ as it will be shown in Proposition~\ref{prop:twist}, but first we define the braiding with help of the (inverse) twist element, as well.
\begin{defi}\label{def:braiding}
Let $M,N$ be tube algebra representations. The \textit{braiding}
\begin{equation}\label{eq:braiding}
    \mathrm{B}_{M,N}\Colon M\Boxt N\to N\Boxt M
\end{equation}
is defined by the assignment%\vspace{-1em}
\begin{equation*}
m\otimes n\otimes\pielem{a}{b}{c}{x}{\pi} 
\mapsto
\sum_{z\in\Irr\cx} n \otimes m.\twist_{a}^{-1}\otimes\!\!
\begin{tikzpicture}[yscale=1,xscale=1,baseline={([yshift=-.5ex]current bounding box.center)}]
 \draw[line width=1]
    (0.5*\len,0) .. controls (1.1*\len,0*\len) and (1.1*\len,2*0.375*\len).. (0.5*\len,0.75*\len) node [midway,rotate=-90]  {\scriptsize $\Arr$}
    node [midway,right]  { $z$};

\draw[line width=1] 
    (-0.5*\len,0) to (0.5*\len,0)
    (-0.5*\len,0) to (-0.5*\len,0.75*\len)
    (0.5*\len,0) to (0.5*\len,0.75*\len)
    (-0.5*\len,0.75*\len) to (0.5*\len,0.75*\len)
    (-0.5*\len,0.75*\len) to (-1*\len,0.75*\len)
    (-0.5*\len,0) to [in=-90,out=180] (-1*\len,0.75*\len);

\node[rotate=90] at (0.5*\len,1.5*0.375*\len) {\scriptsize $\Arr$};
% \node[rotate=90] at (0.8*\len,1.5*0.375*\len) {\scriptsize $\Arr$};
\node[rotate=-70] at (-0.94*\len,0.375*\len) {\scriptsize $\Arr$};
\node[rotate=180] at (0,0) {\scriptsize $\Arr$};
\node[rotate=180] at (-0.75*\len,0.75*\len) {\scriptsize $\Arr$};
\node[rotate=180] at (0,0.75*\len) {\scriptsize $\Arr$};

\filldraw [black] (-0.5*\len,0) circle (1pt);
\filldraw [black] (0.5*\len,0) circle (1pt);
\filldraw [black] (-0.5*\len,0.75*\len) circle (1pt);
\filldraw [black] (0.5*\len,0.375*\len) circle (1pt);
\filldraw [black] (0.5*\len,0.75*\len) circle (1pt);
\filldraw [black] (-1*\len,0.75*\len) circle (1pt);

% Text Node
\node at (-0.45*\len,0.375*\len)  {$\bullet_x$};
\node at (0.55*\len,0.5*0.375*\len) {$\bullet_x$};
\node at (0,0) [anchor=north] {$c$};
\node at (0,0.75*\len) [anchor=south] {$b$};
\node at (-0.75*\len,0.75*\len) [anchor=south] {$a$};
\node at (0.5*\len,1.5*0.375*\len) [anchor=east] {$a$};
% \node at (1*\len,1.5*0.325*\len) [anchor=east,font=\footnotesize] {$\tilde{a}$};
% \node at (1.1*\len,0.375*\len) [anchor=west] {$z$};
\node at (-1*\len,0.35*\len) [anchor=east] {$z$};

\draw (0,0.375*\len) node {${\pi}$};
\node at (-0.7*\len,0.45*\len) {{$\overline{\varphi}_{\!k}$}};
\node at (0.7*\len,0.375*\len) {{$\varphi^k$}};
% \node at (0.635*\len,1.5*0.375*\len) {\footnotesize {$\alpha^i$}};
\end{tikzpicture}
\end{equation*}
for $m\otimes n\otimes \pi\in M_a\otimes N_b\otimes\auxprod$.
\end{defi}

\begin{prop}
The above $\mathrm{B}_{M,N}$ is a well-defined natural isomorphism of $\Tubecx$-representations.
\end{prop}
\begin{proof}
That $\mathrm{B}_{M,N}$ is an intertwiner follows from composing pairs of dual bases according to Lemma~\ref{bases_comp}.
To prove that the braiding is well-defined, we consider the relation~\eqref{eq:Rel1}. For simplicity, assume that $a\in\Irr\cx$, then
\begin{equation*}
\rd_{y}^{-1}\sum_{\tilde{b},z,w\in\Irr\cx} n.\overline{\beta}_j \otimes m.\twist_{a'}^{-1}\otimes\!\!
\begin{tikzpicture}[baseline={([yshift=-.5ex]current bounding box.center)},xscale=0.75]
\draw[line width=1]
(-0.75*\len,0) .. controls (-1.25*\len,0.25*\len) and (-1.25*\len,0.75*\len).. (-0.75*\len,\len);
\draw[line width=1]
(0.75*\len,0) .. controls (1.25*\len,0.25*\len) and (1.25*\len,0.75*\len).. (0.75*\len,\len);
\draw[line width=1]
(0.75*\len,0) .. controls (2*\len,0) and (1.5*\len,1.75*\len).. (0,\len)
node[midway,rotate=-80]{\scriptsize{$\Arr$}}
node[midway,right]{\footnotesize{$w$}};

\node at (-0.7*\len,0.25*\len) {\normalsize $\bullet_x$};
\node at (0.8*\len,0.25*\len) {\normalsize $\bullet_x$};
\node at (-0.7*\len,0.75*\len) {\normalsize $\bullet_y$};
\node at (0.8*\len,0.75*\len) {\normalsize $\bullet_y$};

\draw[line width=1] 
    (-0.75*\len,0) to (0.75*\len,0)
    (-0.75*\len,0) to (-0.75*\len,0.5*\len)
    (0.75*\len,0) to (0.75*\len,0.5*\len)
    (-0.75*\len,0.5*\len) to (0.75*\len,0.5*\len)
    (0,0.5*\len) to (0,\len)
    (0,\len) to (0.75*\len,\len)
    (0.75*\len,0.5*\len) to (0.75*\len,\len)
    (-0.75*\len,0.5*\len) to (-0.75*\len,\len)
    (0,\len) to (-1.5*\len,\len)
    (-0.75*\len,0) to [in=-90,out=180] (-1.5*\len,\len)
    ;

\node at (-1.05*\len,0.35*\len) {\footnotesize $\bullet_z$};
\node at (1.15*\len,0.35*\len) {\footnotesize$\bullet_z$};
\node[rotate=180] at (0,0) {\scriptsize $\Arr$};
\node[rotate=-90] at (0,0.75*\len) {\scriptsize $\Arr$};
\node[rotate=-80] at (-1.4*\len,0.5*\len) {\scriptsize $\Arr$};
\node[rotate=180] at (-0.4*\len,\len) {\scriptsize $\Arr$};
\node[rotate=180] at (-1.2*\len,\len) {\scriptsize $\Arr$};
\node[rotate=180] at (-0.4*\len,0.5*\len) {\scriptsize $\Arr$};
\node[rotate=180] at (0.4*\len,0.5*\len) {\scriptsize $\Arr$};
\node[rotate=180] at (0.4*\len,\len) {\scriptsize $\Arr$};

\filldraw [black] (0,\len) circle (1pt);
\filldraw [black] (0.75*\len,\len) circle (1pt);
\filldraw [black] (-0.75*\len,\len) circle (1pt);
\filldraw [black] (-1.5*\len,\len) circle (1pt);
\filldraw [black] (-0.75*\len,0) circle (1pt);
\filldraw [black] (0.75*\len,0) circle (1pt);
\filldraw [black] (0,0.5*\len) circle (1pt);
\filldraw [black] (-0.75*\len,0.5*\len) circle (1pt);
\filldraw [black] (0.75*\len,0.5*\len) circle (1pt);

\node at (-1.4*\len,0.5*\len) [anchor=east] {$w$};
\node at (0,0.75*\len) [anchor=west] {$y$};
\node at (0,0) [anchor=north] {$c$};
\node at (0.4*\len,0.5*\len) [anchor=north] {$a$};
\node at (0.45*\len,1.05*\len) [anchor=north] {\footnotesize$a'$};
\node at (-0.4*\len,0.5*\len) [anchor=north] {$b$};
\node at (-0.4*\len,\len) [anchor=south] {$\tilde{b}$};
\node at (-1.15*\len,\len) [anchor=south] {$a'$};

\node at (0.45*\len,0.7*\len) {{$f$}};
\node at (-0.35*\len,0.75*\len) { {$\beta^j$}};

\node at (-0.9*\len,0.5*\len) { {\footnotesize$\overline{\psi}_{\!k}$}};
\node at (0.95*\len,0.5*\len) {\footnotesize {$\psi^k$}};
\node at (-1.25*\len,0.75*\len) { {\footnotesize$\overline{\mu}_{i}$}};
\node at (1.2*\len,0.8*\len) {\footnotesize {$\mu^i$}};
\draw (0,0.25*\len) node {${\pi}$};
\end{tikzpicture}
\end{equation*}
represents the image of the right hand side of~\eqref{eq:Rel1} under $\mathrm{B}_{M,N}$.

%%%%%%%%%%%%%%%%%%%%%%%%%%%%%%%%%%%%%%%%%%%%%%%%5
Now, the braiding $\mathrm{B}_{M,N}$ assigns to $m.f\otimes n \otimes\pi$ the following value%\vspace{-2em}
\begin{equation*}
\sum_{z\in\Irr\cx} n \otimes m.(f\cdot\twist_{a}^{-1})\otimes\!\!
\begin{tikzpicture}[yscale=0.8,xscale=0.8,baseline={([yshift=-.5ex]current bounding box.center)}]
 \draw[line width=1]
    (0.5*\len,0) .. controls (1*\len,0*\len) and (1*\len,2*0.375*\len).. (0.5*\len,0.75*\len)
    node[midway,rotate=-90]{\scriptsize $\Arr$}
    node[midway,right]{$z$};
\draw[line width=1] 
    (-0.5*\len,0) to (0.5*\len,0)
    (-0.5*\len,0) to (-0.5*\len,0.75*\len)
    (0.5*\len,0) to (0.5*\len,0.75*\len)
    (-0.5*\len,0.75*\len) to (0.5*\len,0.75*\len)
    (-0.5*\len,0.75*\len) to (-1*\len,0.75*\len)
    (-0.5*\len,0) to [in=-90,out=180] (-1*\len,0.75*\len);

\node[rotate=90] at (0.5*\len,1.5*0.375*\len) {\scriptsize $\Arr$};
\node[rotate=-70] at (-0.94*\len,0.375*\len) {\scriptsize $\Arr$};
\node[rotate=180] at (0,0) {\scriptsize $\Arr$};
\node[rotate=180] at (-0.75*\len,0.75*\len) {\scriptsize $\Arr$};
\node[rotate=180] at (0,0.75*\len) {\scriptsize $\Arr$};

\filldraw [black] (-0.5*\len,0) circle (1pt);
\filldraw [black] (0.5*\len,0) circle (1pt);
\filldraw [black] (-0.5*\len,0.75*\len) circle (1pt);
\filldraw [black] (0.5*\len,0.375*\len) circle (1pt);
\filldraw [black] (0.5*\len,0.75*\len) circle (1pt);
\filldraw [black] (-1*\len,0.75*\len) circle (1pt);

% Text Node
\node at (-0.45*\len,0.375*\len)  {$\bullet_x$};
\node at (0.55*\len,0.5*0.375*\len) {$\bullet_x$};
\node at (0,0) [anchor=north] {$c$};
\node at (0,0.75*\len) [anchor=south] {$b$};
\node at (-0.75*\len,0.75*\len) [anchor=south] {$a$};
\node at (0.5*\len,1.5*0.375*\len) [anchor=east] {$a$};
\node at (-1*\len,0.35*\len) [anchor=east] {$z$};

\draw (0,0.375*\len) node {${\pi}$};
\node at (-0.7*\len,0.40*\len) {{$\overline{\varphi}_{\!k}$}};
\node at (0.7*\len,0.42*\len) {{$\varphi^k$}};
\end{tikzpicture} 
\overset{\eqref{eq:Rel2}}{\sim}\!\!
\sum_{\tilde{b},z,w\in\Irr\cx} \!\!\!\!n.\frac{\overline{\beta}_j}{\rd_{y}} \otimes m.\twist_{a'}^{-1}\otimes\!\!
\begin{tikzpicture}[yscale=0.8,xscale=0.75,baseline={([yshift=-.5ex]current bounding box.center)}]
\draw[line width=1]
    (-0.5*\len,0) .. controls (-1.65*\len,0*\len) and (-1.5*\len,1.15*\len).. (-1*\len,1.25*\len)
    node[midway,rotate=-90]{\scriptsize $\Arr$}
    node[midway,left]{\footnotesize $w$};
\draw[line width=1]
    (0.5*\len,0) .. controls (1.5*\len,0*\len) and (1.5*\len,1.25*\len).. (0.5*\len,1.25*\len)
    node[midway,rotate=-90]{\scriptsize $\Arr$}
    node[midway,right]{\footnotesize $w$};    
\draw[line width=1]
    (0.5*\len,0) .. controls (1*\len,0*\len) and (1*\len,2*0.375*\len).. (0.5*\len,0.75*\len)
    node[midway]{ $~\bullet_z$};
\draw[line width=1] 
    (-0.5*\len,0) to (0.5*\len,0)
    (-0.5*\len,0) to (-0.5*\len,0.75*\len)
    (0.5*\len,0) to (0.5*\len,0.75*\len)
    (-0.5*\len,0.75*\len) to (0.5*\len,0.75*\len)
    (-0.5*\len,0.75*\len) to (-1*\len,0.75*\len)
    (-0.5*\len,0) to [in=-90,out=180] (-1*\len,0.75*\len);
\draw[line width=1] 
    (-1*\len,1.25*\len) to (0.5*\len,1.25*\len)
    (-0.5*\len,0.75*\len) to (-0.5*\len,1.25*\len)
    (0.5*\len,0.75*\len) to (0.5*\len,1.25*\len)
    (-1*\len,0.75*\len) to (-1*\len,1.25*\len);
    
\node[rotate=90] at (0.5*\len,1.5*0.375*\len) {\scriptsize $\Arr$};
\node[rotate=20] at (-0.9*\len,0.425*\len) {\footnotesize $\bullet_z$};
\node[rotate=-90] at (-0.5*\len,1*\len) {\scriptsize $\Arr$};
\node at (-1*\len,1*\len) {\footnotesize $~\bullet_y$};
\node at (0.5*\len,1*\len) {\footnotesize $~\bullet_y$};
\node[rotate=180] at (0,0) {\scriptsize $\Arr$};
\node[rotate=180] at (-0.75*\len,0.75*\len) {\scriptsize $\Arr$};
\node[rotate=180] at (0,1.25*\len) {\scriptsize $\Arr$};
\node[rotate=180] at (-0.75*\len,1.25*\len) {\scriptsize $\Arr$};
\node[rotate=180] at (0,0.75*\len) {\scriptsize $\Arr$};

\filldraw [black] (-0.5*\len,0) circle (1pt);
\filldraw [black] (0.5*\len,0) circle (1pt);
\filldraw [black] (-0.5*\len,0.75*\len) circle (1pt);
\filldraw [black] (0.5*\len,0.375*\len) circle (1pt);
\filldraw [black] (0.5*\len,0.75*\len) circle (1pt);
\filldraw [black] (-1*\len,0.75*\len) circle (1pt);
\filldraw [black] (-0.5*\len,1.25*\len) circle (1pt);
\filldraw [black] (0.5*\len,1.25*\len) circle (1pt);
\filldraw [black] (-1*\len,1.25*\len) circle (1pt);

% Text Node
\node at (-0.5*\len,1*\len) [anchor=west]  {$y$};
\node at (-0.45*\len,0.375*\len)  {$\bullet_x$};
\node at (0.55*\len,0.5*0.375*\len) {$\bullet_x$};
\node at (0,0) [anchor=north] {$c$};
\node at (0,0.75*\len) [anchor=north] {\footnotesize$b$};
\node at (-0.75*\len,0.75*\len) [anchor=north] {\footnotesize$a$};
\node at (0,1.25*\len) [anchor=south] {$\tilde{b}$};
\node at (-0.75*\len,1.25*\len) [anchor=south] {$a'$};

\node at (0.5*\len,1.5*0.375*\len) [anchor=east] {$a$};

\draw (-0.75*\len,1*\len) node {${f}$};
\draw (0*\len,1*\len) node {${\beta^j}$};
\draw (0,0.325*\len) node {${\pi}$};
\node at (-0.65*\len,0.25*\len) {\footnotesize{$\overline{\varphi}_{\!k}$}};
\node at (0.7*\len,0.45*\len) {\footnotesize{$\varphi^k$}};
\node at (-1.15*\len,0.65*\len) {\footnotesize{$\overline{\nu}_i$}};
\node at (1*\len,0.75*\len) {\footnotesize{$\nu^i$}};
\end{tikzpicture}
\end{equation*}
that can be further simplified by base change~\eqref{eq:base_change} which translates $f$ to the right side of $\beta^j$. Lemma~\ref{bases_comp} gives us a pair of dual bases for the space $\Hsp{\overline{x}\overline{y}a'w}$ common to both expressions, thereby showing the desired result. A similar argument implies that the braiding preserves relation~\eqref{eq:Rel2}.
%%%%%%%%%%%%%%%%%%%%%%%%%%%%%%%%%%%%%%%%%%%%%%%%%%
Finally, one can show that the assignment given by 
%\vspace{-0.5em}
\def\len{2cm}
\begin{equation*}
n\otimes m\otimes
\pielem{b}{a}{c}{y}{\upsilon}
\mapsto
\sum_{z\in\Irr\cx}  n \otimes m.\twist_{a}\otimes
\begin{tikzpicture}[yscale=1,xscale=0.9,baseline={([yshift=-.5ex]current bounding box.center)}]
 \draw[line width=1]
    (-0.5*\len,0) .. controls (-1*\len,0*\len) and (-1*\len,2*0.375*\len).. (-0.5*\len,0.75*\len)
    node[midway,rotate=-90] {\scriptsize $\Arr$}
    node[midway,left] {$z$};
% \draw[line width=0.75]
%     (-0.5*\len,0.375*\len) .. controls (-0.9*\len,1.15*0.375*\len) and (-0.9*\len,1.85*0.375*\len).. (-0.5*\len,0.75*\len);
  
\draw[line width=1] 
    (-0.5*\len,0) to (0.5*\len,0)
    (-0.5*\len,0) to (-0.5*\len,0.75*\len)
    (0.5*\len,0) to (0.5*\len,0.75*\len)
    (-0.5*\len,0.75*\len) to (0.5*\len,0.75*\len)
    (0.5*\len,0.75*\len) to (1*\len,0.75*\len)
    (0.5*\len,0) to[in=-90,out=0] (1*\len,0.75*\len);

% \node[rotate=-90] at (-0.8*\len,1.5*0.375*\len) {\scriptsize $\Arr$};
\node[rotate=-115] at (0.94*\len,0.375*\len) {\scriptsize $\Arr$};
\node[rotate=180] at (0,0) {\scriptsize $\Arr$};

\node[rotate=180] at (0.75*\len,0.75*\len) {\scriptsize $\Arr$};
\node[rotate=180] at (0,0.75*\len) {\scriptsize $\Arr$};
% \node[rotate=-90] at (-1.05*\len,0.375*\len) {\scriptsize $\Arr$};
\node[rotate=-90] at (-0.5*\len,1.5*0.375*\len) {\scriptsize $\Arr$};

\filldraw [black] (-0.5*\len,0) circle (1pt);
\filldraw [black] (0.5*\len,0) circle (1pt);
\filldraw [black] (-0.5*\len,0.75*\len) circle (1pt);
\filldraw [black] (-0.5*\len,0.375*\len) circle (1pt);
\filldraw [black] (0.5*\len,0.75*\len) circle (1pt);
\filldraw [black] (\len,0.75*\len) circle (1pt);

% Text Node
\node at (-0.45*\len,0.5*0.375*\len)  {$\bullet_y$};
\node at (0.55*\len,0.375*\len) {$\bullet_y$};
\node at (0,0) [anchor=north] {$c$};
\node at (0,0.75*\len) [anchor=south] {$b$};
\node at (0.75*\len,0.75*\len) [anchor=south] {$a$};
\node at (-0.5*\len,1.5*0.375*\len) [anchor=west] {$a$};
% \node at (-1*\len,1.5*0.325*\len) [anchor=west,font=\footnotesize] {$\tilde{a}$};
% \node at (-1.1*\len,0.375*\len) [anchor=east] {$z$};
\node at (1*\len,0.35*\len) [anchor=west] {$z$};

\draw (0,0.375*\len) node {$ {\upsilon}$};
\node at (-0.7*\len,0.375*\len) {{$\overline{\psi}_{\!k}$}};
\node at (0.8*\len,0.45*\len) {{$\psi^k$}};
% \node at (-0.635*\len,1.5*0.375*\len) {\footnotesize {$\alpha^i$}};
\end{tikzpicture}
\end{equation*}
provides the inverse braiding $\mathrm{B}_{M,N}^{-1}\colon N\Boxt M\to M\Boxt N$, by means of Lemma~\ref{bases_comp} and~\eqref{eq:rel_simple_decomp1}.
\end{proof}

\begin{prop}\label{prop:hexagon_axioms}
Let $M,N$ and $L$ be tube algebra representations. The hexagon axioms
{
% \scriptsize
\begin{equation*}
\begin{tikzcd}[column sep=large]
&(M\Boxt N)\Boxt L\ar[r,"\quad~\mathrm{A}_{M,N,L}\quad~"]
\ar[dl,"\mathrm{B}_{M,N}\Boxt \id_L",swap]
&M\Boxt (N\Boxt L)\ar[dr,"\mathrm{B}_{M,N\Boxt L}"]&\\
(N\Boxt M)\Boxt L\ar[dr,"\mathrm{A}_{N,M,L}",swap]
&&&(N\Boxt L)\Boxt M\ar[dl,"\mathrm{A}_{N,L,M}"]\\
&N\Boxt (M\Boxt L)\ar[r,"\id_N\Boxt\mathrm{B}_{M,L}",swap]&N\Boxt (L\Boxt M)&
\end{tikzcd}
\end{equation*}
\text{ and }
%%%%%%%%%%%%%%%%%%%%%%%%%
\begin{equation*}
\begin{tikzcd}[column sep=large]
&(M\Boxt N)\Boxt L\ar[r,"\mathrm{B}_{M\Boxt N,L}"]
&L\Boxt (M\Boxt N)\ar[dr,"\mathrm{A}^{-1}_{L,M,N}"]&\\
M\Boxt (N\Boxt L)\ar[ur,"\mathrm{A}^{-1}_{M,N,L}"]\ar[dr,"\id_M\Boxt\mathrm{B}_{N,L}",swap]
&&&(L\Boxt M)\Boxt N \ar[dl,"\mathrm{B}_{L,M}\Boxt\id_N"]\\
&M\Boxt (L\Boxt N )\ar[r,"\mathrm{A}^{-1}_{M,L,N}",swap]&(M\Boxt L)\Boxt N&
\end{tikzcd}
\end{equation*}}
are fulfilled by the braiding~\eqref{eq:braiding} and associators~\eqref{eq:associator}. 
\end{prop}
\begin{proof}
Owing to Proposition~\ref{prop:rel_unit_simple} (iii), it suffices to verify the hexagon for elements of the form
\def\len{1.4cm}
\begin{equation}\label{eq:three_product_element}
\left(m\otimes n\otimes \pielemUnit{a}{b}{d}{\pi}\right)\otimes
l\otimes \pielemUnit{d}{c}{v}{\rho}
\end{equation}
where $a,b,c,d\in\Irr\cx$. The image of~\eqref{eq:three_product_element} under $\mathrm{A}_{N,L,M}\circ\mathrm{B}_{M,N\Boxt L}\circ\mathrm{A}_{M,N,L}$ equals
\begin{equation}\label{eq:interstep_hex1}
    \rd_a^{-1}\, n\otimes(l\otimes m.\twist_a\otimes \Id_{ca})\otimes 
\begin{tikzpicture}[xscale=0.8,baseline={([yshift=-.5ex]current bounding box.center)}]
 \draw[line width=1] 
    (0.4*\len,0.75*\len) to node[midway,left] {\footnotesize$d$} node[midway,rotate=-55] {\scriptsize $\Arr$}
    (1.2*\len,0);

\draw[line width=1] 
    (-1.2*\len,0) to (1.2*\len,0)
    (-1.2*\len,0) to (-1.2*\len,0.75*\len)
    (1.2*\len,0) to  (1.2*\len,0.75*\len)
    (-1.2*\len,0.75*\len) to (1.2*\len,0.75*\len);
    
  \draw[line width=1,dotted] 
    (-1.2*\len,0)  to (-0.4*\len,0.75*\len);

\node[rotate=90] at (-1.2*\len,0.375*\len) {\scriptsize $\Arr$};
\node[rotate=90] at (1.2*\len,0.375*\len) {\scriptsize $\Arr$};
\node[rotate=180] at (0,0) {\scriptsize $\Arr$};
\node[rotate=180] at (-0.8*\len,0.75*\len) {\scriptsize $\Arr$};
\node[rotate=180] at (0.8*\len,0.75*\len) {\scriptsize $\Arr$};
\node[rotate=180] at (0,0.75*\len) {\scriptsize $\Arr$};

\filldraw [black] (-1.2*\len,0) circle (1pt);
\filldraw [black] (1.2*\len,0) circle (1pt);
\filldraw [black] (0.4*\len,0.75*\len) circle (1pt);
\filldraw [black] (-0.4*\len,0.75*\len) circle (1pt);
\filldraw [black] (-1.2*\len,0.75*\len) circle (1pt);
\filldraw [black] (1.2*\len,0.75*\len) circle (1pt);

% Text Node
\node at (-1.2*\len,0.375*\len) [anchor=east] {$a$};
\node at (1.2*\len,0.375*\len) [anchor=west] {$a$};
\node at (0,0) [anchor=north] {$v$};
\node at (0.8*\len,0.75*\len) [anchor=south] {$b$};
\node at (0,0.75*\len) [anchor=south] {$c$};
\node at (-0.8*\len,0.75*\len) [anchor=south] {$a$};

\draw (1*\len,0.55*\len) node {${\pi}$};
\draw (-0.15*\len,0.375*\len) node {${\rho}$};
\node at (-0.9*\len,0.55*\len) {\footnotesize{$\Id_a$}};
\end{tikzpicture}
\end{equation}    
On the other hand, $\id_N\Boxt\mathrm{B}_{M,L}\circ\mathrm{A}_{N,M,L}\circ\mathrm{B}_{M,N}\Boxt \id_L$ assigns to~\eqref{eq:three_product_element} the value
\begin{equation}\label{eq:interstep_hex}
\rd_a^{-2}\sum_{\tilde{c}\in\Irr\cx} \rd_{\tilde{c}}\, n\otimes\underbrace{\left(l.\frac{\overline{\gamma}_k}{\rd_a}\otimes m.\twist_a^2\otimes
\begin{tikzpicture}[xscale=0.8,yscale=0.9,baseline={([yshift=-.5ex]current bounding box.center)}]
    \draw[line width=1] 
     (1.2*\len,0.75*\len) to (1.2*\len,1.5*\len)
    (-0.4*\len,0.75*\len) to (-0.4*\len,1.5*\len)
    (-0.4*\len,1.5*\len) to (1.2*\len,1.5*\len)
    (-0.4*\len,0.75*\len) to (1.2*\len,0.75*\len);
    \draw[line width=1,dotted]
    (-0.4*\len,0.75*\len) to (0.4*\len,1.5*\len); 
\node[rotate=90] at (-0.4*\len,1.125*\len) {\scriptsize $\Arr$};
\node[rotate=90] at (1.2*\len,1.125*\len) {\scriptsize $\Arr$};
\node[rotate=180] at (0,0.75*\len) {\scriptsize $\Arr$};
\node[rotate=180] at (0.8*\len,0.75*\len) {\scriptsize $\Arr$};
\node[rotate=180] at (0.8*\len,1.5*\len) {\scriptsize $\Arr$};
\node[rotate=180] at (0,1.5*\len) {\scriptsize $\Arr$};
\filldraw [black] (-0.4*\len,0.75*\len) circle (1pt);
\filldraw [black] (1.2*\len,0.75*\len) circle (1pt);
\filldraw [black] (1.2*\len,1.5*\len) circle (1pt);
\filldraw [black] (0.4*\len,1.5*\len) circle (1pt);
\filldraw [black] (0.4*\len,0.75*\len) circle (1pt);
\filldraw [black] (-0.4*\len,1.5*\len) circle (1pt);
\node at (-0.4*\len,1.125*\len) [anchor=east,font=\footnotesize] {$a$};
%
% Text Node
\node at (1.2*\len,1.125*\len)  [anchor=west]{$a$};
\node at (0,0.75*\len) [anchor=north] {$\tilde{c}$};
\node at (0.8*\len,0.75*\len) [anchor=north] {$a$};
\node at (0.8*\len,1.5*\len) [anchor=south] {$\tilde{c}$};
\node at (0,1.5*\len) [anchor=south] {$a$};
\node at (0.55*\len,1.12*\len) {\footnotesize{$\Id_{a\tilde{c}}$}};
\node at (-0.15*\len,1.3*\len) {\footnotesize{$\Id_a$}};
\end{tikzpicture}
\right)}_{x}\otimes 
% \raisebox{-2.5em}{}   
\begin{tikzpicture}[xscale=0.8,yscale=0.9,baseline={([yshift=-.5ex]current bounding box.center)}]
    \draw[line width=1] 
    (-1.2*\len,0) to (1.2*\len,0)
    (1.2*\len,0.75*\len) to (1.2*\len,1.5*\len)
    (-0.4*\len,0.75*\len) to (-0.4*\len,1.5*\len)
    (-1.2*\len,0.75*\len) to (-1.2*\len,1.5*\len)
    (-1.2*\len,1.5*\len) to (1.2*\len,1.5*\len)
    (-1.2*\len,0.75*\len) to (1.2*\len,0.75*\len);
    \draw[line width=1,dotted]
    (-1.2*\len,0) to (-1.2*\len,0.75*\len)
    (1.2*\len,0) to (1.2*\len,0.75*\len)
    (-0.4*\len,0.75*\len) to (0.4*\len,1.5*\len); 

\node[rotate=90] at (-0.4*\len,1.125*\len) {\scriptsize $\Arr$};
\node[rotate=90] at (-1.2*\len,1.125*\len) {\scriptsize $\Arr$};
\node[rotate=90] at (1.2*\len,1.125*\len) {\scriptsize $\Arr$};
\node[rotate=180] at (0,0) {\scriptsize $\Arr$};
\node[rotate=180] at (-0.8*\len,0.75*\len) {\scriptsize $\Arr$};
\node[rotate=180] at (0.4*\len,0.75*\len) {\scriptsize $\Arr$};
\node[rotate=180] at (0.8*\len,1.5*\len) {\scriptsize $\Arr$};
\node[rotate=180] at (0,1.5*\len) {\scriptsize $\Arr$};
\node[rotate=180] at (-0.8*\len,1.5*\len) {\scriptsize $\Arr$};

\filldraw [black] (-1.2*\len,0) circle (1pt);
\filldraw [black] (1.2*\len,0) circle (1pt);
\filldraw [black] (-0.4*\len,0.75*\len) circle (1pt);
\filldraw [black] (-1.2*\len,0.75*\len) circle (1pt);
\filldraw [black] (1.2*\len,0.75*\len) circle (1pt);
\filldraw [black] (1.2*\len,1.5*\len) circle (1pt);
\filldraw [black] (0.4*\len,1.5*\len) circle (1pt);
\filldraw [black] (-1.2*\len,1.5*\len) circle (1pt);
\filldraw [black] (-0.4*\len,1.5*\len) circle (1pt);

\node at (-0.4*\len,1*\len) [anchor=east,font=\footnotesize] {$a$};

% Text Node
\node at (-1.2*\len,1*\len) [anchor=east] {$a$};
\node at (1.2*\len,1*\len)  [anchor=west]{$a$};
\node at (0,0) [anchor=north] {$v$};
\node at (0.4*\len,0.75*\len) [anchor=north] {$d$};
\node at (0.8*\len,1.5*\len) [anchor=south] {$b$};
\node at (0,1.5*\len) [anchor=south] {$a$};
\node at (-0.8*\len,1.5*\len) [anchor=south] {$\tilde{c}$};
\node at (-0.8*\len,0.75*\len) [anchor=north] {$c$};
\draw (0,0.35*\len) node {${\rho}$};
\draw (-0.8*\len,1.2*\len) node {{$\gamma^k$}};
\node at (0.5*\len,1.125*\len) {{$\pi$}};
\node at (-0.15*\len,1.3*\len) {\footnotesize{$\Id_a$}};
\end{tikzpicture}
\end{equation}
Now, using relation~\eqref{eq:Rel2}, it follows that in $L\Boxt M$ the vector $x$ is related to 
\def\len{2cm}
\begin{equation*}
\rd_a^{-1}\;\sum_{t\in\Irr\cx}\rd_t\;l.(\overline{\gamma}_k\cdot\overline{\eta}_j)\otimes m.\twist_a\otimes
\begin{tikzpicture}[xscale=1,yscale=1,baseline={([yshift=-.5ex]current bounding box.center)}]
\draw[line width=1] 
     (0,0) to (\len,0)
    (-0.5*\len,0.5*\len) to (0.5*\len,0.5*\len)
    (0.5*\len,0.5*\len) to (0.5*\len,0)
    (0,0.5*\len) to (0,0);

\draw[line width=1,dotted]
    (-0.5*\len,0.5*\len) to (0,0)
    (0.5*\len,0.5*\len) to (\len,0); 
    
\node[rotate=180] at (0.25*\len,0.5*\len) {\scriptsize $\Arr$};
\node[rotate=180] at (0.25*\len,0) {\scriptsize $\Arr$};
\node[rotate=180] at (-0.25*\len,0.5*\len) {\scriptsize $\Arr$};
\node[rotate=180] at (0.75*\len,0) {\scriptsize $\Arr$};
\node[rotate=-90] at (0,0.25*\len) {\scriptsize $\Arr$};
\node[rotate=-90] at (0.5*\len,0.25*\len) {\scriptsize $\Arr$};
\filldraw [black] (0,0) circle (1pt);
\filldraw [black] (\len,0) circle (1pt);
\filldraw [black] (-0.5*\len,0.5*\len) circle (1pt);
\filldraw [black] (0.5*\len,0.5*\len) circle (1pt);
\filldraw [black] (0.5*\len,0) circle (1pt);
\filldraw [black] (0,0.5*\len) circle (1pt);
%
% \node at (-0.4*\len,1.125*\len) [anchor=east,font=\footnotesize] {$a$};
% %
% % Text Node
% \node at (1.2*\len,1.125*\len)  [anchor=west]{$a$};
% \node at (0,0.75*\len) [anchor=north] {$\tilde{c}$};
\node at (0.25*\len,0.5*\len) [anchor=south] {$t$};
\node at (-0.25*\len,0.5*\len) [anchor=south] {$a$};
\node at (0.75*\len,0) [anchor=north] {$a$};
\node at (0.25*\len,0) [anchor=north] {$\tilde{c}$};
% \node at (0,0.3*\len) [anchor=west] {\footnotesize$a$};
% \node at (0.5*\len,0.2*\len) [anchor=east] {\footnotesize$a$};
\node at (0.25*\len,0.25*\len) {\footnotesize{$\eta^j$}};
\node at (0.7*\len,0.15*\len) {\footnotesize{$\Id_a$}};
\node at (-0.2*\len,0.35*\len) {\footnotesize{$\Id_a$}};
\end{tikzpicture}
\end{equation*}
and thus, again by~\eqref{eq:Rel2}, we have that in $N\Boxt (L\Boxt M)$ the vector~\eqref{eq:interstep_hex} is related to 
\def\len{1.4cm}
\begin{equation*}
\rd_a^{-3}\sum_{\tilde{c},t\in\Irr\cx} \rd_{\tilde{c}}\,\rd_t\, n\otimes\left(l.(\overline{\gamma}_k\cdot\overline{\eta}_j)\otimes m.\twist_a\otimes \Id_{ta}
\right)\otimes 
\begin{tikzpicture}[xscale=0.8,yscale=0.9,baseline={([yshift=-.5ex]current bounding box.center)}]
    \draw[line width=1] 
    (-1.2*\len,0) to (1.2*\len,0)
    (1.2*\len,0.75*\len) to (1.2*\len,1.5*\len)
    (-0.4*\len,0.75*\len) to (-0.4*\len,1.5*\len)   
    (-1.2*\len,0.75*\len) to (-1.2*\len,1.5*\len)
    (-1.2*\len,1.5*\len) to (1.2*\len,1.5*\len)
    (-1.2*\len,2.25*\len) to (0.4*\len,2.25*\len)    
    (-1.2*\len,0.75*\len) to (1.2*\len,0.75*\len);
    \draw[line width=1,dotted]
    (0.4*\len,2.25*\len) to (0.4*\len,1.5*\len)     
    (-1.2*\len,0) to (-1.2*\len,0.75*\len)
    (1.2*\len,0) to (1.2*\len,0.75*\len)
    (-0.4*\len,0.75*\len) to (0.4*\len,1.5*\len)
    (-1.2*\len,1.5*\len) to (-1.2*\len,2.25*\len)    ; 

\node[rotate=90] at (-0.4*\len,1.125*\len) {\scriptsize $\Arr$};
\node[rotate=90] at (-1.2*\len,1.125*\len) {\scriptsize $\Arr$};
\node[rotate=90] at (1.2*\len,1.125*\len) {\scriptsize $\Arr$};
\node[rotate=180] at (0,0) {\scriptsize $\Arr$};
\node[rotate=180] at (-0.8*\len,0.75*\len) {\scriptsize $\Arr$};
\node[rotate=180] at (0.4*\len,0.75*\len) {\scriptsize $\Arr$};
\node[rotate=180] at (0.8*\len,1.5*\len) {\scriptsize $\Arr$};
\node[rotate=180] at (0,1.5*\len) {\scriptsize $\Arr$};
\node[rotate=180] at (-0.8*\len,1.5*\len) {\scriptsize $\Arr$};
\node[rotate=180] at (0,2.25*\len) {\scriptsize $\Arr$};
\node[rotate=180] at (-0.8*\len,2.25*\len) {\scriptsize $\Arr$};

\filldraw [black] (-1.2*\len,0) circle (1pt);
\filldraw [black] (1.2*\len,0) circle (1pt);
\filldraw [black] (-0.4*\len,0.75*\len) circle (1pt);
\filldraw [black] (-1.2*\len,0.75*\len) circle (1pt);
\filldraw [black] (-1.2*\len,2.25*\len) circle (1pt);
\filldraw [black] (-0.4*\len,2.25*\len) circle (1pt);
\filldraw [black] (0.4*\len,2.25*\len) circle (1pt);
\filldraw [black] (1.2*\len,0.75*\len) circle (1pt);
\filldraw [black] (1.2*\len,1.5*\len) circle (1pt);
\filldraw [black] (0.4*\len,1.5*\len) circle (1pt);
\filldraw [black] (-1.2*\len,1.5*\len) circle (1pt);
\filldraw [black] (-0.4*\len,1.5*\len) circle (1pt);

\node at (-0.4*\len,1*\len) [anchor=east,font=\footnotesize] {$a$};

% Text Node
\node at (-1.2*\len,1*\len) [anchor=east] {$a$};
\node at (1.2*\len,1*\len)  [anchor=west]{$a$};
\node at (0,0) [anchor=north] {$v$};
\node at (0.4*\len,0.75*\len) [anchor=north] {$d$};
\node at (0.8*\len,1.5*\len) [anchor=south] {$b$};
\node at (0,1.5*\len) [anchor=south] {$a$};
\node at (-0.8*\len,1.5*\len) [anchor=south] {$\tilde{c}$};
\node at (0,2.25*\len) [anchor=south] {$t$};
\node at (-0.8*\len,2.25*\len) [anchor=south] {$a$};
\node at (-0.8*\len,0.75*\len) [anchor=north] {$c$};
\draw (0,0.35*\len) node {${\rho}$};
\draw (-0.8*\len,1.2*\len) node {{$\gamma^k$}};
\draw (-0.4*\len,1.95*\len) node {{$\eta^j$}};
\node at (0.5*\len,1.125*\len) {{$\pi$}};
\node at (-0.15*\len,1.3*\len) {\footnotesize{$\Id_a$}};
\end{tikzpicture}
\end{equation*}
According to Lemma~\ref{star_prop} (v) we have that
\begin{equation*}
\sum_{\tilde{c}\in\Irr\cx} \rd_{\tilde{c}}\,\overline{\gamma}_k\cdot\overline{\eta}_j\otimes (\eta^j\circ_{\tilde{c}}\gamma^k)\circ_a\Id_a=\rd_a^2\;\overline{\zeta}_i\otimes \Id_a\circ_\un\zeta^i     
\end{equation*}
where $(\overline{\zeta},\zeta)$ is a pair of dual bases of $\Hsp{\overline{t}\,c}$, but since $c\in\Irr\cx$, it follows that $t=c$. We can choose the bases $\zeta=\id_c$ and $\overline{\zeta}=\rd_c^{-1}\id_c$ to finally conclude that the expressions~\eqref{eq:interstep_hex} and~\eqref{eq:interstep_hex1} agree. In a similar manner, one can verify the second hexagon axiom.
\end{proof}

\begin{defi}
The \textit{twist} of $M\in\Rep \Tubecx$ is given by the assignment
\begin{equation*}
    \Theta_M\colon M\to M,\quad m_a\mapsto m_a.\twist_a\;.
\end{equation*}
\end{defi}
\begin{prop}\label{prop:twist}
$\Theta$ defines a twist on $\Rep \Tubecx$ in the sense of\rm{~\cite[Definition 8.10.1]{EGNO}}.   
\end{prop}
\begin{proof}
It is clear that $\Theta_M$ is a natural isomorphism, with inverse given by the action of the inverse twist element. That $\Theta_M$ is an intertwiner follows from~\eqref{eq:twist-central}. Given tube algebra representations $M,N\in\Rep\Tubecx$, one can check that $\Theta_{M}\Boxt \Theta_{N}\circ \mathrm{B}_{N,M}\circ\mathrm{B}_{M,N}$ and $\Theta_{M\Boxt N}$ agree, by a direct computation involving~\eqref{eq:twist-central} and Lemma~\ref{bases_comp}.
\end{proof}

\begin{thm}
Let $\cx$ be a spherical semisimple multitensor category. The data $\left(\Boxt,\mathbb{I},\mathrm{A},\mathrm{B},\Theta\right)$ endow the $\bk$-linear category $\Rep\Tubecx$ with the structure of a braided monoidal category with twist. Moreover, the full $\bk$-linear subcategory $\Rep^{\mathrm{l.f.}}\Tubecx$ of locally finite tube algebra representations is a ribbon semisimple tensor category.
\end{thm}
\begin{proof}
The statement compiles the results of Proposition~\ref{prop:tensor_prod},~\ref{prop:pentagon_axiom},~\ref{prop:hexagon_axioms} and~\ref{prop:twist}. The rigidity of the subcategory $\Rep^{\mathrm{l.f.}}\Tubecx$ was shown in Proposition~\ref{prop:rigidity}. It remains to be checked that $\Theta_{\overline{M}}=\overline{\Theta_M}$. A computation similar to~\eqref{prop:rigidity} shows that for $\varphi\in\overline{M}_a$ it holds that 
\begin{equation*}
    \overline{\Theta_M}(\varphi)= \textstyle{\sum_{j}}\;\varphi(
m_a^j.\twist_a)  \;
\overline{m}_a^j\,,
\end{equation*}
but, since $\twist_a^\#=\twist_{\overline{a}}$ and $\{m_a^j\}$ and $\{\overline{
m}_a^j\}$ are dual bases, the result folows.
\end{proof}

\section{Tube algebra representations and the Drinfeld center}\label{sec:drinfeld_center}

We are now ready to compare the category of representations of $\Tubecx$ and the Drinfeld center of $\cx$.
For simplicity on the arguments, we will consider from now on that $\rI$ in Definition~\ref{def:data-for-tube-alg} given by $\mathrm{I}_i=\Irr\cx_{i }$, even though the statements given next hold for a larger indexing set.

As shown in~\cite[Proposition 4.2]{NY15tubealgebra}, there is an equivalence of $\bk$-linear categories
\begin{equation}\label{eq:equivalence}
    \EQ\Colon\mathcal{Z}(\Ind\cx)\xrightarrow{\quad\simeq\quad}\Rep \Tube,
\end{equation}
where $\EQ(X) = \EQ(X, \sigma)$ is given by the direct summands $\EQ(X)_d = \Hsp{\overline{d}X}$ for $d \in \Irr\cx_{i }$, together with the right tube algebra action defined by the formula%\vspace{-1em}
\begin{equation*}
m.f\coloneqq\rd_x
\begin{tikzpicture}[xscale=0.8,yscale=0.9,baseline={([yshift=-.5ex]current bounding box.center)}]
\draw[line width=1] 
    (-0.5*\len,0) to (0.5*\len,0)
    (-0.5*\len,0) to (-0.5*\len,0.5*\len)
    (0.5*\len,0) to (0.5*\len,0.5*\len)
    (-0.5*\len,0.5*\len) to (0.5*\len,0.5*\len);
\draw[line width=1]
    (0.5*\len,0.5*\len) .. controls (0.35*\len,0.9*\len) and (-0.35*\len,0.9*\len).. (-0.5*\len,0.5*\len)
    node[midway,above] {\scriptsize $X$}
    node[midway,rotate=180] {\scriptsize $\Arr$};
\draw[line width=1]
    (0.5*\len,0) .. controls (0.75*\len,0.25*\len) and (0.75*\len,1.25*\len) .. (0,1.25*\len)
    node[above] {\scriptsize $X$}
    node[rotate=180] {\scriptsize $\Arr$};    
\draw[line width=1]
    (-0.5*\len,0) .. controls (-0.75*\len,0.25*\len) and (-0.75*\len,1.25*\len) .. (0,1.25*\len);    

\node[rotate=-90] at (-0.5*\len,0.25*\len) {\scriptsize $\Arr$};
\node[rotate=-90] at (0.5*\len,0.25*\len) {\scriptsize $\Arr$};
\node[rotate=180] at (0,0) {\scriptsize $\Arr$};
\node[rotate=180] at (0,0.5*\len) {\scriptsize $\Arr$};

\filldraw [black] (-0.5*\len,0) circle (1pt);
\filldraw [black] (0.5*\len,0) circle (1pt);
\filldraw [black] (-0.5*\len,0.5*\len) circle (1pt);
\filldraw [black] (0.5*\len,0.5*\len) circle (1pt);
% Text Node
\node at (-0.5*\len,0.25*\len) [anchor=west] {\footnotesize$x$};
\node at (0.5*\len,0.25*\len) [anchor=east] {\footnotesize$x$};
\node at (0,0) [anchor=north] {$b$};
\node at (0,0.5*\len) [anchor=north] {\footnotesize$a$};
\draw (0,0.2*\len) node  {${f}$};
\draw (0,0.65*\len) node  {${m}$};
\draw (-0.15*\len,1.1*\len) node  {\footnotesize$\sigma_{x}$};
\end{tikzpicture}
\end{equation*}
for $m\in \Hsp{\overline{a}X}$ and $f\in\Tube$. Module associativity follows from the naturality of $\sigma$ and Lemma~\ref{lem:pairing_properties} (ii). Locality of the action is ensured since $\sigma_\un=\id_X$.

The purpose of this section is to explicitly define a braided monoidal structure on the equivalence~\eqref{eq:equivalence}. To this end, consider for central objects $X,Y\in\cz(\Ind\cx)$ the assignment
\begin{equation}\label{eq:monoidal_str_T}
\begin{aligned}
\Psi_{X,Y}\Colon \EQ(X)\Boxt \EQ(Y)&\to \EQ(X\otimes Y)\\
\centrelem{s}{X}{m}\otimes \centrelem{t}{Y}{n}\otimes \pielem{s}{t}{c}{x}{\pi}
&\mapsto
\begin{tikzpicture}[xscale=0.8,baseline={([yshift=-.5ex]current bounding box.center)}]
\draw[line width=1]
    (0,0.5*\len) .. controls (0.15*\len,0.9*\len) and (0.6*\len,0.9*\len).. node[midway,rotate=180] {\scriptsize $\Arr$}(0.75*\len,0.5*\len)
    node[midway,above] {\scriptsize $X$};
\draw[line width=1]
    (0,0.5*\len) .. controls (-0.15*\len,0.9*\len) and (-0.6*\len,0.9*\len).. node[midway,rotate=180] {\scriptsize $\Arr$}(-0.75*\len,0.5*\len)
    node[midway,above] {\scriptsize $Y$};    
\draw[line width=1]
    (0.75*\len,0) .. controls (1.25*\len,0.25*\len) and (1.25*\len,1.25*\len).. node[midway,rotate=125] {\scriptsize $\Arr$}(0,1.25*\len) 
    node[midway,right] {\footnotesize $X$};
\draw[line width=1]
    (-0.75*\len,0) .. controls (-1.25*\len,0.25*\len) and (-1.25*\len,1.25*\len).. node[midway,rotate=245] {\scriptsize $\Arr$}(0,1.25*\len) 
    node[midway,left] {\footnotesize $Y$};

\draw[line width=1] 
    (-0.75*\len,0) to (0.75*\len,0)
    (-0.75*\len,0) to (-0.75*\len,0.5*\len)
    (0.75*\len,0) to (0.75*\len,0.5*\len)
    (-0.75*\len,0.5*\len) to (0.75*\len,0.5*\len);

\node[rotate=-90] at (-0.75*\len,0.25*\len) {\scriptsize $\Arr$};
\node[rotate=-90] at (0.75*\len,0.25*\len) {\scriptsize $\Arr$};
\node[rotate=180] at (0,0) {\scriptsize $\Arr$};
\node[rotate=180] at (-0.4*\len,0.5*\len) {\scriptsize $\Arr$};
\node[rotate=180] at (0.4*\len,0.5*\len) {\scriptsize $\Arr$};

\filldraw [black] (-0.75*\len,0) circle (1pt);
\filldraw [black] (0.75*\len,0) circle (1pt);
\filldraw [black] (0,0.5*\len) circle (1pt);
\filldraw [black] (0,1.25*\len) circle (1pt);
\filldraw [black] (-0.75*\len,0.5*\len) circle (1pt);
\filldraw [black] (0.75*\len,0.5*\len) circle (1pt);

% Text Node
\node at (-0.75*\len,0.25*\len) [anchor=west] {$x$};
\node at (0.75*\len,0.25*\len) [anchor=east] {$x$};
\node at (0,0) [anchor=north] {$c$};
\node at (0.4*\len,0.5*\len) [anchor=north] {$s$};
\node at (-0.4*\len,0.5*\len) [anchor=north] {$t$};

\draw (0,1.05*\len) node {\footnotesize${\sigma^{XY}_x}$};
\draw (0,0.25*\len) node {${\pi}$};
\draw (0.375*\len,0.65*\len) node {$m$};
\draw (-0.375*\len,0.65*\len) node {$n$};
\end{tikzpicture}
\end{aligned}
\end{equation}

\begin{prop}
The assignment~\eqref{eq:monoidal_str_T} is a well-defined isomorphism of tube algebra representations.
\end{prop}
\begin{proof}
We verify the preservation of relation~\eqref{eq:Rel1}, a vector $m.f\otimes n\otimes \pi$ gets assigned the value
\begin{equation}\label{eq:inter_step_m_str1}
\rd_x\rd_y\;
\begin{tikzpicture}[xscale=0.8,baseline={([yshift=-.5ex]current bounding box.center)}]
\draw[line width=1]
    (0,1*\len) .. controls (0.15*\len,1.4*\len) and (0.6*\len,1.4*\len).. node[midway,rotate=180] {\scriptsize $\Arr$}(0.75*\len,1*\len)
    node[midway,above] {\scriptsize $X$};
\draw[line width=1]
    (0,0.5*\len) .. controls (-0.15*\len,0.9*\len) and (-0.6*\len,0.9*\len).. node[midway,rotate=180] {\scriptsize $\Arr$}(-0.75*\len,0.5*\len)
    node[midway,above] {\scriptsize $Y$};

\draw[line width=1]
    (0,0.5*\len) to node[pos=0.8,rotate=105] {\scriptsize $\Arr$}
    node[pos=0.7,left] {\scriptsize $x$}(-0.25*\len,1.5*\len);    

\draw[line width=1]
    (0.75*\len,0) .. controls (1.75*\len,0.25*\len) and (2*\len,1.8*\len).. node[midway,rotate=105] {\scriptsize $\Arr$}(-0.25*\len,1.5*\len)
    node[midway,right] {\footnotesize $X$};

\draw[line width=1]
    (-0.75*\len,0) .. controls (-1.25*\len,0.25*\len) and (-1.25*\len,1.25*\len).. node[midway,rotate=245] {\scriptsize $\Arr$}(-0.25*\len,1.5*\len) 
    node[midway,left] {\footnotesize $Y$};

\draw[line width=1] 
    (-0.75*\len,0) to (0.75*\len,0)
    (-0.75*\len,0) to (-0.75*\len,0.5*\len)
    (0.75*\len,0) to (0.75*\len,1*\len)
    (-0.75*\len,0.5*\len) to (0.75*\len,0.5*\len)
    (0,0.5*\len) to (0,1*\len)
    (0,1*\len) to (0.75*\len,1*\len);

\node[rotate=-90] at (-0.75*\len,0.25*\len) {\scriptsize $\Arr$};
\node[rotate=-90] at (0.75*\len,0.25*\len) {\scriptsize $\Arr$};
\node[rotate=-90] at (0,0.75*\len) {\scriptsize $\Arr$};
\node[rotate=-90] at (0.75*\len,0.75*\len) {\scriptsize $\Arr$};
\node[rotate=180] at (0,0) {\scriptsize $\Arr$};
\node[rotate=180] at (-0.4*\len,0.5*\len) {\scriptsize $\Arr$};
\node[rotate=180] at (0.4*\len,0.5*\len) {\scriptsize $\Arr$};
\node[rotate=180] at (0.4*\len,1*\len) {\scriptsize $\Arr$};

\filldraw [black] (-0.75*\len,0) circle (1pt);
\filldraw [black] (0.75*\len,0) circle (1pt);
\filldraw [black] (0,0.5*\len) circle (1pt);
\filldraw [black] (0,1*\len) circle (1pt);
\filldraw [black] (-0.75*\len,0.5*\len) circle (1pt);
\filldraw [black] (0.75*\len,0.5*\len) circle (1pt);
\filldraw [black] (0.75*\len,1*\len) circle (1pt);
\filldraw [black] (-0.25*\len,1.5*\len) circle (1pt);

% Text Node
\node at (-0.75*\len,0.25*\len) [anchor=west] {\footnotesize$x$};
\node at (0.75*\len,0.25*\len) [anchor=east] {\footnotesize$x$};
\node at (0,0.75*\len) [anchor=west] {\footnotesize$y$};
\node at (0.75*\len,0.75*\len) [anchor=east] {\footnotesize$y$};

\node at (0,0) [anchor=north] {$c$};
\node at (0.4*\len,0.5*\len) [anchor=north] {\footnotesize$s$};
\node at (0.4*\len,1.03*\len) [anchor=north] {\footnotesize$s'$};
\node at (-0.4*\len,0.5*\len) [anchor=north] {\footnotesize$t$};

\draw (1.1*\len,0.75*\len) node {\footnotesize$\sigma^{X}_{{xy}}$};
\draw (-0.65*\len,1*\len) node {\footnotesize$\sigma^{Y}_{{x}}$};
\draw (0.35*\len,0.68*\len) node {${f}$};
\draw (0,0.25*\len) node {${\pi}$};
\draw (0.375*\len,1.15*\len) node {$m$};
\draw (-0.375*\len,0.65*\len) node {$n$};
\end{tikzpicture}
\end{equation}
by the map $\Psi_{X,Y}$. On the other hand,
%%%%%%%%%%%%%%%%%%%%%%%%%%%%%%%
\begin{equation}\label{eq:inter_step_m_str2}
\sum_{z\in\Irr\cx}\rd_x\rd_y\rd_z\;
\begin{tikzpicture}[xscale=0.8,baseline={([yshift=-.5ex]current bounding box.center)}]
\draw[line width=1]
(0.75*\len,0) .. controls (1.25*\len,0.25*\len) and (1.25*\len,0.75*\len).. (0.75*\len,\len) node[midway, rotate=-90] {\scriptsize $\Arr$} node[midway,right] {\small $z$};
\draw[line width=1]
    (0,1*\len) .. controls (0.15*\len,1.4*\len) and (0.6*\len,1.4*\len).. node[midway,rotate=180] {\scriptsize $\Arr$}(0.75*\len,1*\len)
    node[midway,above] {\scriptsize $X$};
\draw[line width=1]
    (0,0.5*\len) .. controls (-0.15*\len,0.9*\len) and (-0.6*\len,0.9*\len).. node[midway,rotate=180] {\scriptsize $\Arr$}(-0.75*\len,0.5*\len)
    node[midway,above] {\scriptsize $Y$};

\draw[line width=1]
    (0,1*\len) to node[pos=0.7,rotate=95] {\scriptsize $\Arr$}
    node[pos=0.7,left] {\scriptsize $z$}(-0.15,1.75*\len);    

\draw[line width=1]
    (0.75*\len,0) .. controls (1.85*\len,0.15*\len) and (2*\len,1.8*\len).. node[midway,rotate=105] {\scriptsize $\Arr$}(-0.15,1.75*\len)
    node[midway,right] {\footnotesize $X$};

\draw[line width=1]
    (-0.75*\len,0) .. controls (-2*\len,0*\len) and (-2.25*\len,1.75*\len).. node[midway,rotate=255] {\scriptsize $\Arr$}(-0.15,1.75*\len)
    node[midway,left] {\footnotesize $Y$};

\draw[line width=1] 
    (-0.75*\len,0) to (0.75*\len,0)
    (-0.75*\len,0) to (-0.75*\len,0.5*\len)
    (0.75*\len,0) to (0.75*\len,1*\len)
    (-1.25*\len,0.5*\len) to (0.75*\len,0.5*\len)
    (0,0.5*\len) to (0,1*\len)
    (0,1*\len) to (0.75*\len,1*\len)
    (-0.75*\len,0) to [in=-90,out=180] node[pos=0.8, rotate=-75]  {\scriptsize $\Arr$} node[pos=0.8,left] {\footnotesize $z$} (-1.25*\len,0.5*\len);

\node[rotate=-90] at (-0.75*\len,0.25*\len) {\scriptsize $\Arr$};
\node[rotate=-90] at (0.75*\len,0.25*\len) {\scriptsize $\Arr$};
\node[rotate=-90] at (0,0.75*\len) {\scriptsize $\Arr$};
\node[rotate=-90] at (0.75*\len,0.75*\len) {\scriptsize $\Arr$};
\node[rotate=180] at (0,0) {\scriptsize $\Arr$};
\node at (-1*\len,0.5*\len) {\scriptsize $\Arr$};
\node[rotate=180] at (-0.4*\len,0.5*\len) {\scriptsize $\Arr$};
\node[rotate=180] at (0.4*\len,0.5*\len) {\scriptsize $\Arr$};
\node[rotate=180] at (0.4*\len,1*\len) {\scriptsize $\Arr$};

\filldraw [black] (-0.75*\len,0) circle (1pt);
\filldraw [black] (0.75*\len,0) circle (1pt);
\filldraw [black] (0,0.5*\len) circle (1pt);
\filldraw [black] (0,1*\len) circle (1pt);
\filldraw [black] (-0.75*\len,0.5*\len) circle (1pt);
\filldraw [black] (-1.25*\len,0.5*\len) circle (1pt);
\filldraw [black] (0.75*\len,0.5*\len) circle (1pt);
\filldraw [black] (0.75*\len,1*\len) circle (1pt);
\filldraw [black] (-0.15,1.75*\len) circle (1pt);

% Text Node
\node at (-0.75*\len,0.25*\len) [anchor=west] {\footnotesize$x$};
\node at (0.75*\len,0.25*\len) [anchor=east] {\footnotesize$x$};
\node at (0,0.75*\len) [anchor=west] {\footnotesize$y$};
\node at (0.75*\len,0.75*\len) [anchor=east] {\footnotesize$y$};

\node at (0,0) [anchor=north] {$c$};
\node at (0.4*\len,0.5*\len) [anchor=north] {\footnotesize$s$};
\node at (0.4*\len,1.03*\len) [anchor=north] {\footnotesize$s'$};
\node at (-0.4*\len,0.5*\len) [anchor=north] {\footnotesize$t$};
\node at (-1*\len,0.5*\len) [anchor=south] {\footnotesize$y$};

\draw (1.15*\len,1.15*\len) node {\footnotesize$\sigma^{X}_{z}$};
\draw (-0.9*\len,1.15*\len) node {\footnotesize$\sigma^{Y}_{\overline{y}z}$};
\draw (0.35*\len,0.68*\len) node {${f}$};
\draw (0,0.25*\len) node {${\pi}$};
\draw (0.375*\len,1.15*\len) node {$m$};
\draw (-0.375*\len,0.65*\len) node {$n$};
\node at (-0.95*\len,0.3*\len) {\footnotesize{$\overline{\nu}_i$}};
\node at (0.9*\len,0.5*\len) {\footnotesize{$\nu^i$}};
\end{tikzpicture}
\end{equation}
is the right hand side of~\eqref{eq:Rel1}.
Applying the naturality of $\sigma^X$ and $\sigma^Y$, we can transport the morphisms $\overline{\nu}_i$ and $\nu^i$ to be glued together and by~\eqref{dominance} we obtain that the expressions~\eqref{eq:inter_step_m_str1} and~\eqref{eq:inter_step_m_str2} coincide.
As usual the second relation~\eqref{eq:Rel2} holds by the symmetry of the situation. One can verify, with similar arguments, that the following assignment gives the inverse of $\Psi_{M,N}$
\begin{equation*}
\raisebox{-0.25em}{\pielemUnit{\text{\footnotesize $X$}}{\text{\footnotesize $Y$}}{c}{f}}\mapsto\sum_{t,s\in\Irr\cx}\rd_t\rd_s\;\overline{\alpha}_i\otimes\overline{\beta}_j\otimes 
\begin{tikzpicture}[xscale=0.8,baseline={([yshift=-.5ex]current bounding box.center)}]
\draw[line width=1]
    (0,0.5*\len) .. controls (0.15*\len,0.9*\len) and (0.6*\len,0.9*\len).. node[midway,rotate=180] {\scriptsize $\Arr$}(0.75*\len,0.5*\len)
    node[midway,above] { $s$};
\draw[line width=1]
    (0,0.5*\len) .. controls (-0.15*\len,0.9*\len) and (-0.6*\len,0.9*\len).. node[midway,rotate=180] {\scriptsize $\Arr$}(-0.75*\len,0.5*\len)
    node[midway,above] { $t$};    

\draw[line width=1] 
    (-0.75*\len,0) to (0.75*\len,0)
    (-0.75*\len,0.5*\len) to (0.75*\len,0.5*\len);
    \draw[line width=1,dotted] 
    (-0.75*\len,0) to (-0.75*\len,0.5*\len)
    (0.75*\len,0) to (0.75*\len,0.5*\len);

\node[rotate=180] at (0,0) {\scriptsize $\Arr$};
\node[rotate=180] at (-0.4*\len,0.5*\len) {\scriptsize $\Arr$};
\node[rotate=180] at (0.4*\len,0.5*\len) {\scriptsize $\Arr$};

\filldraw [black] (-0.75*\len,0) circle (1pt);
\filldraw [black] (0.75*\len,0) circle (1pt);
\filldraw [black] (0,0.5*\len) circle (1pt);
% \filldraw [black] (0,1.25*\len) circle (1pt);
\filldraw [black] (-0.75*\len,0.5*\len) circle (1pt);
\filldraw [black] (0.75*\len,0.5*\len) circle (1pt);

% Text Node
\node at (0,0) [anchor=north] {$c$};
\node at (0.4*\len,0.5*\len) [anchor=north] {\scriptsize$X$};
\node at (-0.4*\len,0.5*\len) [anchor=north] {\scriptsize$Y$};

% \draw (0,1.075*\len) node 
\draw (0,0.25*\len) node {${f}$};
\draw (0.375*\len,0.65*\len) node {$\alpha^i$};
\draw (-0.375*\len,0.65*\len) node {$\beta^j$};
\end{tikzpicture}
\end{equation*}
by making use of Lemma~\ref{lem:base_change} and relation~\eqref{eq:rel_unit}. That $\Psi_{X,Y}$ is an intertwiner follows from the naturality of the half-braidings involved and relation~\eqref{dominance}.
\end{proof}

\begin{prop}
The natural isomorphisms~\eqref{eq:monoidal_str_T} endow the equivalence~\eqref{eq:equivalence} with the structure of a braided monoidal equivalence.
\end{prop}

\begin{proof}
Given central objects $X,Y,Z\in \cz(\Ind\cx)$, the associativiy condition 
\begin{equation}\label{eq:associativity_psi}
    \begin{tikzcd}[column sep=2.5cm]
        (\EQ(X)\Boxt \EQ(Y))\Boxt\EQ(Z)\ar[d,"\Psi_{X,Y}\Boxt \id",swap]\ar[r,"\mathrm{A}_{\EQ(X),\EQ(Y),\EQ(Z)}"]
        & \EQ(X)\Boxt (\EQ(Y)\Boxt\EQ(Z))\ar[d,"\id\Boxt \Psi_{Y,Z}"]\\
        \EQ(X\otimes Y)\Boxt\EQ(Z) \ar[d," \Psi_{X\otimes Y, Z}",swap]
        & \EQ(X)\Boxt \EQ(Y\otimes Z)\ar[d," \Psi_{X,Y\otimes Z}"]\\
        \EQ(X\otimes Y\otimes Z)\ar[r,"\cong"]& \EQ(X\otimes Y\otimes Z)
    \end{tikzcd}
\end{equation}
must be fulfilled. In view of Proposition~\ref{prop:rel_unit_simple} (iii), it is enough to check the commutativity of the previous diagram for vectors of the form 
\def\len{1.4cm}
\begin{equation}\label{eq:element_interstep}
\left(\centrelem{s}{X}{m}\otimes \centrelem{t}{Y}{n}\otimes{\def\len{1.6cm}\pielemUnit{s}{t}{c}{\pi}}\right)\otimes
\centrelem{r}{Z}{l}\otimes{\def\len{1.6cm}\pielemUnit{c}{r}{d}{\rho}}\;\in (\EQ(X)\Boxt \EQ(Y))\Boxt \EQ(Z)\,.
\end{equation}
A direct computation, involving Lemma~\ref{lem:associators_unit}, shows that both compositions in the diagram~\eqref{eq:associativity_psi} assign to~\eqref{eq:element_interstep} the value
\def\len{1.2cm}
% {\small
\begin{equation}\begin{tikzpicture}[baseline={([yshift=-.5ex]current bounding box.center)}]
\draw[line width=1]
    (1.2*\len,1.5*\len) .. controls (1.05*\len,2.05*\len) and (0.55*\len,2.05*\len).. node[midway,rotate=180] {\scriptsize $\Arr$}(0.4*\len,1.5*\len)
    node[midway,above] {\scriptsize $X$};
\draw[line width=1]
    (0.4*\len,1.5*\len) .. controls (0.25*\len,2.05*\len) and (-0.25*\len,2.05*\len).. node[midway,rotate=180] {\scriptsize $\Arr$}(-0.4*\len,1.5*\len)
    node[midway,above] {\scriptsize $Y$};     
\draw[line width=1]
    (-0.4*\len,0.75*\len) .. controls (-0.55*\len,1.3*\len) and (-1.05*\len,1.3*\len).. node[midway,rotate=180] {\scriptsize $\Arr$}(-1.2*\len,0.75*\len)
    node[midway,above] {\scriptsize $Z$};
    \draw[line width=1] 
    (-1.2*\len,0) to (1.2*\len,0)
    (-0.4*\len,1.5*\len) to (1.2*\len,1.5*\len)
    (-1.2*\len,0.75*\len) to (1.2*\len,0.75*\len);
    \draw[line width=1,dotted] 
    (-1.2*\len,0) to (-1.2*\len,0.75*\len)
        (-0.4*\len,0.75*\len) to (-0.4*\len,1.5*\len)
    (1.2*\len,0) to (1.2*\len,1.5*\len);
\node[rotate=180] at (0,0) {\scriptsize $\Arr$};
\node[rotate=180] at (-0.8*\len,0.75*\len) {\scriptsize $\Arr$};
\node[rotate=180] at (0.4*\len,0.75*\len) {\scriptsize $\Arr$};
\node[rotate=180] at (0.8*\len,1.5*\len) {\scriptsize $\Arr$};
\node[rotate=180] at (0,1.5*\len) {\scriptsize $\Arr$};
\filldraw [black] (-1.2*\len,0) circle (1pt);
\filldraw [black] (1.2*\len,0) circle (1pt);
\filldraw [black] (-0.4*\len,0.75*\len) circle (1pt);
\filldraw [black] (-1.2*\len,0.75*\len) circle (1pt);
\filldraw [black] (1.2*\len,0.75*\len) circle (1pt);
\filldraw [black] (1.2*\len,1.5*\len) circle (1pt);
\filldraw [black] (0.4*\len,1.5*\len) circle (1pt);
\filldraw [black] (-0.4*\len,1.5*\len) circle (1pt);
%
% Text Node
%
\node at (0,0) [anchor=north] {$d$};
\node at (0.4*\len,0.75*\len) [anchor=north] {$c$};
\node at (0.8*\len,1.5*\len) [anchor=north] {$s$};
\node at (0,1.5*\len) [anchor=north] {$t$};
\node at (-0.8*\len,0.75*\len) [anchor=north] {$r$};
\draw (0*\len,0.35*\len) node {${\rho}$};
\draw (-0.8*\len,0.95*\len) node  {{$l$}};
\draw (0.8*\len,1.7*\len) node  {{$m$}};
\draw (0*\len,1.7*\len) node  {{$n$}};
\node at (0.4*\len,1.125*\len) {{$\pi$}};
\end{tikzpicture}
\end{equation}
% }
in $\EQ(X\otimes Y\otimes Z)$ and thus we conclude that the desired commutativity holds. To prove that $\EQ$ is braided, we have to check that the diagram
\begin{equation*}
\begin{tikzcd}[column sep=2cm, row sep=normal]
    \EQ(X)\Boxt \EQ(Y) \ar[d,"\Psi_{X,Y}",swap]\ar[r,"\mathrm{B}_{\EQ(X), \EQ(Y)}"]
    &\EQ(Y)\Boxt \EQ(X) \ar[d,"\Psi_{Y,X}"]\\
    \EQ(X\otimes Y)\ar[r,"\EQ(\sigma_{X,Y})",swap]&\EQ(Y\otimes X)
\end{tikzcd}
\end{equation*}
commutes. For simplicity, we rely once more on Proposition~\ref{prop:rel_unit_simple} (iii), and verify the commutativity for 
\def\len{1.4cm}
\begin{equation*}
\centrelem{s}{X}{m}\otimes \centrelem{t}{Y}{n}\otimes{\def\len{1.6cm}\pielemUnit{s}{t}{c}{\pi}}\;\in \EQ(X)\Boxt \EQ(Y)\,.
\end{equation*}
Its image under $\EQ(\sigma_{X,Y})\circ \Psi_{X,Y}$ is given by the value
\def\len{1.4cm}
\begin{equation}\label{eq:braided_interstep}
\begin{tikzpicture}[xscale=0.8,baseline={([yshift=-.5ex]current bounding box.center)}]
\draw[line width=1]
    (0,1.75*\len) .. controls (0.5*\len,1.75*\len) and (0.85*\len,0.9*\len).. node[midway,rotate=115] {\scriptsize $\Arr$}(0.75*\len,0.5*\len)
    node[midway,right] {\footnotesize $X$};
\draw[line width=1]
    (0,1.75*\len) .. controls (-0.5*\len,1.75*\len) and (-0.85*\len,0.9*\len).. node[midway,rotate=255] {\scriptsize $\Arr$}(-0.75*\len,0.5*\len)
    node[midway,left] {\footnotesize $Y$};

\draw[line width=1]
    (0,0.5*\len) .. controls (0.15*\len,0.9*\len) and (0.6*\len,0.9*\len).. node[midway,rotate=180] {\scriptsize $\Arr$}(0.75*\len,0.5*\len)
    node[midway,above] {\scriptsize $Y$};
\draw[line width=1]
    (0,0.5*\len) .. controls (-0.15*\len,0.9*\len) and (-0.6*\len,0.9*\len).. node[midway,rotate=180] {\scriptsize $\Arr$}(-0.75*\len,0.5*\len)
    node[midway,above] {\scriptsize $X$};    

\draw[line width=1] 
    (-0.75*\len,0) to (0.75*\len,0)
    (-0.75*\len,0.5*\len) to (0.75*\len,0.5*\len);
    \draw[line width=1,dotted] 
    (-0.75*\len,0) to (-0.75*\len,0.5*\len)
    (0.75*\len,0) to (0.75*\len,0.5*\len);

\node[rotate=180] at (0,0) {\scriptsize $\Arr$};
\node[rotate=180] at (-0.4*\len,0.5*\len) {\scriptsize $\Arr$};
\node[rotate=180] at (0.4*\len,0.5*\len) {\scriptsize $\Arr$};

\filldraw [black] (-0.75*\len,0) circle (1pt);
\filldraw [black] (0.75*\len,0) circle (1pt);
\filldraw [black] (0,0.5*\len) circle (1pt);
\filldraw [black] (0,1.75*\len) circle (1pt);
\filldraw [black] (-0.75*\len,0.5*\len) circle (1pt);
\filldraw [black] (0.75*\len,0.5*\len) circle (1pt);

% Text Node
\node at (0,0) [anchor=north] {$d$};
\node at (0.4*\len,0.5*\len) [anchor=north] {\scriptsize$s$};
\node at (-0.4*\len,0.5*\len) [anchor=north] {\scriptsize$t$};

\draw (0,1.25*\len) node {$\sigma^Y_X$};
\draw (0,0.25*\len) node {$\pi$};
\draw (0.375*\len,0.65*\len) node {$m$};
\draw (-0.375*\len,0.65*\len) node {$n$};
\end{tikzpicture}
\end{equation}
in $\EQ(Y\otimes X)$, and
\def\len{1.6cm}
\begin{equation*}
\sum_{z\in\Irr\cx}\rd_z
\begin{tikzpicture}[yscale=1,xscale=1,baseline={([yshift=-.5ex]current bounding box.center)}]
 \draw[line width=1]
    (0.5*\len,0) .. controls (1.1*\len,0*\len) and (1.1*\len,2*0.375*\len).. (0.5*\len,0.75*\len) node [midway,rotate=-90]  {\scriptsize $\Arr$}
    node [midway,right]  {\footnotesize $z$};
\draw[line width=1]
    (0*\len,1.75*\len) .. controls (-2.75*\len,1.75*\len) and (-1.75*\len,0*\len).. node[pos=0.4,rotate=-105] {\scriptsize $\Arr$}(-0.5*\len,0*\len)
    node[pos=0.4,left] {\footnotesize $X$};
\draw[line width=1]
    (0*\len,1.75*\len) .. controls (0.75*\len,1.75*\len) and (2.25*\len,0.15*\len).. node[pos=0.4,rotate=135] {\scriptsize $\Arr$}(0.5*\len,0*\len)
    node[pos=0.4,right] {\footnotesize $Y$};    
\draw[line width=1] 
    (-0.5*\len,0) to (0.5*\len,0)
    (-0.5*\len,0.75*\len) to (-0.5*\len,1.25*\len)
    (-1*\len,0.75*\len) to (-1*\len,1.25*\len)
    (-0.5*\len,0.75*\len) to (0*\len,1.75*\len)
    (0.5*\len,0) to (0.5*\len,0.75*\len)
    (-0.5*\len,0.75*\len) to (0.5*\len,0.75*\len)
    (-0.5*\len,0.75*\len) to (-1*\len,0.75*\len)
    (-0.5*\len,0) to [in=-90,out=180] (-1*\len,0.75*\len);
\draw[line width=1]
    (-0.5*\len,1.25*\len) .. controls (-0.75*\len,1.3*\len) .. node[midway,rotate=180] {\scriptsize $\Arr$} (-1*\len,1.25*\len)
    node[midway,above] {\scriptsize $X$};
\draw[line width=1]
    (-0.5*\len,0.75*\len) .. controls (-0.25*\len,1.2*\len) and (0.25*\len,1.2*\len).. node[midway,rotate=180] {\scriptsize $\Arr$}(0.5*\len,0.75*\len)
    node[midway,above] {\scriptsize $Y$};
\draw[line width=1,dotted] 
(-0.5*\len,0.75*\len) to (-1*\len,1.25*\len)
    (-0.5*\len,0) to (-0.5*\len,0.75*\len);

\node[rotate=90] at (0.5*\len,0.375*\len) {\scriptsize $\Arr$};
\node[rotate=-90] at (-0.5*\len,1*\len) {\scriptsize $\Arr$};
\node[rotate=65] at (-0.15*\len,1.45*\len) {\scriptsize $\Arr$};
\node[rotate=-90] at (-1*\len,1*\len) {\scriptsize $\Arr$};
\node[rotate=-70] at (-0.94*\len,0.375*\len) {\scriptsize $\Arr$};
\node[rotate=180] at (0,0) {\scriptsize $\Arr$};
\node[rotate=180] at (-0.75*\len,0.75*\len) {\scriptsize $\Arr$};
\node[rotate=180] at (0,0.75*\len) {\scriptsize $\Arr$};

\filldraw [black] (-0.5*\len,0) circle (1pt);
\filldraw [black] (0.5*\len,0) circle (1pt);
\filldraw [black] (-0.5*\len,1.25*\len) circle (1pt);
\filldraw [black] (-1*\len,1.25*\len) circle (1pt);
\filldraw [black] (-0.5*\len,0.75*\len) circle (1pt);
\filldraw [black] (0.5*\len,0.75*\len) circle (1pt);
\filldraw [black] (-1*\len,0.75*\len) circle (1pt);

\filldraw [black] (0*\len,1.75*\len) circle (1pt);
% Text Node
\node at (0,0) [anchor=north] {$d$};
\node at (0,0.75*\len) [anchor=north] {\footnotesize$t$};
\node at (-0.75*\len,0.75*\len) [anchor=north] {\footnotesize$s$};
\node at (0.5*\len,0.375*\len) [anchor=east] {\footnotesize$s$};
\node at (-0.55*\len,1.15*\len) [anchor=west] {\footnotesize$s$};
\node at (-0.15*\len,1.45*\len) [anchor=east] {\footnotesize$z$};

\node at (-1*\len,1.15*\len) [anchor=east] {\footnotesize$s$};
\node at (-0.95*\len,0.35*\len) [anchor=east] {\footnotesize$z$};

\draw (0,0.9*\len) node {$n$};
\draw (0,0.325*\len) node {${\pi}$};
\node at (-0.85*\len,0.9*\len) {\footnotesize{$\Id_s$}};
\node at (-0.7*\len,1.14*\len) {\footnotesize{$m$}};
\node at (-0.7*\len,0.375*\len) {{$\overline{\varphi}_{\!k}$}};
\node at (0.7*\len,0.375*\len) {{$\varphi^k$}};
\node at (0.6*\len,1.15*\len) {{$\sigma^Y_z$}};
\node at (-1.4*\len,0.85*\len) {{$\sigma^X_{sz}$}};
\end{tikzpicture}
\end{equation*}
represents its image under the composition $\Psi_{Y,X}\circ\mathrm{B}_{\EQ(X), \EQ(Y)}$.
Now, by applying naturality of $\sigma^X$, we can transport $m$ to bound $\overline{\varphi}_k$. Next, by means of~\eqref{eq:base_change}, we further transport $m$ to the right side of $n$ obtaining bases $\overline{\psi}$ and $\psi$ of $\Hsp{X\,\overline{z}}$ in the process. By naturality of $\sigma^X$ and $\sigma^Y$ the bases $\overline{\psi}$ and $\psi$ are glued together and using~\eqref{dominance}, we finally recover the expression~\eqref{eq:braided_interstep}, thereby proving that $\EQ$ is braided.
\end{proof}

Let us summarize what we have obtained so far.

\begin{thm}\label{thm:main}
Let $\cx$ be a spherical semisimple multitensor category. There is a braided monoidal equivalence
\begin{equation*}
    \EQ\Colon\mathcal{Z}(\Ind\cx)\xrightarrow{\quad\simeq\quad}\Rep \Tube
\end{equation*}
given by the assignment~\eqref{eq:equivalence} and monoidal structure~\eqref{eq:monoidal_str_T}. Moreover, restricting $\EQ$ to dualizable objects, we obtain an equivalence
\begin{equation*}
    \mathcal{Z}(\cx)\xrightarrow{\quad\simeq\quad}\Rep^{\rm l.f.} \Tube
\end{equation*}
of ribbon semisimple tensor categories.
\end{thm}

%\bibliographystyle{alpha}
%\bibliography{./references.bib}

\end{document}